\documentclass[reqno,11pt]{amsart}

\usepackage{graphics}
\usepackage{amssymb, amsmath}
\usepackage{version}
\usepackage{color}
\usepackage{mathrsfs}
\usepackage{mathtools}
\usepackage{color}

\theoremstyle{theorem}
\newtheorem{theorem}{\sc Theorem}[section]
\newtheorem{lemma}[theorem]{\sc Lemma}

\newtheorem{proposition}[theorem]{\sc Proposition}
\newtheorem{corollary}[theorem]{\sc Corollary}

\theoremstyle{definition}
\newtheorem{definition}[theorem]{\sc Definition}

\theoremstyle{remark}
\newtheorem{remark}[theorem]{\sc Remark}

\makeatletter

\@addtoreset{equation}{section}

\renewcommand{\d}{\text{\rm d}}
\newcommand{\vep}{\varepsilon}

\newcommand{\R}{\mathbb{R}}
\newcommand{\N}{\mathbb{N}}

\newcommand{\X}{\mathcal{X}_\theta(\Omega)}

\allowdisplaybreaks[4]

\newcommand{\prf}[1]{}  

\newcommand{\UUU}{\color{blue}}

\usepackage[normalem]{ulem}

\begin{document}
\title[Fractional nonlinear diffusion equations in domains]{Energy solutions of the Cauchy--Dirichlet problem for fractional nonlinear diffusion equations\prf{\\{\tt [Extended Edition]}}}
\author{Goro Akagi}
\address{Mathematical Institute and Graduate School of Science, Tohoku University, 6-3 Aoba, Aramaki, Aoba-ku, Sendai 980-8578, Japan}
\email{goro.akagi@tohoku.ac.jp}
\author{Florian Salin}
\address{Mathematical Institute and Graduate School of Science, Tohoku University, 6-3 Aoba, Aramaki, Aoba-ku, Sendai 980-8578, Japan, and Institut Camille Jordan, \'Ecole Centrale de Lyon, 36 avenue Guy de Collongue 69134 Ecully Cedex, France}
\email{salin.florian.denis.jacques.r2@dc.tohoku.ac.jp}

\thanks{G.A.~is supported by JSPS KAKENHI Grant Numbers JP21KK0044, JP21K18581, JP20H01812 and JP20H00117. This work was also supported by the Research Institute for Mathematical Sciences, an International Joint Usage/Research Center located in Kyoto University. F.S.~is supported by JST SPRING, Grant Number JPMJSP2114, and by a Scholarship of Tohoku University, Division for Interdisciplinary Advanced Research and Education.}
\date{\today}
\maketitle
\begin{abstract}
The present paper is concerned with the Cauchy--Dirichlet problem for \emph{fractional} (and non-fractional) \emph{nonlinear diffusion equations} posed in bounded domains. Main results consist of well-posedness in an \emph{energy class} with no sign restriction and convergence of such (possibly sign-changing) energy solutions to asymptotic profiles after a proper rescaling. They will be proved in a variational scheme only, without any use of semigroup theories nor classical quasilinear parabolic theories. Proofs are self-contained and performed in a totally unified fashion for both fractional and non-fractional cases as well as for both porous medium and fast diffusion cases.
\end{abstract}

\section{Introduction}

Let $\Omega$ be a bounded domain of $\mathbb{R}^d$ with boundary $\partial \Omega$. The present paper is concerned with (possibly sign-changing) energy solutions to the Cauchy--Dirichlet problem for the (fractional) nonlinear diffusion equation of the form,
\begin{alignat}{3}
\partial_t (|u|^{q-2}u) + (-\Delta)^\theta u &= 0 \quad &&\text{ in } \ \Omega \times (0,\infty),\label{pde}\\ 
u &= 0 \quad &&\text{ on } \ (\R^d \setminus \Omega) \times (0,\infty),\label{bc}\\
(|u|^{q-2}u)(\,\cdot\,,0) &= \rho_0 \quad &&\text{ in } \ \Omega,\label{ic}
\end{alignat}
where $\partial_t := \partial/\partial t$, $1 < q < \infty$, $0 < \theta \leq 1$ and $\rho_0$ is an initial datum. In case $0 < \theta < 1$, $(-\Delta)^\theta$ denotes the so-called \emph{fractional Laplacian} given by
$$
(-\Delta)^{\theta} f(x) = \mathcal{F}^{-1} \left[ |\xi|^{2\theta} \mathcal{F} f \right](x)
= C(d,\theta) \, \mathrm{P.V.} \int_{\R^d} \frac{f(x)-f(y)}{|x-y|^{d+2\theta}} \, \d y
$$
for, e.g., $f \in C^\infty_c(\R^d)$ (see \S \ref{Ss:fLap} below for details). In case $\theta = 1$, $(-\Delta)^\theta$ stands for the standard (minus) Laplacian (i.e., $-\Delta$), and the Dirichlet condition \eqref{bc} is replaced by the standard one,
$$
u = 0 \quad \text{ on } \ \partial \Omega \times (0,\infty).\eqno{\eqref{bc}'}
$$
Therefore the case $\theta = 1$ fully corresponds to the classical \emph{porous medium} ($1 < q < 2$) and \emph{fast diffusion equations} ($q > 2$), for which an enormous number of contributions have already been made (see, e.g.,~\cite{Vazquez,BoFi24}). Moreover, due to the change of variables,
$$
\rho(x,t) = (|u|^{q-2}u)(x,t) \ \mbox{ and } \ m = \frac 1{q-1} \in (0,\infty),
$$
the Cauchy--Dirichlet problem \eqref{pde}--\eqref{ic} can (equivalently) be rewritten as
\begin{alignat*}{2}
\partial_t \rho + (-\Delta)^\theta (|\rho|^{m-1}\rho) &= 0 \quad &&\mbox{ in } \ \Omega \times (0,\infty),\\
(|\rho|^{m-1}\rho) &= 0 \quad &&\mbox{ on } \ (\R^d \setminus \Omega) \times (0,\infty),\\
\rho(\cdot,0) &= \rho_0 \quad &&\mbox{ in } \ \Omega
\end{alignat*}
in the sense of Definition \ref{D:sol} below. In case $\theta = 1$, the homogeneous Dirichlet boundary condition $\eqref{bc}'$ is nothing but the following:
$$
(|\rho|^{m-1}\rho) = 0 \quad \text{ on } \ \partial \Omega \times (0,\infty).
$$

Most of studies on fractional nonlinear diffusion equations have been performed for the Cauchy problem, where the equations are posed in $\R^d$ (see, e.g.,~\cite{PQRV11,PQRV12,Va14,BoVa14,PQR16,PQR18,BoEn23} and references therein). On the other hand, there have also been a limited number of contributions to the Cauchy--Dirichlet problem in bounded domains. Among those, the \emph{fractional porous medium equation} ($0 < \theta < 1$, $1 < q < 2$) is studied in~\cite{BSV,BV15-2,BV15-1,NgVa}, where well-posedness is proved in the class of (possibly sign-changing) ``$H^*$-solutions'', which is an analogue of the $H^{-1}$ solution for $\theta = 1$ to the fractional setting, by a direct application of the Br\'ezis--K\=omura theory (see~\cite{HB1,HB2}) and asymptotic behaviors of nonnegative $H^*$ solutions are also studied (see also~\cite{BoFiRO,BoFiVa18}). Thanks to a B\'enilan--Crandall type estimate, the convergence of nonnegative $H^*$ solutions (after a rescaling) to a single asymptotic profile is proved. Moreover, convergence of (possibly) sign-changing weak solutions (after the rescaling) to sign-definite asymptotic profiles in the uniform norm are also proved under certain variational criteria for initial energy in~\cite{FraVol23}, whose strategy is based on the paper~\cite{BraVol} concerned with the (non-fractional) porous medium equation (see also~\cite[Remark 3.3]{AK15} for some instability results). The study on the \emph{fractional fast diffusion equation} ($0 < \theta < 1$, $q > 2$) is still ongoing. We refer the reader to a pioneer work~\cite{KiLe11}, where H\"older regularity and asymptotic behavior of nonnegative solutions are studied, as well as to a recent comprehensive study~\cite{BII}, where existence and uniqueness for \emph{nonnegative} ``weak dual solutions'' are proved by means of the Br\'ezis--K\=omura theory (to construct $H^*$ solutions first as an approximation) along with a Steklov averaging and a monotonicity argument (available only for nonnegative data), and where $L^p$-$L^q$ smoothing estimates and  (finite-time) extinction estimates (in various norms) are further discussed for nonnegative weak dual solutions. Moreover,  some of these works (see, e.g.,~\cite{BSV,BII}) also cover more general nonlocal elliptic operators (e.g., spectral fractional Laplacian). On the other hand, concerning $0 < \theta < 1$, the study of (possibly) \emph{sign-changing} weak solutions has not yet been fully pursued, and furthermore, \emph{full convergence} to asymptotic profiles has never been proved even for nonnegative weak solutions to the fractional fast diffusion equation (cf.~see \cite{FeiSim00} for the classical one $\theta = 1$). 

In the present paper, we are concerned with existence, uniqueness and continuous dependence on initial data of (possibly \emph{sign-changing}) weak solutions in an \emph{energy class} to the Cauchy--Dirichlet problem for \eqref{pde}--\eqref{ic} with $1 < q < \infty$ and $0 < \theta \leq 1$, i.e., involving both nonlocal and local, both porous medium and fast diffusion equations.
Concerning $\theta = 1$ (i.e., the local case), the well-posedness of the Cauchy--Dirichlet problem has been proved in several different schemes, the most standard one of which may rely on \emph{classical theories for quasilinear parabolic equations} (see, e.g.,~\cite{Friedman,LSU,Lieberman}). Suitable modifications of the equation \eqref{pde} and initial data $\rho_0$ provide approximate problems being neither degenerate nor singular, and being instead uniformly parabolic (see~\cite{Vazquez}). In this scheme, approximate solutions enjoy classical regularity, e.g., of class $C^{2+\alpha,1+\alpha/2}(\Omega \times (0,T))$, and hence, one can manipulate the approximate equations as well as their solutions rigorously without any additional arguments. On the other hand, such approximations may not preserve intrinsic structures of the original equation, and therefore, one may make an extra effort to fill the gap between structures of the original and approximate problems. Another scheme is based on a couple of \emph{nonlinear semigroup theories}. One of them is the so-called $H^{-1}$ framework, where the Cauchy--Dirichlet problem \eqref{pde}--\eqref{ic} is reduced, without any approximation, into the Cauchy problem for a nonlinear evolution equation governed by a subdifferential operator in $H^{-1}(\Omega)$, the dual space of $H^1_0(\Omega)$. Then well-posedness can be proved by means of the Br\'ezis--K\=omura theory (see~\cite{HB1,HB2}). Another one is often called an $L^1$ framework, which is based on the celebrated Crandall--Liggett theory (see~\cite{CL}). In this scheme, well-posedness of the Cauchy--Dirichlet problem is proved in a certain class of solutions (i.e., $H^{-1}$ and $L^1$ solutions) by a direct application of nonlinear semigroup theories. On the other hand, this scheme is less flexible to extract additional information from solutions beyond the frame of each theory. Actually, although the $H^{-1}$ framework is consistent with a variational (or energy) structure of \eqref{pde}, the regularity of $H^{-1}$ solutions may be too weak to perform some of energy estimates rigorously. The $L^1$ framework can provide stronger solutions, but it is less relevant to the variational structure. The third scheme relies on a \emph{metric gradient flow approach}, where the Cauchy--Dirichlet problem is reduced to the Cauchy problem for a gradient flow in the Monge--Kantrovich metric (see~\cite{Otto01}). It enables us to extract an entropy structure from the equation as well as to prove well-posedness in a fairly weak sense by a direct application of a general theory (see~\cite{AGS}). In each scheme, one can prove well-posedness, for which the class of solutions may depend on the scheme. On the other hand, these solutions eventually coincide each other, since the uniqueness can be proved in a rather weak formulation. 

In this paper, we are also interested in revealing asymptotic behaviors of solutions to the Cauchy--Dirichlet problem posed in bounded domains, for which energy structures of the equation play a crucial role. Therefore we shall work in a proper class of weak solutions (called \emph{energy solutions}) which enjoy energy identities (or inequalities) and corresponding regularities required for the analysis of their asymptotic behaviors. For the local case $\theta = 1$, this purpose can be achieved with the first scheme mentioned above based on the quasilinear parabolic theory. However, for the nonlocal case $0 < \theta < 1$, this scheme is no longer applicable due to the nonlocality of the equation; indeed, the classical quasilinear parabolic theory is established for local equations only, and hence, we already face a difficulty in the approximation level. On the other hand, it is enough to exploit a nonlinear semigroup theory just for proving well-posedness in a weak formulation (like $H^*$ solutions) and it is also possible to cover an even more general class of (nonlocal and nonlinear) diffusion equations simultaneously; however, there may still remain a delicate issue on whether or not energy inequalities and sufficient regularities are assured for such weak solutions (see Remark \ref{R:gen} in \S \ref{S:main}). \prf{Indeed, even for $\theta = 1$, it requires a lot of preliminary steps and tedious calculations to prove that. }In the present paper, we shall hence select another scheme, which is based on the time-discretization and energy methods, to prove well-posedness of the Cauchy--Dirichlet problem \eqref{pde}--\eqref{ic} posed in general bounded domains as well as to derive energy identities (and inequalities) for the local and nonlocal cases. Our methods of proofs are fully variational and do not rely on the classical quasilinear parabolic theories nor nonlinear semigroup theories. It is self-contained, and moreover, it can handle the local ($\theta = 1$) and nonlocal cases ($0 < \theta < 1$) as well as the porous medium ($1 < q < 2$) and fast diffusion cases ($q > 2$) in a totally unified fashion. Furthermore, in order to prove existence of energy solutions, we can preserve intrinsic variational structures of the equation, and therefore, energy identities and inequalities can be derived in detail without extra complexity. To the best of authors' knowledge, there is no literature providing a complete description of such a variational proof for existence of (possibly sign-changing) energy solutions even for the local case $\theta = 1$ (cf.~see~\cite{A:EnSol} for an outline).

This paper is composed of six sections. The next section is devoted to preliminary facts on the fractional Sobolev space and the fractional Laplacian. Main results of the present paper will be stated in Section \ref{S:main}. The main technical novelties of the present paper may reside in proofs for (i) well-posedness along with energy identities and inequalities (in Section \ref{S:pr_main}) and also for (ii) the monotonicity of the Rayleigh quotient (see \eqref{R} below) through the evolution of energy solutions (in Section \ref{S:Rayleigh}). We shall pay a careful attention to the defect of absolute continuity of the Rayleigh quotient and bridge such a gap by means of the theory for functions of bounded variation (see, e.g.,~\cite{AFP}). Then the results such as (i) and (ii) enable us to derive immediately (optimal) decay and extinction estimates for the energy solutions (see also Section \ref{S:Rayleigh}). Finally, Section \ref{S:conv} is dedicated to discussing convergence of (properly rescaled) energy solutions to asymptotic profiles. The novelty of this part may be found in an application of a {\L}ojasiewicz--Simon inequality of Feireisl--Simon type for the restricted fractional Laplacian (see~\cite{ASS19}) to prove the full convergence to a single asymptotic profile for the fractional fast diffusion equation.

\bigskip
\noindent
{\bf Notation.} Let $u = u(x,t): \Omega \times [0,\infty) \to \mathbb R$ be a function with space and time variables. Throughout the paper, for each $t \geq 0$ fixed, we simply denote by $u(t)$ the function $u(\cdot,t) : \Omega \to \mathbb R$ with only the space variable. For each $\alpha > 0$ and $r \in \R$, we simply write $|r|^{\alpha-1}r = |r|^\alpha \, \mathrm{sgn} \,r$, where $\mathrm{sgn} \,r = r/|r|$ if $r \neq 0$ and $\mathrm{sgn} \,r = 0$ if $r = 0$. Moreover, $q' := q/(q-1)$ denotes the H\"older conjugate of $q \in (1,\infty)$. For each interval $I$ in $\R$, $C_{\rm weak}(I;X)$ and $C_+(I;X)$ stand for the sets of weakly continuous and strongly right-continuous functions on $I$ with values in a normed space $X$, respectively. We denote by $C$ a generic nonnegative constant which may vary from line to line. 

\section{Preliminaries}

In this section, we shall briefly recall some preliminary facts on the \emph{fractional Sobolev space} as well as the \emph{fractional Laplacian} for later use.

\subsection{Fractional Sobolev spaces}\label{Ss:fSob}

For each $\theta\in (0,1)$, the \emph{fractional Sobolev space} $H^\theta(\mathbb{R}^d)$ is defined by
$$
H^\theta(\mathbb{R}^d) := \left\{ w \in L^2(\mathbb{R}^d) \colon [w]_{H^\theta(\mathbb{R}^d)} < \infty \right\},
$$
where $[\,\cdot\,]_{H^\theta(\R^d)}$ is the \emph{Gagliardo seminorm} defined by
$$
[w]_{H^\theta(\mathbb{R}^d)}^2 := \frac{1}{2} \iint_{\mathbb{R}^d \times \mathbb{R}^d} \frac{|w(x)-w(y)|^2}{|x-y|^{d+2\theta}} \, \d x \d y
\quad \text{ for } \ w \in L^2(\R^d).
$$
Then $H^\theta(\mathbb{R}^d)$ is a Hilbert space endowed with the inner product,
\begin{align*}
(v,w)_{H^\theta(\R^d)}
&:= (v,w)_{L^2(\R^d)}\\
&\quad + \frac{1}{2} \iint_{\R^d \times \R^d} \frac{(v(x)-v(y))(w(x)-w(y))}{|x-y|^{d+2\theta}} \, \d x \d y
\end{align*}
for $v,w \in H^\theta(\R^d)$. Hence we have
$$
\|w\|_{H^\theta(\R^d)}^2 = \|w\|_{L^2(\R^d)}^2 + [w]_{H^\theta(\mathbb{R}^d)}^2 \quad \text{ for } \ w \in H^\theta(\R^d).
$$

We set 
$$
\mathcal{X}_1(\Omega) := H^1_0(\Omega)
$$
and define a subspace of $H^\theta(\R^d)$ as
$$
\X:= \left\{ w \in H^\theta(\mathbb{R}^d) \colon w = 0 \ \text{ a.e.~on } \mathbb{R}^d \setminus \Omega \right\} \quad \text{ for } \ \theta \in (0,1),
$$
for which we also refer the reader to~\cite{BrLiPa,EdEv} and references therein. In what follows, we shall use the same letter $f$ to denote the restriction of $f \in \X$ to $\Omega$ when no confusion can arise. Moreover, we may also use the same notation for the zero extension to $\R^d$ of $f : \Omega \to \R$.

A Poincar\'e type inequality is also available for the class $\X$, $0 < \theta < 1$ (see, e.g.,~\cite{ASS16},~{\cite[Proposition A.1]{ASS19}}).

\begin{proposition}[Poincar\'e type inequality]\label{P:Poincare}
Let $\Omega$ be a bounded domain of $\R^d$ and let $\theta \in (0,1)$. Then there exists a constant $C > 0$ depending only on $d$, $\theta$ and the diameter of $\Omega$ such that
$$
\|w\|_{L^2(\mathbb{R}^d)} \leq C[w]_{H^\theta(\R^d)} \quad \text{ for all } \ w \in \X.
$$
\end{proposition}

Thanks to Proposition \ref{P:Poincare}, for each $\theta \in (0,1)$, the Gagliardo seminorm $[\,\cdot\,]_{H^\theta(\mathbb{R}^d)}$ turns out to be a norm of $\X$ which is equivalent to $\|\cdot\|_{H^\theta(\R^d)}$. Therefore we endow the space $\X$ with $\|\cdot\|_{\X}:=[\,\cdot\,]_{H^\theta(\R^d)}$. Moreover, we set $\|\cdot\|_{\mathcal{X}_1(\Omega)} := \| |\nabla \cdot| \|_{L^2(\Omega)}$. Therefore, for $\theta \in (0,1]$, the space $(\X,\|\cdot\|_{\X})$ is a Hilbert space equipped with the inner product,
$$
(v,w)_{\X} := \begin{cases}
\int_\Omega \nabla v(x) \cdot \nabla w(x) \, \d x &\text{ if } \ \theta = 1,\\
\frac{1}{2} \iint_{\mathbb{R}^d \times \mathbb{R}^d} \frac{(v(x)-v(y))(w(x)-w(y))}{|x-y|^{d+2\theta}} \, \d x \d y &\text{ if } \ \theta \in (0,1)
\end{cases}
$$
for $v,w \in \X$. We also refer the reader to, e.g.,~\cite{LM},~\cite{Lunardi},~\cite{Hitchhiker} for more details on fractional Sobolev spaces. 

We close this subsection with stating fractional Sobolev embeddings for the space $\X$ and density results. To this end, we define
\begin{equation}\label{crt_exp}
 2_\theta^*:= \begin{cases}
	       \frac{2d}{d-2\theta} &\text{ if } \ d > 2\theta,\\
	       \infty &\text{ otherwise}.
	      \end{cases}
\end{equation}
We first recall fractional Sobolev embeddings (see, e.g.,~\cite[Theorems 6.7, 6.9 and 7.2]{Hitchhiker}). 

\begin{proposition}[Sobolev--Poincar\'e type inequality]\label{P:SP}
Let $\Omega$ be a bounded domain of $\R^d$ and let $\theta \in (0,1]$. Suppose that
$$
q \in \begin{cases}
       [1,2_\theta^*] &\text{ if } \ d > 2\theta,\\
       [1,\infty) &\text{ if } \ d = 2\theta,\\
       [1,\infty] &\text{ if } \ d < 2\theta.
      \end{cases}
$$
Then $\X$ is continuously embedded in $L^q(\Omega)$, i.e.,~the following Sobolev--Poincar\'e type inequality holds,
\begin{equation}\label{SP}
\|w\|_{L^q(\Omega)} \leq C_q \|w\|_{\X} \quad \text{ for all } \ w \in \X,
\end{equation}
where $C_q>0$ denotes the best constant. Moreover, $\X$ is compactly embedded in $L^q(\Omega)$, provided that $1 \leq q < 2_\theta^*$.
\end{proposition}

In particular, we note that, for any $0 < \theta \leq 1$,
\begin{equation}\label{pivot}
\X \hookrightarrow L^2(\Omega) \equiv L^2(\Omega)^* \hookrightarrow \X^*
\end{equation}
with compact and continuous canonical injections. Indeed, the compact embedding $L^2(\Omega)^* \hookrightarrow \X^*$ follows from the compact embedding $\X \hookrightarrow L^2(\Omega)$ along with Schauder's theorem (see, e.g.,~\cite[Theorem 6.4]{B-FA}). 

In order to give a representation to the space $\X$, we present  a density result. See Appendix \S \ref{S:apdx} for a proof.
\begin{proposition}[Density]\label{P:density}
Let $\Omega$ be a smooth bounded domain of $\R^d$ and let $0<\theta< 1$. Then the space $C^\infty_c(\Omega)$ is dense in $\X$.
\end{proposition}

Hence, if $\theta \in (0,1] \setminus \{1/2\}$ and $\Omega$ is smooth, then $\X$ can be identified with 
$$
 H^\theta_0(\Omega) := \overline{ C^\infty_c(\Omega) }^{H^\theta(\Omega)}.
$$
Here for each $\theta \in (0,1)$ we denote 
$$
H^\theta(\Omega) := \left\{ w \in L^2(\Omega) \colon [w]_{H^\theta(\Omega)} < \infty \right\}
$$
equipped with a seminorm $[\,\cdot\,]_{H^\theta(\Omega)}$ and a norm $\|\cdot\|_{H^\theta(\Omega)}$ defined by
$$
[w]_{H^\theta(\Omega)}^2 := \frac 12 \iint_{\Omega \times \Omega} \frac{|w(x)-w(y)|^2}{|x-y|^{d+2\theta}} \, \d x \d y \quad \text{ for } \ w \in H^\theta(\Omega)
$$
and
$$
\|w\|_{H^\theta(\Omega)}^2 := \|w\|_{L^2(\Omega)}^2 + [w]_{H^\theta(\Omega)}^2 \quad \text{ for } \ w \in H^\theta(\Omega),
$$
respectively (see, e.g.,~\cite{LM}). Indeed, Theorem \ref{T:ext0} (proved in~\cite{LM}) shows that $H^\theta_0(\Omega) \hookrightarrow \X$ when $\theta \neq 1/2$, and Proposition \ref{P:density} implies that $\X \hookrightarrow H^\theta_0(\Omega)$ when $\Omega$ is smooth. Furthermore, as for $\theta = 1/2$, $\X$ is equivalent to the space $H^{1/2}_{00}(\Omega)$, which is introduced in~\cite{LM} and different from both $H^{1/2}(\Omega)$ and $H^{1/2}_0(\Omega)$. When $\theta \leq 1/2$ and $\Omega$ is smooth, it is shown in~\cite{LM} that $H^\theta_0(\Omega) = H^\theta(\Omega)$ (see also Appendix \ref{S:apdx}).

\subsection{Restricted fractional Laplacian}\label{Ss:fLap}

For each $\theta \in (0,1)$, the \emph{fractional Laplacian} $(-\Delta)^\theta$ is defined for, e.g., $\varphi \in C_c^\infty(\R^d)$ by
\begin{align*}
(-\Delta)^\theta \varphi(x) 
&= \mathcal{F}^{-1} \left[ |\xi|^{2\theta} \mathcal{F} f \right](x) \nonumber\\
&= C(d,\theta) \, \mathrm{P.V.} \int_{\R^d} \frac{\varphi(x)-\varphi(y)}{|x-y|^{d+2\theta}} \, \d y \nonumber\\
&= C(d,\theta) \lim_{\varepsilon \to 0} \int_{\R^d \setminus B(x;\vep)} \frac{\varphi(x)-\varphi(y)}{|x-y|^{d+2\theta}} \, \d y \quad \mbox{ for } \ x \in \R^d,
\end{align*}
where $\mathcal{F}$ and $\mathcal{F}^{-1}$ denote the Fourier transform operator and its inverse, respectively, and $\mathrm{P.V.}$ stands for the \emph{Cauchy principal value}, and moreover, $B(x;\vep)$ denotes the open ball of radius $\vep$ centered at $x$ and $C(d,\theta)$ is the constant given by
$$
C(d,\theta) = \left( \int_{\R^d} \frac{1-\cos \xi_1}{|\xi|^{d+2\theta}} \, \d \xi \right)^{-1}.
$$
In what follows, we may neglect the constant $C(d,\theta)$ for simplicity. Note that, even though $\varphi$ has a compact support, the support of $(-\Delta)^\theta \varphi$ may no longer be compact.

For every $\theta \in (0,1)$, the \emph{restricted fractional Laplacian}, still denoted by $(-\Delta)^\theta$, is defined for each $\varphi \in C^\infty_c(\Omega)$ in terms of the restriction of $(-\Delta)^\theta \varphi : \R^d \to \R$ to the domain $\Omega$.\footnote{Here and henceforth, we use the same letter for the zero extension to $\R^d$ of each function defined in $\Omega$, when no confusion can arise} For each $\varphi,\psi \in C_c^\infty(\Omega)$, one can show that
\begin{align*}
\int_\Omega \psi(x) (-\Delta)^\theta \varphi(x) \, \d x 
&= \frac{1}{2} \iint_{\R^d \times \R^d} \frac{(\psi(x)-\psi(y))(\varphi(x)-\varphi(y))}{|x-y|^{d+2\theta}} \, \d x \d y\\
&= (\psi,\varphi)_{\X},
\end{align*}
which is symmetric and continuous in $\X$ with respect to $\varphi$ and $\psi$. Hence a weak form of the restricted fractional Laplacian with the Dirichlet condition is given as follows:

\begin{definition}[A weak form of the restricted fractional Laplacian]\label{D:RFL}
Let $\Omega$ be a domain of $\R^d$. For each $\theta \in (0,1)$, the \emph{restricted fractional Laplacian with the Dirichlet condition} is defined in a weak sense as an operator, still denoted by $(-\Delta)^\theta$, from $\X$ to $\X^*$ by the relation,
$$
\langle (-\Delta)^\theta u,v \rangle_{\X}
= \frac{1}{2} \iint_{\R^d \times \R^d} \frac{ (u(x)-u(y))(v(x)-v(y)) }{ |x-y|^{d+2\theta} } \, \d x \d y
$$
for $u, v \in \X$.
\end{definition}

As for $\theta = 1$, we denote by $(-\Delta)^1 := - \Delta$ the (minus) classical Dirichlet Laplacian, whose weak form is given as
$$
\langle - \Delta u, v \rangle_{\mathcal X_1(\Omega)} = \int_\Omega \nabla u(x) \cdot \nabla v(x) \, \d x
$$
for $u,v \in \mathcal X_1(\Omega) = H^1_0(\Omega)$ (and then, $\mathcal X_1(\Omega)^* = H^{-1}(\Omega)$).

From the above, for every $\theta \in (0,1]$, $(-\Delta)^\theta$ turns out to be the Riesz map from $\X$ to its dual space $\X^*$, and moreover, we observe that
\begin{equation}\label{Lap-vari}
(-\Delta)^\theta w = \d \left( \frac{1}{2} \|w\|_{\X}^2 \right) \ \text{ for } \ w \in \X,
\end{equation}
where $\d$ means the Fr\'echet derivative in $\X$.

\section{Main results}\label{S:main}

This section is devoted to stating main results of the present paper. To this end, we begin with defining \emph{{weak} solutions} for the Cauchy--Dirichlet problem \eqref{pde}--\eqref{ic}.

\begin{definition}[{Weak} solutions to \eqref{pde}--\eqref{ic}]\label{D:sol}
Let $1 < q < \infty$, $\theta \in (0,1]$ and $\rho_0 \in \X^*$. For $0 < T < \infty$, a pair of measurable functions $(u,\rho) : \R^d \times (0,T) \to \R^2$ is called a \emph{{weak} solution} on $[0,T]$ to the Cauchy--Dirichlet problem \eqref{pde}--\eqref{ic} if the following (i)--(iii) hold true:
\begin{enumerate}
 \item For any $\vep \in (0,T)$, it holds that $u \in L^2(\vep,T;\X) \cap L^q(Q_T)$ and $\rho \in C([0,T];\X^*) \cap W^{1,2}(\vep,T;\X^*) \cap L^{q'}(Q_T)$, where $Q_T := \Omega \times (0,T)$.
 \item It holds that $\rho = |u|^{q-2}u$ a.e.~in $Q_T$, and moreover,
\begin{equation}\label{weak_form}
\left\langle \partial_t \rho(t), w \right\rangle_{\X} + ( u(t), w )_{\X} = 0 \quad \mbox{ for all } \ w \in \X
\end{equation}
for a.e.~$t \in (0,T)$. Here $\partial_t \rho$ stands for the derivative of the function $t \mapsto \rho(t)$.
 \item The initial condition $\rho(0) = \rho_0$ holds, that is, $\rho(t) \to \rho_0$ strongly in $\X^*$ as $t \to 0_+$.
\end{enumerate}
Furthermore, for $0 < T \leq \infty$, a function $(u,\rho) : \R^d \times (0,T) \to \R^2$ is called a {weak} solution on $[0,T)$ to the Cauchy--Dirichlet problem \eqref{pde}--\eqref{ic} if it is a weak solution on $[0,S]$ for any $0 < S < T$. In particular, when $T = \infty$, it may simply be called a {weak} solution to the Cauchy--Dirichlet problem \eqref{pde}--\eqref{ic}. 

In what follows, the {weak} solution may be denoted by $u$ only (instead of the pair $(u,\rho)$), when no confusion can arise; indeed, $\rho$ can uniquely be determined by $u$ from the relation $\rho = |u|^{q-2}u$ of (ii). Moreover, $\rho$ will then be denoted by $|u|^{q-2}u$ simply.
\end{definition}

Weak solutions complying with some energy inequalities and regularity are called \emph{energy solutions} (see (ii) of Theorem \ref{T:main} below for more details).

\begin{remark}[Other classes of solutions]\label{R:others}
As for $\theta = 1$, the class of {weak} solutions in the sense of Definition \ref{D:sol} clearly contains $H^{-1}$ solutions, that is, $\rho \in C([0,T];H^{-1}(\Omega))$ complying with the following conditions:
\begin{itemize}
 \item $\rho \in L^{m+1}(Q_T) \cap W^{1,2}_{\rm loc}((0,T];H^{-1}(\Omega))$,
 \item $u = |\rho|^{m-1}\rho \in L^2_{\rm loc}((0,T];H^1_0(\Omega))$, 
 \item the absolute continuity of the function $t \mapsto \|\rho(t)\|_{L^m(\Omega)}^m$ on $(0,T]$, 
 \item the weak form,
$$
\langle \partial_t \rho(t), w \rangle_{H^1_0(\Omega)} + \int_\Omega \nabla u(t) \cdot \nabla w \, \d x \quad \mbox{ for } \ w \in H^1_0(\Omega)
$$
for a.e.~$t \in (0,T)$,
 \item and the initial condition $\rho (0) = \rho_0$
\end{itemize}
(see~\cite{HB2},~\cite{Vazquez}). Conversely, as will be seen in \S \ref{Ss:enid}, the {weak} solutions turn out to be $H^{-1}$ solutions. Thus the class of {weak} solutions coincides with that of $H^{-1}$ solutions. Such an observation can also be extended to the fractional case $0 < \theta < 1$; indeed, the $H^*$ solution introduced in~\cite{BSV} is an analogue of the $H^{-1}$ solution to the fractional case $0 < \theta < 1$ and it is equivalent to the weak solution in the sense of Definition \ref{D:sol}. On the other hand, the notion of \emph{weak dual solutions} (see~\cite{BII}) seems like a ``very weak'' version of a \emph{weighted} $L^1$ solution. Hence there may be no direct relation between the weak dual solution and the weak solution concerned in this paper; however, as proved in~\cite{BII}, due to the weight, every \emph{nonnegative} weak solution turns out to be a weak dual solution. 
\end{remark}

\begin{remark}[As a gradient flow]\label{R:GF}
The weak form \eqref{weak_form} can also be written as an \emph{evolution equation},
\begin{equation}\label{ee}
\partial_t \rho(t) + (-\Delta)^\theta u(t) = 0 \ \text{ in } \X^* \quad \text{ for a.e. } t \in (0,T),
\end{equation}
which along with the relation $\rho = |u|^{q-2}u$ has been well studied so far. Moreover, Definition \ref{D:sol} is consistent with a (generalized) gradient flow structure of the nonlinear diffusion equation \eqref{pde} (along with \eqref{bc}). Indeed, thanks to \eqref{Lap-vari}, one can rewrite \eqref{ee} as
\begin{equation*}
\partial_t ( |u|^{q-2}u )(t) = - \d \varphi (u(t))  \ \text{ in } \X^* \quad \text{ for a.e. } t\in(0,T),
\end{equation*}
where $\d \varphi : \X \to \X^*$ denotes the Fr\'echet derivative of the functional $\varphi : \X \to \R$ given by
$$
\varphi(w) = \frac{1}{2} \|w\|_{\X}^2 \ \text{ for } \ w \in \X.
$$
Such a generalized gradient flow structure enables us to derive energy identities and inequalities, which will play a crucial role in analyzing long-time behaviors of weak solutions to \eqref{pde}--\eqref{ic}.
\end{remark}

\begin{remark}[Energy identity]\label{R:ws-reg}
The following facts follow immediately from Definition \ref{D:sol} for general weak solutions $(u,\rho)$ on $[0,T]$ to the Cauchy--Dirichlet problem \eqref{pde}--\eqref{ic}: We first observe that the function 
$$
t \mapsto \| \rho(t)\|_{\X^*}^2 = \left\langle \rho(t), (-\Delta)^{-\theta} \rho(t) \right\rangle_{\X}
$$
is absolutely continuous on $(0,T]$. Testing \eqref{ee} by $(-\Delta)^{-\theta} \rho(t)$ directly, we then find that
$$
\frac 1 2 \frac{\d}{\d t} \left\| \rho(t) \right\|_{\X^*}^2 + \|u(t)\|_{L^q(\Omega)}^q = 0 \quad \mbox{ for a.e. } t \in (0,T),
$$
which implies
\begin{align}
\frac 12 \left\| \rho(t) \right\|_{\X^*}^2 + \int^t_s \|u(r)\|_{L^q(\Omega)}^q \, \d r
& = \frac 12 \left\| \rho(s) \right\|_{\X^*}^2\label{enineq0}
\end{align}
for $0 \leq s \leq t \leq T$ (including $s = 0$ by virtue of the fact that $\rho \in C([0,T];\X^*)$). We shall further derive another energy identity for general weak solutions in \S \ref{Ss:enid}.
\end{remark}

The main results of the present paper are stated as follows:

\begin{theorem}[Well-posedness and regularity of energy solutions]\label{T:main}
Let $1 < q < \infty$ and $0 < \theta \leq 1$. Then the following {\rm (i)--(iii)} hold true\/{\rm :}
\begin{enumerate}
\item[\rm (i)] For each $\rho_0 \in \X^*$ and $T > 0$, the {weak} solution on $[0,T]$ of the Cauchy--Dirichlet problem \eqref{pde}--\eqref{ic} is unique. Moreover, let $\rho_{0,1}, \rho_{0,2} \in \X^*$ and let $u_1, u_2$ be {weak} solutions on $[0,T]$ to \eqref{pde}--\eqref{ic} for the initial data $\rho_{0,1},\rho_{0,2}$, respectively. Then it holds that
\begin{align}
\MoveEqLeft{
\left\| (|u_1|^{q-2}u_1)(t) - (|u_2|^{q-2}u_2)(t) \right\|_{\X^*} 
}\nonumber\\
&\leq \| \rho_{0,1} - \rho_{0,2} \|_{\X^*}\label{conti-dep}
\end{align}
for any $t \in [0,T]$.
 \item[\rm (ii)] For every $u_0 \in \X \cap L^q(\Omega)$ satisfying that
\begin{equation}\label{ini-hyp+}
\rho_0 := |u_0|^{q-2}u_0 \in L^{(2_\theta^*)'}(\Omega) 
\ \text{ if } \ q > 2_\theta^*,
\end{equation}
where $(2_\theta^*)'$ is the H\"older conjugate of $2_\theta^*$, the Cauchy--Dirichlet problem \eqref{pde}--\eqref{ic} admits a unique {weak} solution $u = u(x,t)$ such that
\begin{align*}
 u &\in C_{\rm weak}([0,\infty);\X) \cap C_+([0,\infty);\X) \\
 &\quad \cap C([0,\infty);L^q(\Omega)) \cap L^q(\Omega \times (0,\infty)),
\end{align*}
which also implies
$$
u(t) \to u_0 \quad \text{ strongly in } \X \ \text{ as } \ t \to 0_+\/{\rm ;}
$$
hence the initial condition \eqref{ic} can also be written as 
$$
u(\cdot,0) = u_0 \ \text{ in } \Omega.
$$
Moreover, it holds that
\begin{align*}
|u|^{q-2}u &\in C([0,\infty);L^{q'}(\Omega)) \cap L^{q'}(\Omega \times (0,\infty)),\\ 
\partial_t (|u|^{q-2}u ) &\in C_{\rm weak}([0,\infty);\X^*) \cap C_+([0,\infty);\X^*),\\ 
|u|^{(q-2)/2}u &\in W^{1,2}(0,\infty;L^2(\Omega)).
\end{align*}
Furthermore, the function $t \mapsto \|u(t)\|_{L^q(\Omega)}^q$ is absolutely continuous on $[0,\infty)$ and the following energy identity holds\/{\rm :}
\begin{equation}\label{energyineq1}
\frac 1 {q'} \frac{\d}{\d t} \|u(t)\|_{L^q(\Omega)}^q + \|u(t)\|_{\X}^2 = 0
\end{equation}
for a.e.~$t > 0$, and moreover, the following energy inequality holds\/{\rm :}
\begin{align}
\MoveEqLeft{
\frac{4}{qq'} \int_s^t \left\| \partial_t ( |u|^{(q-2)/2}u )(r) \right\|_{L^2(\Omega)}^2 \, \d r + \frac{1}{2}\|u(t)\|_{\X}^2
}\nonumber\\
&\leq \frac{1}{2} \|u(s)\|_{\X}^2\label{energyineq2}
\end{align}
for all $0 \leq s < t <\infty$. Therefore the function $t \mapsto \|u(t)\|_{\X}^2$ is nonincreasing on $[0,\infty)$, and hence, it is differentiable a.e.~in $(0,\infty)$ and the following differential form also holds\/{\rm :}
\begin{equation}\label{en-ineq}
 \frac 4{qq'} \left\| \partial_t (|u|^{(q-2)/2}u)(t) \right\|_{L^2(\Omega)}^2 + \frac 1 2 \frac{\d}{\d t} \|u(t)\|_{\X}^2 \leq 0
\end{equation}
for a.e.~$t > 0$. In what follows, such a weak solution is called an \emph{energy solution} to the Cauchy--Dirichlet problem \eqref{pde}--\eqref{ic} {\rm (}cf.~see also Corollary {\rm \ref{C:wex}} below{\rm )}. Furthermore, if $q > 2$, then $\partial_t (|u|^{q-2}u)$ belongs to $L^2(0,\infty;L^{q'}(\Omega))$.

In addition, if $\rho_0 \in L^\alpha(\Omega)$ for some $\alpha \in (1,\infty)$, then
\begin{align*}
|u|^{q-2}u &\in C_{\rm weak}([0,\infty);L^\alpha(\Omega)) \cap C_+([0,\infty);L^\alpha(\Omega)),\\ 
|u|^{(\beta-2)/2}u &\in L^2(0,\infty;\X),
\end{align*}
where $\beta := (q-1)(\alpha-1)+1$, and
\begin{align}
\MoveEqLeft{
\frac{1}{\alpha} \left\| (|u|^{q-2}u)(t) \right\|_{L^\alpha(\Omega)}^\alpha + \frac{4}{\beta\beta'} \int^t_s \left\| (|u|^{(\beta-2)/2}u)(r) \right\|_{\X}^2 \,\d r
}\nonumber\\
&\leq \frac{1}{\alpha} \left\| (|u|^{q-2}u)(s) \right\|_{L^\alpha(\Omega)}^\alpha \nonumber
\end{align}
for $0 \leq s \leq t < \infty$.
\item[\rm (iii)] Let $u_{0,1}, u_{0,2}$ fulfill the same assumptions for $u_0$ as in {\rm (ii)} and set $\rho_{0,j} = |u_{0,j}|^{q-2}u_{0,j} \in L^1(\Omega)$ for $j = 1,2$. Let $u_1,u_2$ be the energy solutions to the Cauchy--Dirichlet problem \eqref{pde}--\eqref{ic} for the initial data $\rho_{0,1}, \rho_{0,2}$, respectively. Then it holds that
\begin{align}\label{L1-contr}
\MoveEqLeft{
\left\| (|u_1|^{q-2}u_1)(t) - (|u_2|^{q-2}u_2)(t) \right\|_{L^1(\Omega)} 
}\nonumber\\
&\leq \| \rho_{0,1} - \rho_{0,2} \|_{L^1(\Omega)}
\end{align}
for $t \geq 0$. In addition, if $\rho_{0,1} \leq \rho_{0,2}$ a.e.~in $\Omega$, then $u_1 \leq u_2$ a.e.~in $\Omega \times (0,\infty)$.

Moreover, let $u = u(x,t)$ be a unique energy solution to \eqref{pde}--\eqref{ic} with an initial datum $\rho_0 = |u_0|^{q-2}u_0 \in L^1(\Omega)$ satisfying the same assumptions as in {\rm (ii)}. Then 
$$
|u|^{q-2}u \in W^{1,1}_{\rm loc}((0,\infty);L^1(\Omega))
$$
and the following B\'enilan--Crandall estimate holds\/{\rm :} there exists a constant $C > 0$ such that
\begin{equation}\label{BC}
\| (-\Delta)^\theta u(t) \|_{L^1(\Omega)} = \left\| \partial_t (|u|^{q-2}u)(t) \right\|_{L^1(\Omega)} \leq \frac C t \| \rho_0 \|_{L^1(\Omega)}
\end{equation}
for a.e.~$t > 0$. \prf{The above estimate also holds with $L^1(\Omega)$ replaced by $\X^*$. }In addition, if $\rho_0 \geq 0$ a.e.~in $\Omega$, then it holds that
\begin{equation}\label{BC-pnt}
\partial_t (|u|^{q-2}u)(x,t) \leq \frac{q-1}{q-2} \frac{(|u|^{q-2}u)(x,t)}{t} 
 \quad \mbox{ if } \ q > 2
\end{equation}
and
\begin{equation}\label{BC-pnt+}
\partial_t (|u|^{q-2}u)(x,t) \geq - \frac{q-1}{2-q} \frac{(|u|^{q-2}u)(x,t)}{t} 
 \quad \mbox{ if } \ 1 < q < 2
\end{equation}
for a.e.~$x \in \Omega$ and $t > 0$.
\end{enumerate}  
\end{theorem}

With the aid of energy inequalities obtained so far, we can further verify the following:

\begin{theorem}[Monotonicity of the Rayleigh quotient]\label{T:Rayleigh}
Let $1 < q < \infty$ and $0 < \theta \leq 1$. Let $u_0$ satisfy the same assumptions as in {\rm (ii)} of Theorem {\rm \ref{T:main}} and $u_0 \not\equiv 0$ and let $u = u(x,t)$ be the energy solution to the Cauchy--Dirichlet problem \eqref{pde}--\eqref{ic} with the initial datum $\rho_0 := |u_0|^{q-2}u_0$. Define the Rayleigh quotient by
\begin{equation}\label{R}
R(w):=\frac{\|w\|_{\X}^2}{\|w\|_{L^q(\Omega)}^2}\quad\text{ for } \ w \in (\X \cap L^q(\Omega)) \setminus \{0\}.
\end{equation}
Then the function $t \mapsto R(u(t))$ is nonincreasing as far as $u(t) \neq 0$.
\end{theorem}

As in~\cite{BH80,Kwong88-3,DKV91,SavareVespri,AK13}, Theorems \ref{T:main} and \ref{T:Rayleigh} imply

\begin{corollary}[Decay and extinction estimates for energy solutions (cf.{~\cite[\S 2.6]{BII},~\cite{KiLe11}})]\label{C:decay}
Let $0 < \theta \leq 1$ and suppose that
\begin{equation}\label{hyp-decay}
1 < q \leq 2_\theta^*, \quad u_0\in \X \ \text{ and } \ u_0 \not\equiv 0.
\end{equation}
Let $u = u(x,t)$ be the energy solution to \eqref{pde}--\eqref{ic} with the initial datum $\rho_0 = |u_0|^{q-2}u_0$. Then it holds that
\begin{enumerate}
\item in case $1 < q < 2$, there exists a constant $c_1 > 0$ such that
\begin{equation}\label{asymp1}
c_1(t+1)^{-1/(2-q)}\leq \|u(t)\|_{L^q(\Omega)} \leq c_1^{-1}(t+1)^{-1/(2-q)}
\end{equation}
for $t \geq 0$\/{\rm ;}
\item in case $2 < q \leq 2_\theta^*$, there exist constants $t_* > 0$ and $c_2 > 0$ such that
\begin{equation}\label{asymp2}
c_2(t_*-t)_+^{1/(q-2)}\leq \|u(t)\|_{L^q(\Omega)} \leq c_2^{-1}(t_*-t)_+^{1/(q-2)}
\end{equation}
for $t \geq 0$. Here $t_* = t_*(u_0)$ is called the \emph{extinction time} of the energy solution $u$ and uniquely determined by the initial datum $u_0$.
\end{enumerate}
Moreover, \eqref{asymp1} and \eqref{asymp2} also hold true with the norm $\|\cdot\|_{L^q(\Omega)}$ replaced by $\|\cdot\|_{\X}$. Furthermore, in case $2 < q \leq 2_\theta^*$, the extinction time $t_*$ satisfies
$$
\frac{q-1}{q-2} \frac{\|u_0\|_{L^q(\Omega)}^q}{\|u_0\|_{\X}^2} \leq t_*(u_0) \leq \frac{q-1}{q-2} C_q^2 \|u_0\|_{L^q(\Omega)}^{q-2},
$$
where $C_q$ is the best constant of the Sobolev--Poincar\'e type inequality \eqref{SP}.
\end{corollary}

As the optimal decay/extinction rate has been determined in Corollary \ref{C:decay}, we next discuss an \emph{asymptotic profile} for the energy solution $u = u(x,t)$ to \eqref{pde}--\eqref{ic} with an initial datum $\rho_0 = |u_0|^{q-2}u_0$, $u_0\in \X$. To this end, we define
\begin{align}
v(x,s) &:=
\begin{cases}
(t + 1)^{1/(2-q)} u(x,t) &\text{ if } \ 1 < q < 2,\\
(t_* - t)_+^{-1/(q-2)} u(x,t) &\text{ if } \ 2 < q \leq 2_\theta^*,
\end{cases}\label{def1}\\
s &:=
\begin{cases}
\log (t + 1) &\text{ if } \ 1 < q < 2,\\
\log (t_*/(t_*-t)) &\text{ if } \ 2 < q \leq 2_\theta^*.
\end{cases}
\label{def2}
\end{align}
In particular, when $2 < q \leq 2_\theta^*$, the finite interval $[0,t_*)$ for $t$ is mapped by \eqref{def2} onto the half-line $[0,\infty)$ for $s$. Now, an \emph{asymptotic profile} $\phi = \phi(x)$ of the energy solution $u = u(x,t)$ is defined as a limit of $v(\cdot,s)$ in $H^1_0(\Omega)$ as $s \to \infty$. Moreover, the rescaled function $v = v(x,s)$ turns out to be the energy solution to the following Cauchy--Dirichlet problem:
\begin{alignat}{4}
\partial_s \left( |v|^{q-2}v \right) + (-\Delta)^\theta v &= \lambda_q |v|^{q-2}v \quad &&\text{ in } \ \Omega \times (0,\infty),\label{eq:1.6}\\ 
v &= 0 \quad && \text{ on } \ (\R^d \setminus \Omega) \times (0,\infty),\label{eq:1.7}\\
(|v|^{q-2}v)(\cdot,0) &= |v_0|^{q-2}v_0 \quad && \text{ in } \ \Omega,\label{eq:1.8}
\end{alignat}
where $\lambda_q := (q-1)/|q-2| > 0$ and $v_0$ is the initial datum given by
\begin{equation}\label{v0}
 v_0 := \begin{cases}
	u_0 &\text{ if } \ 1 < q < 2,\\
	t_*(u_0)^{-1/(q-2)} u_0 &\text{ if } \ 2 < q \leq 2_\theta^*
       \end{cases}
\end{equation}
(see also \S \ref{Ss:rescale} below). In particular, if $\theta = 1$, the Dirichlet condition \eqref{eq:1.7} is replaced by the standard one, that is, $v = 0$ on $\partial \Omega \times (0,\infty)$. Here the \emph{energy solution} to \eqref{eq:1.6}--\eqref{eq:1.8} is also defined in an analogous way to Definition \ref{D:sol}. Moreover, it is worth mentioning that \eqref{eq:1.6}, \eqref{eq:1.7} can be reformulated as a generalized gradient system of the form,
$$
\partial_s \left( |v|^{q-2}v \right)(s) = - \d J(v(s)) \ \text{ in } \X^*, \quad 0 < s < \infty,
$$
where $\d J : \X \to \X^*$ denotes the Fr\'echet derivative of the functional $J : \X \to \R$ defined by
$$
J(w) := \frac 12 \|w\|_{\X}^2 - \frac{\lambda_q}{q} \int_\Omega |w|^q \, \d x
\quad \text{ for } \ w \in \X.
$$
Therefore the issue is reduced into proving the convergence of the rescaled solution $v(\cdot,s)$ for such a (generalized) gradient system to an equilibrium $\phi$ as $s \to \infty$.

As in~\cite{BH80,Kwong88-3,DKV91,SavareVespri,AK13}, we can verify the quasi-convergence of $v(\cdot,s)$ as $s \to \infty$ to some stationary solution to \eqref{eq:1.6}, \eqref{eq:1.7} under the assumption that
\begin{equation}\label{hyp-ap}
q\in(1,2_\theta^*)\setminus\{2\}, \quad u_0 \in \mathcal{X}_\theta(\Omega) \ \text{ and } \ u_0 \not\equiv 0,
\end{equation}
where the fractional Sobolev critical exponent $2_\theta^*$ (defined as in \eqref{crt_exp}) is excluded from the range for $q$.

\begin{theorem}[Quasi-convergence to asymptotic profiles (cf.{~\cite[Theorem 5.4]{KiLe11}})]\label{T:ap}
Let $0 < \theta \leq 1$ and assume that \eqref{hyp-ap} holds. Let $u=u(x,t)$ be an energy solution to \eqref{pde}--\eqref{ic} and let $v=v(x,s)$ be defined by \eqref{def1}, \eqref{def2}. Then, for any sequence $(s_n)$ in $(0,\infty)$ diverging to $\infty$, there exist a subsequence $(n_k)$ of $(n)$ and $\phi \in \X\setminus\{0\}$ such that 
$$
v(s_{n_k}) \to \phi \quad \text{ strongly in } \X
$$
and $\phi$ is a nontrivial solution to the Dirichlet problem,
\begin{alignat}{4}
(-\Delta)^\theta \phi &= \lambda_q |\phi|^{q-2}\phi \ &&\text{ in } \ \Omega,\label{eq:1.10}\\
\phi &= 0 \ &&\text{ on } \ \R^d \setminus \Omega,\label{eq:1.11}
\end{alignat}
where \eqref{eq:1.11} is replaced by the standard Dirichlet condition, i.e., $\phi = 0$ on $\partial \Omega$, when $\theta = 1$.
\end{theorem}

As for the full-convergence, we state the following theorem:

\begin{theorem}[Full-convergence to asymptotic profiles]\label{T:conv}
Let $\Omega$ be a bounded $C^{1,1}$ domain of $\R^d$. In addition to the same assumptions as in Theorem {\rm \ref{T:ap}}, suppose that $d > 2\theta$ when $0 < \theta < 1$. Let $v = v(x,s)$ be the energy solution to the Cauchy--Dirichlet problem \eqref{eq:1.6}--\eqref{eq:1.8} and let $\phi = \phi(x)$ be a nontrivial solution to the Dirichlet problem \eqref{eq:1.10}, \eqref{eq:1.11} such that $v(s_n)\rightarrow \phi$ strongly in $\X$ for some sequence $(s_n)$ in $(0,\infty)$ diverging to $\infty$. In addition, assume that either of the following holds true\/{\rm :}
\begin{enumerate}
 \item[\rm (i)] $\phi \geq 0$ a.e.~in $\Omega$,
 \item[\rm (ii)] $q$ is even {\rm(}then $|r|^{q-2}r$ is a monomial{\rm )}.
\end{enumerate}
Then it holds that
$$
v(s) \to \phi \quad \text{ strongly in } \X \ \text{ as } \ s \to \infty.
$$
\end{theorem}

The porous medium case (i.e., $1 < q < 2$ and $\phi \geq 0$) has already been discussed in~\cite[Theorem 2.3]{BSV} and~\cite[Theorem 1.1]{FraVol23} for sign-definite and (possibly) sign-changing rescaled weak solutions $v = v(x,s)$, respectively, and hence, the novelty of Theorem \ref{T:conv} may reside in the fast diffusion case (i.e., $2 < q < 2^*_\theta$ and either $\phi \geq 0$ or $q$ is even). 

All these results will be proved in subsequent sections.

\section{Proof of Theorem \ref{T:main}}\label{S:pr_main}

In this section, we give a proof for Theorem \ref{T:main}. In \S \ref{Ss:uniq}, we shall briefly prove the uniqueness part (i) in a standard way. The main ingredients of the proof will be exhibited in \S \ref{Ss:ex}, \S \ref{Ss:enid} and \S \ref{Ss:gen}, where the additional regularity of the energy solution as well as the energy identity and inequalities will directly be proved within the weak formulation as in Definition \ref{D:sol} without relying on any classical parabolic theories (see, e.g.,~\cite{LSU,Friedman,Lieberman}) even for $\theta = 1$; actually, they are no longer useful for the fractional case $0 < \theta < 1$. The $L^1$ contraction estimate and the B\'enilan--Crandall estimate may be more or less standard (see, e.g.,~\cite{Vazquez,PQRV12}), but they will finally be proved in \S \ref{Ss:BC} for the completeness.

\subsection{Uniqueness of energy solutions}\label{Ss:uniq}

In this subsection, we shall give a proof for (i) of Theorem \ref{T:main}, i.e., the uniqueness of weak solutions to \eqref{pde}--\eqref{ic}. 

\begin{proof}[Proof of {\rm (i)} of Theorem {\rm \ref{T:main}}]
Let $\rho_{0,1}, \rho_{0,2}\in \X^*$ be two initial data and let $(u_1,\rho_1), (u_2,\rho_2)$ be weak solutions on $[0,T]$ to the Cauchy--Dirichlet problem for the initial data $\rho_{0,1}, \rho_{0,2}$, respectively. Then it follows from Remark \ref{R:GF} that
\begin{equation}\label{uniq:ei}
\partial_t \left( \rho_1-\rho_2 \right)(t) + (-\Delta)^\theta \left( u_1(t)-u_2(t) \right) = 0 \ \text{ in } \X^*,
\end{equation}
where $\rho_j := |u_j|^{q-2}u_j$ ($j=1,2$), for a.e.~$t \in (0,T)$. Let $(-\Delta)^{-\theta} : \X^* \to \X$ be the inverse of the restricted fractional Laplacian $(-\Delta)^\theta : \X \to \X^*$. Testing both sides of \eqref{uniq:ei} by $(-\Delta)^{-\theta} \left( \rho_1(t)-\rho_2(t) \right)$, we see that
\begin{align*}
\MoveEqLeft{
\left\langle \partial_t ( \rho_1-\rho_2 )(t), (-\Delta)^{-\theta} \left( \rho_1(t)-\rho_2(t) \right) \right\rangle_{\X}
}\\
&+ \left\langle (-\Delta)^\theta (u_1(t)-u_2(t)), (-\Delta)^{-\theta} \left(\rho_1(t)-\rho_2(t)\right) \right\rangle_{\X}=0
\end{align*}
for a.e.~$t \in (0,T)$. Here \prf{we note that the functional $w \mapsto \|w\|_{\X^*}^2$ is of class $C^1$ in $\X^*$, and }we easily observe that
$$
(-\Delta)^{-\theta} w = \d \left( \frac 12 \|w\|_{\X^*}^2 \right) \ \text{ for } \ w \in \X^*.
$$
Since $\rho_1 - \rho_2 \in W^{1,2}_{\rm loc}((0,T];\X^*)$ and $(-\Delta)^{-\theta}(\rho_1 - \rho_2) \in W^{1,2}_{\rm loc}((0,T];\X)$, thanks to the standard chain-rule formula, the function $t \mapsto \|\rho_1(t)-\rho_2(t)\|_{\X^*}^2$ turns out to be absolutely continuous on $(0,T]$ and to fulfill the relation,
$$
\frac 12 \frac{\d}{\d t} \|\rho_1(t)-\rho_2(t)\|_{\X^*}^2
= \left\langle \partial_t ( \rho_1 - \rho_2 )(t), (-\Delta)^{-\theta} ( \rho_1(t) - \rho_2(t) ) \right\rangle_{\X}
$$
for a.e.~$t \in (0,T)$. On the other hand, since $(- \Delta)^\theta$ is symmetric and $\rho_j = |u_j|^{q-2}u_j$ for $j = 1,2$, it follows that
\begin{align*}
\MoveEqLeft{
\left\langle(-\Delta)^\theta(u_1(t)-u_2(t)),(-\Delta)^{-\theta}\left(\rho_1(t)-\rho_2(t)\right)\right\rangle_{\X}
}\\
&= \int_\Omega \left( u_1(t) - u_2(t) \right) \left( \rho_1(t) - \rho_2(t) \right) \, \d x \geq 0
\end{align*}
for a.e.~$t \in (0,T)$. Therefore combining all these facts, we obtain
$$
\frac{1}{2} \frac{\d}{\d t} \|\rho_1(t)-\rho_2(t)\|_{\X^*}^2 \leq 0 \quad \text{ for a.e. } \ t \in (0,T),
$$
which yields
$$
\|\rho_1(t)-\rho_2(t)\|_{\X^*}^2 \leq \|\rho_1(s)-\rho_2(s)\|_{\X^*}^2
$$
for any $0 < s < t \leq T$. Passing to the limit as $s \to 0_+$ and using the fact that $\rho_1,\rho_2 \in C([0,T];\X^*)$ along with $\rho_j(0_+) = \rho_{0,j}$ for $j=1,2$ (see Definition \ref{D:sol}), we obtain
$$
\|\rho_1(t)-\rho_2(t)\|_{\X^*}^2 \leq \|\rho_{0,1}-\rho_{0,2}\|_{\X^*}^2
$$
for all $t \in [0,T]$. This completes the proof for (i) of Theorem \ref{T:main}.
\end{proof}

\subsection{Existence of energy solutions}\label{Ss:ex}

In this subsection, we prove existence of energy solutions and energy inequalities under the assumption that
\begin{equation}\label{hyp}
u_0 \in \X \cap L^q(\Omega) \quad \text{ and } \quad \rho_0 := |u_0|^{q-2}u_0 \in L^2(\Omega),
\end{equation}
which will slightly be relaxed in \S \ref{Ss:gen} for proving (ii) of Theorem \ref{T:main}. More precisely, this subsection is devoted to proving the following lemma:

\begin{lemma}[Existence of energy solutions under \eqref{hyp}]\label{L:ex-ensol}
Under the assumption \eqref{hyp}, the Cauchy--Dirichlet problem \eqref{pde}--\eqref{ic} admits a unique weak solution $u = u(x,t)$ complying with the regularity properties\/{\rm :}
\begin{align*}
u &\in C([0,\infty);L^q(\Omega)) \cap L^q(\Omega \times (0,\infty)),\\
u &\in C_{\rm weak}([0,\infty);\X) \cap C_+([0,\infty);\X),\\
|u|^{(q-2)/2}u &\in W^{1,2}(0,\infty;L^2(\Omega)) \cap L^2(0,\infty;\X),\\
|u|^{q-2}u &\in W^{1,\infty}(0,\infty;\X^*) \cap C([0,\infty);L^{q'}(\Omega)) \cap L^{q'}(\Omega \times (0,\infty)),\\
|u|^{q-2}u &\in C_{\rm weak}([0,\infty);L^2(\Omega)) \cap C_+([0,\infty);L^2(\Omega)),\\
\partial_t (|u|^{q-2}u) &\in C_{\rm weak}([0,\infty);\X^*) \cap C_+([0,\infty);\X^*),
\end{align*}
the initial condition\/{\rm :}
$$
u(t) \to u_0 \quad \text{ strongly in } \X \ \text{ as } \ t \to 0_+,
$$
and the following energy inequalities\/{\rm :}
\begin{equation}
 \frac{1}{q'} \|u(t)\|_{L^q(\Omega)}^q + \int_s^t \|u(r)\|_{\X}^2 \,\d r 
\leq \frac{1}{q'} \|u(s)\|_{L^q(\Omega)}^q,\label{enineq1}
\end{equation}
\begin{align}
\MoveEqLeft{
\frac{4}{qq'} \int_s^t \left\| \partial_t (|u|^{(q-2)/2}u)(r) \right\|_{L^2(\Omega)}^2 \,\d r + \frac 12 \|u(t)\|_{\X}^2
}\nonumber\\
 &\leq \frac 12 \|u(s)\|_{\X}^2,\label{enineq2}
\end{align}
\begin{align}
\MoveEqLeft{
\frac{1}{2} \left\| (|u|^{q-2}u)(t) \right\|_{L^2(\Omega)}^2 + \frac{4}{qq'} \int_s^t \left\| (|u|^{(q-2)/2}u)(r) \right\|_{\X}^2 \,\d r
}\nonumber\\
 &\leq \frac{1}{2} \left\|(|u|^{q-2}u)(s) \right\|_{L^2(\Omega)}^2 \label{enineq3}
\end{align}
for any $0 \leq s \leq t < \infty$. Furthermore, if $q > 2$, then $\partial_t (|u|^{q-2}u)$ lies on $L^2(0,\infty;L^{q'}(\Omega))$. 

In addition, if $\rho_0 \in L^\alpha(\Omega)$ for some $\alpha \in (1,\infty)$, then
\begin{align*}
|u|^{q-2}u &\in C_{\rm weak}([0,\infty);L^\alpha(\Omega)) \cap C_+([0,\infty);L^\alpha(\Omega)),\\ 
|u|^{(\beta-2)/2}u &\in L^2(0,\infty;\X),
\end{align*}
where $\beta := (q-1)(\alpha-1)+1$, and
\begin{align}
\MoveEqLeft{
\frac{1}{\alpha} \left\| (|u|^{q-2}u)(t) \right\|_{L^\alpha(\Omega)}^\alpha + \frac{4}{\beta\beta'} \int^t_s \left\| (|u|^{(\beta-2)/2}u)(r) \right\|_{\X}^2 \,\d r
}\nonumber\\
&\leq \frac{1}{\alpha} \left\| (|u|^{q-2}u)(s) \right\|_{L^\alpha(\Omega)}^\alpha \label{enineq4}
\end{align}
for $0 \leq s \leq t < \infty$.
\end{lemma}

\begin{proof}
This proof is divided into four steps.

\smallskip
\noindent
{\bf Step 1 (Discretization via minimization).} We shall introduce a time-discretization for the Cauchy--Dirichlet problem \eqref{pde}--\eqref{ic} and prove existence of solutions to some discretized problems. Set 
$$
X := \X \cap L^q(\Omega)
$$
being endowed with norm
$$
\|w\|_X := \|w\|_{\X} + \|w\|_{L^q(\Omega)} \ \text{ for } \ w \in X.
$$
Let $T>0$ be arbitrarily fixed and let $N \in \N$ be the number of steps. Set $\tau:=T/N > 0$ being the time-step. We now consider the following discretized equations:
\begin{equation}\label{En}
\frac{|u_{n+1}|^{q-2} u_{n+1} - |u_n|^{q-2}u_n}{\tau} + (-\Delta)^\theta u_{n+1} = 0 \ \text{ in } X^*, \quad u_{n+1}\in X
\end{equation}
for $n=0,1,\dots,N-1$. For each fixed $n$, we shall prove existence of $u_{n+1} \in X$ satisfying \eqref{En} and $|u_{n+1}|^{q-2}u_{n+1} \in L^2(\Omega)$ for each (prescribed) $u_n \in X$ satisfying $|u_n|^{q-2}u_n \in L^2(\Omega)$ in a variational fashion. To this end, define a functional $I_n : X \to \R$ by
$$
I_n(w) = \frac{1}{q\tau}\|w\|_{L^q(\Omega)}^q + \frac{1}{2} \|w\|_{\X}^2 - \int_\Omega \frac{|u_n|^{q-2}u_n}{\tau}w \,\d x
$$
for $w \in X$. Then $I_n$ is strictly convex and of class $C^1$ in $X$, and moreover, we find that
$$
\lim_{\|w\|_X\to \infty} I_n(w) = \infty.
$$
Indeed, using H\"older's and Young's inequalities, we have
\begin{align*}
\left| \int_\Omega \frac{|u_n|^{q-2}u_n}{\tau} w \,\d x \right|
&\leq \frac{1}{\tau} \|u_n\|_{L^q(\Omega)}^{q-1} \|w\|_{L^q(\Omega)},\\
&\leq \frac 1 \tau \left( \frac{1}{q'\epsilon^{q'}} \|u_n\|_{L^q(\Omega)}^q + \frac{\epsilon^q}{q} \|w\|_{L^q(\Omega)}^q \right)
\end{align*}
for $\epsilon>0$. Choosing $\epsilon>0$ small enough, we can take a constant $C$ (which may depend on $u_n$ prescribed) such that
$$
\frac{1}{q \tau} \|w\|_{L^q(\Omega)}^q - \int_\Omega \frac{|u_n|^{q-2}u_n}{\tau} w \,\d x \geq \frac{1}{2 q \tau}\|w\|_{L^q(\Omega)}^q - C.
$$
Thus we derive a coercive estimate,
$$
I_n(w) \geq \frac{1}{2 q \tau} \|w\|_{L^q(\Omega)}^q + \frac{1}{2}\|w\|_{\X}^2 - C
$$
for $w \in X$. Combining the coercivity with the convexity and lower semicontinuity of $I_n$ on $X$, we infer that $I_n$ admits a unique minimizer $u_{n+1}$ over $X$ (see, e.g.,~\cite[Corollary 3.23]{B-FA}). Furthermore, the minimizer $u_{n+1} \in X$ solves the Euler-Lagrange equation,
$$
\d I_n(u_{n+1}) = 0 \ \text{ in } X^*,
$$
which is equivalent to the equation \eqref{En}. Moreover, noting that $u_n \in X = \X \cap L^q(\Omega)$, we can observe from \eqref{En} that $(-\Delta)^\theta u_{n+1}$ lies on both $\X^*$ and $L^{q'}(\Omega)$, and hence, it follows that
$$
\frac{|u_{n+1}|^{q-2}u_{n+1} - |u_n|^{q-2}u_n}{\tau} + (-\Delta)^\theta u_{n+1} = 0 \ \text{ in } \X^* \cap L^{q'}(\Omega).
$$
We shall also prove in the next step that $|u_{n+1}|^{q-2}u_{n+1} \in L^2(\Omega)$.

\smallskip
\noindent
{\bf Step 2 (Discrete energy inequalities).} We next derive energy inequalities for the solutions $\{u_n\}_{n=1}^N$ of \eqref{En}. They will provide uniform (with respect to the time-step $\tau$) estimates for approximate solutions to \eqref{pde}--\eqref{ic}. Let $j \in \{0,\ldots,N-1\}$ and test \eqref{En} with $n = j$ by $u_{j+1} \in X$. Then we see that
$$
\int_\Omega \frac{|u_{j+1}|^{q-2}u_{j+1} - |u_j|^{q-2}u_j}{\tau} u_{j+1} \, \d x + \|u_{j+1}\|_{\X}^2 = 0.
$$
By virtue of Young's inequality (or convexity), we have
\begin{align*}
\MoveEqLeft{
\int_\Omega \frac{|u_{j+1}|^{q-2}u_{j+1} - |u_j|^{q-2}u_j}{\tau} u_{j+1} \, \d x}\\
&\geq \frac{1}{\tau q'} \int_\Omega |u_{j+1}|^q \, \d x - \frac{1}{\tau q'} \int_\Omega |u_j|^q \, \d x,
\end{align*} 
which yields
$$
\frac{1}{q'} \|u_{j+1}\|_{L^q(\Omega)}^q - \frac{1}{q'} \|u_j\|_{L^q(\Omega)}^q + \tau\|u_{j+1}\|_{\X}^2 \leq 0.
$$
Summing it up for $j=0,\dots,n$, we infer that
\begin{equation}\label{est1disc}
\frac{1}{q'} \|u_{n+1}\|_{L^q(\Omega)}^q + \sum_{j=0}^n \tau \|u_{j+1}\|_{\X}^2\leq \frac{1}{q'} \|u_0\|_{L^q(\Omega)}^q
\end{equation}
for $n=0,\dots,N-1$.

We next test \eqref{En} with $n = j$ by $(u_{j+1}-u_j)/\tau \in \X$. It then follows that
$$
\int_\Omega \frac{|u_{j+1}|^{q-2}u_{j+1}-|u_j|^{q-2}u_j}{\tau} \frac{u_{j+1}-u_j}{\tau} \, \d x + \left(u_{j+1},\frac{u_{j+1}-u_j}{\tau} \right)_{\X} = 0.
$$ 
Employing Young's inequality (or convexity), we see that
$$
\left( u_{j+1}, \frac{u_{j+1}-u_j}{\tau} \right)_{\X}
\geq \frac{1}{2\tau} \|u_{j+1}\|_{\X}^2 - \frac{1}{2\tau} \|u_j\|_{\X}^2.
$$
On the other hand, from the following elementary inequality:
\begin{equation}\label{EleIne}
 (|a|^{q-2}a - |b|^{q-2}b)(a-b) \geq \frac 4{qq'} \left| |a|^{(q-2)/2}a - |b|^{(q-2)/2}b \right|^2
\end{equation}
for $a,b \in \R$ (see, e.g.,~\cite[Appendix]{A23}), we derive that
\begin{align*}
\MoveEqLeft{
\frac{|u_{j+1}|^{q-2}u_{j+1} - |u_j|^{q-2}u_j}{\tau} \frac{u_{j+1}-u_j}{\tau}
}\\
&\geq \frac{4}{qq'}\left| \frac{|u_{j+1}|^{(q-2)/2}u_{j+1}-|u_j|^{(q-2)/2}u_j}{\tau}\right|^2.
\end{align*}
Therefore we observe that
\begin{align*}
\MoveEqLeft{
\frac{4}{qq'} \tau \left\|\frac{|u_{j+1}|^{(q-2)/2}u_{j+1}-|u_j|^{(q-2)/2}u_j}{\tau}\right\|_{L^2(\Omega)}^2
}\\
&\quad +\frac{1}{2}\|u_{j+1}\|_{\X}^2-\frac{1}{2}\|u_j\|_{\X}^2\leq 0.
\end{align*}
Summing it up for $j=0,\dots,n$, we obtain
\begin{align}
\MoveEqLeft{
\frac{4}{qq'}\sum_{j=0}^n\tau \left\|\frac{|u_{j+1}|^{(q-2)/2}u_{j+1}-|u_j|^{(q-2)/2}u_j}{\tau}\right\|_{L^2(\Omega)}^2
}\nonumber\\
&\quad +\frac{1}{2}\|u_{n+1}\|_{\X}^2\leq \frac{1}{2}\|u_0\|_{\X}^2\label{est2disc}
\end{align}
for any $n=0,\dots,N-1$.

We (formally) test \eqref{En} by $|u_{j+1}|^{q-2}u_{j+1}$ to see that
\begin{align*}
\MoveEqLeft{
\int_\Omega \frac{|u_{j+1}|^{q-2}u_{j+1}-|u_j|^{q-2}u_j}{\tau} |u_{j+1}|^{q-2}u_{j+1} \, \d x
}\\
&\quad +
\left( u_{j+1}, |u_{j+1}|^{q-2}u_{j+1} \right)_{\X} \leq 0.
\end{align*}
Using Young's inequality (or convexity), we see that
\begin{align*}
\MoveEqLeft{
\int_\Omega \frac{|u_{j+1}|^{q-2}u_{j+1}-|u_j|^{q-2}u_j}{\tau} |u_{j+1}|^{q-2}u_{j+1} \, \d x
}\\
&\geq \frac{1}{2\tau} \int_\Omega |u_{j+1}|^{2(q-1)} \, \d x - \frac{1}{2\tau} \int_\Omega |u_{j}|^{2(q-1)} \, \d x.
\end{align*}
It also follows from the elementary inequality \eqref{EleIne} that
\begin{align*}
\MoveEqLeft{
( u_{j+1}(x) - u_{j+1}(y) ) \left( (|u_{j+1}|^{q-2}u_{j+1})(x) - (|u_{j+1}|^{q-2}u_{j+1})(y) \right)
}\\
&\geq \frac{4}{qq'} \left| (|u_{j+1}|^{(q-2)/2}u_{j+1})(x)- (|u_{j+1}|^{(q-2)/2}u_{j+1})(y) \right|^2
\end{align*}
for a.e.~$x,y\in\R^d$. Thus we finally obtain
\begin{align*}
\left( u_{j+1}, |u_{j+1}|^{q-2}u_{j+1} \right)_{\X} \geq \frac{4}{qq'} \left\| |u_{j+1}|^{(q-2)/2}u_{j+1} \right\|_{\X}^2
\end{align*}
for both cases $\theta \in (0,1)$ and $\theta = 1$ (see \S \ref{Ss:fSob}). Therefore if $|u_j|^{q-2}u_j \in L^2(\Omega)$, one can verify that $|u_{j+1}|^{q-2}u_{j+1}\in L^2(\Omega)$, $|u_{j+1}|^{(q-2)/2}u_{j+1}\in \X$ and
\begin{align}
\MoveEqLeft{
\frac{1}{2\tau} \left\| |u_{j+1}|^{q-2}u_{j+1} \right\|_{L^2(\Omega)}^2 + \frac{4}{qq'}\left\||u_{j+1}|^{(q-2)/2}u_{j+1}\right\|_{\X}^2
}\nonumber\\
& \leq \frac{1}{2\tau} \left\| |u_{j}|^{q-2}u_j \right\|_{L^2(\Omega)}^2.
\label{step_n}
\end{align}
Hence summing it up \eqref{step_n} for $j=0,\dots,n$ and using the assumption that $|u_0|^{q-2}u_0 \in L^2(\Omega)$ (see \eqref{hyp}), we conclude that
\begin{align}
\MoveEqLeft{
\frac{1}{2} \left\| |u_{n+1}|^{q-2}u_{n+1} \right\|_{L^2(\Omega)}^2 + \frac{4}{qq'} \sum_{j=0}^n \tau \left\| |u_{j+1}|^{(q-2)/2} u_{j+1} \right\|_{\X}^2
}\nonumber\\
&\leq \frac{1}{2} \left\| |u_0|^{q-2}u_0 \right\|_{L^2(\Omega)}^2\label{est3disc}
\end{align}
for any $n=0,\dots,N-1$.

These formal computations can be justified in a standard way; in case $0 < \theta < 1$, we may introduce a family $(\beta_\vep)_{\vep \in (0,1)}$ of monotone Lipschitz continuous functions satisfying $\beta_\vep(0) = 0$, $|\beta_\vep(r)| \leq |r|^{q-1}$ for $r \in \R$ and $\vep \in (0,1)$ and $\beta_\vep(r) \to |r|^{q-2}r$ as $\vep \to 0_+$ for each $r \in \R$, and then, test \eqref{En} by $\beta_\vep(u_{j+1}) \in \X$ instead of $|u_{j+1}|^{q-2}u_{j+1}$. We can then reach the same conclusion by passing to the limit as $\vep \to 0_+$ and by employing the dominated convergence theorem. In case $\theta = 1$, we may  also require additional properties that $|\beta_\vep'(r)| \leq C|r|^{q-2}$ for $r \in \R$ and $\vep \in (0,1)$ and $\beta'_\vep(r) \to (q-1)|r|^{q-2}$ as $\vep \to 0_+$ for each $r \in \R$.

For later use, we further derive an $L^\alpha$ estimate for $|u_n|^{q-2}u_n$, provided that $\rho_0 \in L^\alpha(\Omega)$ for some $\alpha > 1$. Setting $\rho_{j+1} := |u_{j+1}|^{q-2} u_{j+1}$, we also (formally) test \eqref{En} by $| \rho_{j+1} |^{\alpha-2} \rho_{j+1} = |u_{j+1}|^{\beta-2}u_{j+1}$ with $\beta = (q-1)(\alpha-1)+1 > 1$ and sum it up for $j = 0,\ldots,n$. We can similarly deduce from \eqref{EleIne} that
\begin{equation}\label{est4disc}
\frac 1 \alpha \left\| \rho_{n+1} \right\|_{L^\alpha(\Omega)}^\alpha + \frac{4}{\beta \beta'} \sum_{j=0}^n \left\| |u_{j+1}|^{(\beta-2)/2}u_{j+1} \right\|_{\X}^2 \leq \frac 1 \alpha \left\| \rho_0 \right\|_{L^\alpha(\Omega)}^\alpha,
\end{equation}
where $\beta' = \beta/(\beta - 1)$, for any $n = 0,\ldots, N-1$, provided that $\rho_0 \in L^\alpha(\Omega)$. These formal procedures can also be justified as above. 

\smallskip
\noindent
{\bf Step 3 (Interpolation and weak convergence).} We define a \emph{piecewise constant interpolant} $\bar u_\tau : [0,T] \to X$ of $\{u_n\}_{n=0}^N$ and \emph{piecewise linear interpolants} $\rho_\tau : [0,T] \to L^2(\Omega)$ and $w_\tau: [0,T] \to \X$ of $\{|u_n|^{q-2}u_n\}_{n=0}^N$ and $\{|u_n|^{(q-2)/2}u_n\}_{n=0}^N$, respectively, by setting
\begin{align*}
\bar u_\tau(t) &:= u_{n+1},\\
\rho_\tau(t) &:= |u_n|^{q-2}u_n + \frac{|u_{n+1}|^{q-2}u_{n+1}-|u_n|^{q-2}u_n}{\tau}(t-n\tau),\\
w_\tau(t) &:= |u_n|^{(q-2)/2}u_n + \frac{|u_{n+1}|^{(q-2)/2}u_{n+1} - |u_n|^{(q-2)/2}u_n}{\tau} (t-n\tau)
\end{align*}
for $t\in(n\tau, (n+1)\tau]$ and $n = 0,1,\ldots,N-1$ and
\begin{equation}\label{ini-disc}
\bar u_\tau(0) := u_0, \quad \rho_\tau(0) := |u_0|^{q-2}u_0, \quad w_\tau(0) := |u_0|^{(q-2)/2}u_0. 
\end{equation}
Equations \eqref{En} for $n = 0,1,\ldots,N-1$ are then reduced to the evolution equation,
\begin{equation}\label{eedisc}
\partial_t \rho_\tau(t) + (-\Delta)^\theta \bar u_\tau(t) = 0 \ \text{ in } \X^*, \quad t \in (0,T) \setminus I_0,
\end{equation}
where $I_0 := \{n\tau \colon n = 1,2,\ldots,N-1\}$. Moreover, the discrete energy inequalities \eqref{est1disc}, \eqref{est2disc} and \eqref{est3disc} can also be rewritten as
\begin{align}
\frac{1}{q'} \|\bar u_\tau(t)\|_{L^q(\Omega)}^q + \int_0^t \|\bar u_\tau(r)\|_{\X}^2\,\d r 
&\leq \frac{1}{q'} \|u_0\|_{L^q(\Omega)}^q,\label{enineq10}\\
\frac{4}{qq'} \int_0^t \|\partial_t w_\tau(r)\|_{L^2(\Omega)}^2\,\d r + \frac 12 \|\bar u_\tau(t)\|_{\X}^2 
&\leq \frac 12 \|u_0\|_{\X}^2 \label{enineq20}
\end{align}
and
\begin{align}
\MoveEqLeft{
\frac{1}{2} \left\| (|\bar u_\tau|^{q-2}\bar u_\tau)(t) \right\|_{L^2(\Omega)}^2 + \frac{4}{qq'} \int_0^t \left\| (|\bar u_\tau|^{(q-2)/2}\bar u_\tau)(r) \right\|_{\X}^2 \,\d r
}\nonumber\\
&\leq \frac{1}{2} \left\||u_0|^{q-2}u_0\right\|_{L^2(\Omega)}^2 \label{enineq30}
\end{align}
for all $t \in [0,T]$. Moreover, it follows from the definitions of $\rho_\tau$ and $w_\tau$ that
\begin{align}
\sup_{t \in [0,T]} \| \rho_\tau(t) \|_{L^2(\Omega)}^2
&= \sup_{t \in [0,T]} \left\| (|\bar u_\tau|^{q-2} \bar u_\tau)(t) \right\|_{L^2(\Omega)}^2,\nonumber\\
\sup_{t \in [0,T]} \| w_\tau(t) \|_{L^2(\Omega)}^2
&= \sup_{t \in [0,T]} \| \bar u_\tau(t) \|_{L^q(\Omega)}^q,\nonumber\\
\sup_{t \in (0,T) \setminus I_0} \| \partial_t \rho_\tau(t) \|_{\X^*}
&\stackrel{\eqref{eedisc}}= \sup_{t \in (0,T) \setminus I_0} \left\| (-\Delta)^\theta \bar u_\tau(t) \right\|_{\X^*}\nonumber\\
&=\sup_{t \in (0,T) \setminus I_0} \| \bar u_\tau(t) \|_{\X}.\label{drho}
\end{align}
Therefore taking a sequence $(\tau_n)$ in $(0,1)$ converging to zero, we can deduce that
\begin{alignat}{2}
\bar u_{\tau_n} &\to u \quad && \text{ weakly star in } L^\infty(0,T;X),\label{cvg2}\\
w_{\tau_n} &\to w \quad && \text{ weakly in } W^{1,2}(0,T;L^2(\Omega)),\label{cvg3}\\
|\bar u_{\tau_n}|^{(q-2)/2}\bar u_{\tau_n} &\to \tilde w && \text{ weakly in } L^2(0,T;\X),\label{cvg3.1}\\
\rho_{\tau_n} &\to \rho \quad && \text{ weakly star in } W^{1,\infty}(0,T;\X^*),\label{cvg4}\\
|\bar u_{\tau_n}|^{q-2}\bar u_{\tau_n} &\to \tilde\rho && \text{ weakly star in } L^\infty(0,T;L^2(\Omega)) \nonumber
\end{alignat}
for some weak (star) limits $u \in L^\infty(0,T;X)$, $w \in W^{1,2}(0,T;L^2(\Omega))$, $\tilde w \in L^2(0,T;\X)$, $\rho \in W^{1,\infty}(0,T;\X^*)$ and $\tilde \rho \in L^\infty(0,T;L^2(\Omega))$. Indeed, we observe that \eqref{enineq10} implies the boundedness of $(w_\tau)$ in $L^\infty(0,T;L^2(\Omega))$. Moreover, \eqref{enineq10} and \eqref{enineq20} yield \eqref{cvg2} and \eqref{cvg3}. Furthermore, \eqref{enineq20} along with \eqref{drho} implies that $(\partial_t \rho_{\tau})$ is bounded in $L^\infty(0,T;\X^*)$. Noting that
\begin{align*}
\|\rho_\tau(t)\|_{\X^*} \leq \left\| |u_0|^{q-2}u_0 \right\|_{\X^*} + \int^t_0 \|\partial_t \rho_\tau(r)\|_{\X^*} \, \d r
\end{align*}
for $t \in [0,T]$, we also find from \eqref{drho} that $(\rho_{\tau})$ is bounded in $L^\infty(0,T;\X^*)$. Thus \eqref{cvg4} follows.

Moreover, we can verify that $w = \tilde w$ and $\rho = \tilde \rho$, since $(\partial_t w_\tau)$ and $(\partial_t \rho_\tau)$ are uniformly bounded in $L^2(0,T;L^2(\Omega))$ and $L^\infty(0,T;\X^*)$, respectively. Indeed, we observe from the definition of $w_\tau$ that
\begin{align*}
\sup_{t \in [0,T]} \left\| w_\tau(t) - (|\bar{u}_\tau|^{(q-2)/2} \bar{u}_\tau)(t) \right\|_{L^2(\Omega)}
\stackrel{\eqref{est2disc}}\leq C \tau^{1/2} \|u_0\|_{\X} \to 0
\end{align*}
as $\tau \to 0_+$. Thus one derives $w = \tilde w$ from \eqref{cvg3} and \eqref{cvg3.1}. Similarly, we can prove $\rho = \tilde \rho$. It hence follows from \eqref{cvg2} and \eqref{cvg4} that
\begin{equation*}
\partial_t \rho(t) + (-\Delta)^\theta u(t) = 0 \ \text{ in } \X^*
\end{equation*}
for a.e.~$t \in (0,T)$.

When $|u_0|^{q-2}u_0 \in L^\alpha(\Omega)$ for some $\alpha > 1$, we can further derive from \eqref{est4disc} that
\begin{align}
\MoveEqLeft{
\frac{1}{\alpha} \left\| (|\bar u_\tau|^{q-2}\bar u_\tau)(t) \right\|_{L^\alpha(\Omega)}^\alpha + \frac{4}{\beta\beta'} \int_0^t \left\| (|\bar u_\tau|^{(\beta-2)/2}\bar u_\tau)(r) \right\|_{\X}^2 \,\d r
}\nonumber\\
&\leq \frac{1}{\alpha} \left\||u_0|^{q-2}u_0\right\|_{L^\alpha(\Omega)}^\alpha \label{enineq40}
\end{align}
for all $t \in [0,T]$, and it then follows that
\begin{alignat*}{3}
|\bar u_\tau|^{q-2}\bar u_\tau &\to \rho \quad &&\mbox{ weakly star in } L^\infty(0,T;L^\alpha(\Omega)),\\
|\bar u_\tau|^{(\beta-2)/2}\bar u_\tau &\to \zeta \quad &&\mbox{ weakly in } L^2(0,T;\X)
\end{alignat*}
for some $\zeta \in L^2(0,T;\X)$. Also, $\rho$ turns out to be of class $L^\infty(0,T;L^\alpha(\Omega))$.

\smallskip
\noindent
{\bf Step 4 (Strong convergence and identification of weak limits).} We first verify a \emph{strong} convergence of the approximate solutions in order to identify the weak limits $\rho = |u|^{q-2}u$ and $w = |u|^{(q-2)/2}u$ as well as to check the initial condition. Recall from Step 3 that $(\partial_t \rho_\tau)$ is bounded in $L^\infty(0,T;\X^*)$ uniformly for $\tau \in (0,1)$. Therefore one has
\begin{align*}
\|\rho_\tau(t)-\rho_\tau(s)\|_{\X^*} 
&\leq \int_s^t \|\partial_t \rho_\tau(r)\|_{\X^*} \, \d r\\
&\leq C |t-s| \quad \text{ for any } \ 0 \leq s < t \leq T
\end{align*}
for some $C > 0$ independent of $\tau \in (0,1)$, that is, the family $(\rho_\tau)$ is equicontinuous on $[0,T]$ in the strong topology of $\X^*$. Moreover, it follows from \eqref{est3disc} that
\begin{equation*}
\sup_{t \in [0,T]}\|\rho_{\tau}(t)\|_{L^2(\Omega)} \leq \left\||u_0|^{q-2}u_0 \right\|_{L^2(\Omega)},
\end{equation*}
provided that $|u_0|^{q-2}u_0 \in L^2(\Omega)$ (see \eqref{hyp}). Since $\X$ is compactly embedded in $L^2(\Omega)$, we can deduce (with the aid of Schauder's theorem) that $L^2(\Omega)\equiv L^2(\Omega)^*$ is compactly embedded into $\X^*$ (see Theorem \ref{P:SP} and \eqref{pivot}). Therefore, for each $t \in [0,T]$, the family $(\rho_\tau(t))$ is precompact in $\X^*$. Applying Ascoli's theorem, we can extract a (not relabeled) subsequence of $(\tau_n)$ such that 
\begin{equation}\label{cvgstrong}
\rho_{\tau_n} \to \rho \quad \text{ strongly in } C([0,T];\X^*).
\end{equation}
Moreover, recalling the definitions of $\bar u_\tau$ and $\rho_\tau$ as well as the boundedness of $(\partial_t \rho_\tau)$, we find that
$$
\sup_{t \in [0,T]} \left\| (|\bar u_\tau|^{q-2} \bar u_\tau)(t) - \rho_\tau(t) \right\|_{\X^*} \leq \tau \sup_{t \in [0,T]} \| \partial_t \rho_\tau(t) \|_{\X^*} \leq C \tau \to 0,
$$
which along with \eqref{cvgstrong} implies that
\begin{equation}\label{cvguq1}
|\bar u_{\tau_n}|^{q-2} \bar u_{\tau_n} \to \rho \quad \text{ strongly in } L^\infty(0,T;\X^*).
\end{equation}
We next claim that $w = |u|^{(q-2)/2}u$ and $\rho = |u|^{q-2}u$ a.e.~in $Q_T := \Omega \times (0,T)$. Indeed, it follows from \eqref{EleIne}, \eqref{cvg2} and \eqref{cvguq1} that
\begin{align*}
\MoveEqLeft{
\frac{4}{qq'} \int_0^T \int_{\Omega} \left| |\bar u_{\tau_n}|^{(q-2)/2} \bar u_{\tau_n} - |u|^{(q-2)/2}u \right|^2\, \d x \d t
}\\
&\leq \int_0^T \int_{\Omega} \left( |\bar u_{\tau_n}|^{q-2} \bar u_{\tau_n} - |u|^{q-2}u \right) ( \bar u_{\tau_n} - u ) \, \d x \d t \\
&= \int_0^T \left\langle |\bar u_{\tau_n}|^{q-2} \bar u_{\tau_n} - \rho, \bar u_{\tau_n} - u \right\rangle_{\X} \, \d t \\
&\quad + \int_0^T \left\langle \rho - |u|^{q-2}u, \bar u_{\tau_n} - u \right\rangle_{\X} \,\d t \to 0.
\end{align*}
Therefore we deduce that $|\bar u_{\tau_n}|^{(q-2)/2} \bar u_{\tau_n} \to |u|^{(q-2)/2}u$ strongly in $L^2(Q_T)$, and hence, we obtain $w = |u|^{(q-2)/2}u$. Moreover, we can verify that
\begin{alignat}{2}
\bar u_{\tau_n} &\to u \quad &&\text{ strongly in } L^q(Q_T),\label{cvgu2}\\
|\bar u_{\tau_n}|^{q-2} \bar u_{\tau_n} &\to |u|^{q-2}u \quad &&\text{ strongly in } L^{q'}(Q_T) \label{cvguq2}
\end{alignat}
(see Lemma \ref{L:NemOp} in Appendix). Hence \eqref{cvguq1} and \eqref{cvguq2} yield $\rho = |u|^{q-2}u$ a.e.~in $Q_T$. 

We further observe that
$$
\rho(0) \stackrel{\eqref{cvgstrong}}= \lim_{n \to \infty} \rho_{\tau_n}(0) \stackrel{\eqref{ini-disc}}= |u_0|^{q-2}u_0,
$$
which along with the fact that $\rho \in C([0,T];\X^*)$ gives
$$
(|u|^{q-2}u)(t) = \rho(t) \to \rho_0 = |u_0|^{q-2}u_0 \quad \text{ strongly in } \X^*
$$
as $t \to 0_+$. Thus $u$ turns out to be a weak solution on $[0,T]$ to the Cauchy--Dirichlet problem \eqref{pde}--\eqref{ic} in the sense of Definition \ref{D:sol}.

Finally, we derive additional regularity of the weak solution as well as energy inequalities. Since $W^{1,2}(0,T;L^2(\Omega)) \subset C([0,T];L^2(\Omega))$, we also see that $w = |u|^{(q-2)/2}u \in C([0,T];L^2(\Omega))$. Moreover, we can deduce from Lemma \ref{L:NemOp} that $u \in C([0,T];L^q(\Omega))$ and $\rho = |u|^{q-2}u \in C([0,T];L^{q'}(\Omega))$, which along with $u \in L^\infty(0,T;\X)$ and $|u|^{q-2}u \in L^\infty(0,T;L^2(\Omega))$ yield $u \in C_{\rm weak}([0,T];\X)$ and $|u|^{q-2}u \in C_{\rm weak}([0,T];L^2(\Omega))$, respectively (see~\cite[Chap.\,3, Lemma 8.1]{LM}).\prf{\footnote{\UUU We have already known that $u\in C([0,\infty);L^q(\Omega)) \cap L^\infty(0,\infty;\X)$. Let $t_0 \in [0,\infty)$ and let $t_n \in [0,\infty)$ be an arbitrary sequence converging to $t_0$. Since $(u(t_n))$ is bounded in $\X$, we may extract a subsequence (still denoted by $(t_n)$) such that $u(t_n)$ converges to a function $v \in \X$ weakly in $\X$. Moreover, $u(t_n) \to u(t_0)$ strongly in $L^q(\Omega)$, and hence, we can deduce that $v = u(t_0)$. Thus,
$$
u(t_n) \to u(t_0) \quad \text{ weakly in } \X,
$$
and therefore, we obtain $u \in C_{\rm weak}([0,\infty);\X)$.}} Furthermore, recalling the uniform estimates from \eqref{enineq10}--\eqref{enineq30} along with \eqref{cvgu2} and \eqref{cvguq2}, for a.e.~$t \in (0,T)$, we can derive, up to a (not relabeled) subsequence, that
\begin{alignat*}{2}
\bar u_{\tau_n} (t) &\to u(t) \quad &&\text{ weakly in } \X,\\
(|\bar u_{\tau_n}|^{q-2}\bar u_{\tau_n})(t) &\to (|u|^{q-2}u)(t) \quad &&\text{ weakly in } L^2(\Omega).
\end{alignat*}
Thus combining \eqref{enineq10}--\eqref{enineq30} with the convergence proved so far, we obtain
\begin{align}
\frac{1}{q'} \|u(t)\|_{L^q(\Omega)}^q + \int_0^t \|u(r)\|_{\X}^2\,\d r
&\leq \frac{1}{q'} \|u_0\|_{L^q(\Omega)}^q,\label{enineq10'}\\
\frac{4}{qq'} \int_0^t \left\| \partial_t (|u|^{(q-2)/2}u)(r) \right\|_{L^2(\Omega)}^2\, \d r + \frac 12 \|u(t)\|_{\X}^2 
&\leq \frac 12 \|u_0\|_{\X}^2,\label{enineq20'}
\end{align}
\begin{align}
\MoveEqLeft{
\frac{1}{2} \left\| (|u|^{q-2}u)(t) \right\|_{L^2(\Omega)}^2 + \frac{4}{qq'} \int_0^t \left\| (|u|^{(q-2)/2}u)(r) \right\|_{\X}^2 \,\d r 
}\nonumber\\
&\leq \frac{1}{2} \left\| |u_0|^{q-2}u_0 \right\|_{L^2(\Omega)}^2\label{enineq30'}
\end{align}
for all $t \in [0,T]$. Indeed, we first derive the energy inequalities above for a.e.~$t \in (0,T)$, and then, using the weak continuities along with the weak lower semicontinuity of norms, they can be extended to any $t \in [0,T]$. Furthermore, thanks to the uniqueness of weak solutions (see \S \ref{Ss:uniq}), the above energy inequalities can also be extended to the general forms of \eqref{enineq1}--\eqref{enineq3} for any $0 \leq s \leq t \leq T$. Indeed, for any fixed $s \in [0,T)$, the function $\tilde{u}(\cdot) := u(\,\cdot\,+s)$ turns out to be a weak solution on $[0,T-s]$ of \eqref{pde}--\eqref{ic} with the initial datum $\rho(s) = (|u|^{q-2}u)(s)$, which satisfies \eqref{hyp} due to the regularity obtained so far. Hence from the uniqueness of weak solutions as well as the energy inequalities \eqref{enineq10'}--\eqref{enineq30'}, the assertions \eqref{enineq1}--\eqref{enineq3} follow immediately.

We further claim that $u \in C_+([0,T];\X)$. Indeed, we have already proved that $u \in C_{\rm weak}([0,T];\X)$. Moreover, it follows from \eqref{enineq2} that
$$
\limsup_{t \searrow s} \|u(t)\|_{\X} \leq \|u(s)\|_{\X}
$$
for each $s \in [0,T)$. Thus due to the uniform convexity of $\|\cdot\|_{\X}$, $u(t)$ converges to $u(s)$ strongly in $\X$ as $t \searrow s$ for each $s \in [0,T)$, that is, $u \in C_+([0,T];\X)$. We can similarly show that $|u|^{q-2}u \in C_+([0,T];L^2(\Omega))$ by using \eqref{enineq3} instead of \eqref{enineq2}. Moreover, the (weak) continuity of $(-\Delta)^\theta : \X \to \X^*$ gives
$$
\partial_t (|u|^{q-2}u) = - (-\Delta)^\theta u \in C_+([0,T];\X^*) \cap C_{\rm weak}([0,T];\X^*).
$$

Due to the arbitrariness of $T > 0$ as well as the uniqueness of weak solutions, one can construct a global-in-time weak solution $u : \Omega \times [0,\infty) \to \R$ to the Cauchy--Dirichlet problem \eqref{pde}--\eqref{ic} satisfying the energy inequalities \eqref{enineq1}--\eqref{enineq3} for any $0 \leq s \leq t < \infty$ as well as the regularities listed in the statement of Lemma \ref{L:ex-ensol} (see also \eqref{enineq0} in Remark \ref{R:ws-reg}). 

Concerning $q > 2$, we claim that $\partial_t (|u|^{q-2}u)$ belongs to $L^2(0,\infty;L^{q'}(\Omega))$. Indeed, by the Mean-Value Theorem and H\"older's inequality, we find that
\begin{align*}
\MoveEqLeft{
\left\| (|u|^{q-2}u)(t_1) - (|u|^{q-2}u)(t_2) \right\|_{L^{q'}(\Omega)}
}\\
&\leq C \left\| (|u|^{(q-2)/2}u)(t_1) - (|u|^{(q-2)/2}u)(t_2) \right\|_{L^2(\Omega)}\\
&\leq C \int^{t_2}_{t_1} \left\| \partial_t (|u|^{(q-2)/2}u)(r) \right\|_{L^2(\Omega)} \,\d r \quad \mbox{ for } \ 0 < t_1 < t_2 < \infty,
\end{align*}
which along with $\|\partial_t (|u|^{(q-2)/2}u)(\cdot) \|_{L^2(\Omega)} \in L^2(0,\infty)$ and~\cite[Theorem 1.4.40]{CaHa} yields the desired conclusion.

When $\rho_0 \in L^\alpha(\Omega)$ for some $\alpha > 1$, we can also derive from \eqref{enineq40} that
\begin{align*}
\MoveEqLeft{
\frac{1}{\alpha} \left\| (|u|^{q-2}u)(t) \right\|_{L^\alpha(\Omega)}^\alpha + \frac{4}{\beta\beta'} \int_0^t \left\| (|u|^{(\beta-2)/2}u)(r) \right\|_{\X}^2 \,\d r
}\\
&\leq \frac{1}{\alpha} \left\||u_0|^{q-2}u_0\right\|_{L^\alpha(\Omega)}^\alpha \quad \mbox{ for } \ t \geq 0,
\end{align*}
which along with the uniqueness result yields \eqref{enineq4} as well. Here we used the relation $\zeta = |u|^{(\beta-2)/2}u$, which follows from the fact that $\bar u_\tau \to u$ a.e.~in $\Omega \times (0,\infty)$, up to a subsequence, by virtue of \eqref{cvgu2}. Finally, \eqref{enineq4} together with $|u|^{q-2}u\in C([0,\infty),L^{q'}(\Omega))$ implies $|u|^{q-2}u \in C_{\rm weak}([0,\infty);L^\alpha (\Omega))$ and $|u|^{q-2}u \in C_+([0,\infty);L^\alpha (\Omega))$. This completes the proof.
\end{proof}

\subsection{Energy identity}\label{Ss:enid}

In this subsection, we shall derive the energy \emph{identity} \eqref{energyineq1} of Theorem \ref{T:main} for general {weak} solutions in the sense of Definition \ref{D:sol}. Actually, we have obtained a similar inequality \eqref{enineq1} in Lemma \ref{L:ex-ensol}; however, it may not be enough to verify the absolute continuity of the function $t \mapsto \|u(t)\|_{L^q(\Omega)}^q$ as well as the energy \emph{identity} \eqref{energyineq1}. Instead, we shall use a chain-rule formula for subdifferentials and directly prove the absolute continuity and energy identity for general {weak} solutions to \eqref{pde}--\eqref{ic}. 

Let $E$ and $E^*$ be a reflexive Banach space and its dual space, respectively. Let $\psi : E \to (-\infty,\infty]$ be a proper (i.e., $\psi \not\equiv \infty$) lower-semicontinuous convex functional. Moreover, let $\psi^* : E^* \to (-\infty,\infty]$ be the \emph{convex conjugate} (or \emph{Legendre transform}) of $\psi$ defined by
$$
\psi^*(\eta) = \sup_{w \in E} \left\{ \langle \eta,w \rangle_E - \psi(w) \right\} \ \text{ for } \ \eta \in E^*.
$$
Then it is well known that $\psi^*$ is also proper lower-semicontinuous and convex in $E^*$ (see, e.g.,~\cite[\S 1.4]{B-FA}). Let $\partial_E \psi : E \to 2^{E^*}$ and $\partial_{E^*} \psi^* : E^* \to 2^E$ denote subdifferential operator of $\psi$ and $\psi^*$, respectively, that is,
$$
\partial_E \psi(w) = \left\{ \eta \in E^* \colon \psi(v) - \psi(w) \geq \langle \eta, v - w \rangle_E \ \text{ for } \ v \in E\right\}
$$
for $w \in D(\psi) := \{v \in E \colon \psi(v) < \infty\}$ with domain $D(\partial_E \psi) := \{v \in D(\psi) \colon \partial_E \psi(v) \neq \emptyset\}$ (moreover, $\partial_{E^*} \psi^*$ can be defined analogously with the identification $E^{**} \equiv E$). Then we recall that, for $(w,\eta) \in E \times E^*$, 
$$
\eta \in \partial_E \psi(w) \quad \text{ if and only if } \quad w \in \partial_{E^*} \psi^*(\eta)
$$ 
(see, e.g.,~\cite{B}). We recall that

\begin{lemma}[Chain-rule for subdifferentials of convex conjugates]\label{L:chainrule}
Let $1 < p < \infty$ and let $(a,b)$ be a nonempty open interval in $\R$. Let $E$ be a reflexive Banach space and let $\psi :E \to (-\infty,\infty]$ be a proper lower semicontinuous convex functional over $E$. Let $u \in L^p(a,b;E)$ and $\xi \in W^{1,p'}(a,b;E^*)$ be such that 
\begin{equation}\label{xi-hyp}
\xi(t) \in \partial_E \psi(u(t)) \quad \text{ for a.e. } t \in (a,b).
\end{equation}
Then the function $t\mapsto \psi^*(\xi(t))$ is absolutely continuous on $[a,b]$, and moreover, it holds that
$$
\frac{\d}{\d t} \psi^*(\xi(t)) = \langle \partial_t \xi(t), u(t) \rangle_E \quad \text{ for a.e. } t \in (a,b).
$$
\end{lemma}

\begin{proof}
The lemma above is a variant of the well-known chain-rule formula for subdifferentials in convex analysis (cf.~see~\cite[Lemma 3.3]{HB1}). Indeed, one can prove it by combining the chain-rule formula~\cite[Lemma 4.1]{Colli} (see the appendix of~\cite{A23s} for a proof) with the facts that $E^{**}$ is identified with $E$ by reflexivity and \eqref{xi-hyp} is equivalent to $u(t) \in \partial_{E^*}\psi(\xi(t))$ for a.e.~$t \in (a,b)$.
\end{proof}

Let $T \in (0,\infty)$ and let $u = u(x,t)$ be a {weak} solution on $[0,T]$ to \eqref{pde}--\eqref{ic} in the sense of Definition \ref{D:sol}. Testing \eqref{weak_form} with $u(t)\in\X$, we find that
\begin{equation}\label{enidpre}
\left\langle \partial_t (|u|^{q-2}u)(t), u(t) \right\rangle_{\X} + \|u(t)\|_{\X}^2 = 0
\end{equation}
for a.e.~$t \in (0,T)$. We apply Lemma \ref{L:chainrule} by setting
$$
E = X := \X \cap L^q(\Omega) \quad \text{ and } \quad \psi(w) = \frac 1q \|w\|_{L^q(\Omega)}^q \ \text{ for } \ w \in X.
$$
Then $\psi$ is of class $C^1$ in $X$ and $\partial_X \psi(w) = \{|w|^{q-2}w\}$ for $w \in X$ (indeed, the functional $w \mapsto (1/q)\|w\|_{L^q(\Omega)}^q$ defined for $w \in L^q(\Omega)$ is already of class $C^1$ and its Fr\'echet derivative at $w$ coincides with $|w|^{q-2}w$ in $(L^q(\Omega))^* \equiv L^{q'}(\Omega)$, and hence, a.e.~in $\Omega$ as well). Hence from Definition \ref{D:sol}, we find that $\xi = |u|^{q-2}u \in W^{1,2}_{\rm loc}((0,T];\X^*) \hookrightarrow W^{1,2}_{\rm loc}((0,T];X^*)$ and $u \in L^2_{\rm loc}((0,T];X)$, which is obvious for $q \geq 2$ and still holds true for $1 < q < 2$ from the relation $X = \X$. Therefore thanks to Lemma \ref{L:chainrule} the function $t \mapsto \psi^*(\xi(t)) = \psi^*((|u|^{q-2}u)(t))$ turns out to be absolutely continuous on $(0,T]$ such that
$$
\frac{\d}{\d t} \psi^*( (|u|^{q-2}u)(t) ) = \left\langle \partial_t (|u|^{q-2}u)(t), u(t) \right\rangle_X
$$
for a.e.~$t \in (0,T)$. Moreover, we note that
\begin{align*}
 \psi^*((|u|^{q-2}u)(t)) &= \sup_{w \in X} \left\{ \left\langle (|u|^{q-2}u)(t), w \right\rangle_X - \psi(w) \right\}\\
&\geq \left\langle (|u|^{q-2}u)(t), u(t) \right\rangle_X - \psi(u(t)) = \frac 1 {q'} \|u(t)\|_{L^q(\Omega)}^q,
\end{align*}
since $u(t)$ lies on $X$, for a.e.~$t \in (0,T)$. On the other hand, since $X \subset L^q(\Omega)$, it follows by a simple computation that
\begin{align*}
 \psi^*((|u|^{q-2}u)(t)) &\leq \sup_{w \in L^q(\Omega)} \left\{ \left\langle (|u|^{q-2}u)(t), w \right\rangle_X - \psi(w) \right\}\\
&= \frac 1 {q'} \left\| (|u|^{q-2}u)(t) \right\|_{L^{q'}(\Omega)}^{q'} = \frac 1 {q'} \|u(t)\|_{L^q(\Omega)}^q.
\end{align*}
Combining these facts, we obtain
$$
\psi^*( (|u|^{q-2}u)(t)) = \frac 1 {q'} \|u(t)\|_{L^q(\Omega)}^q
$$
for a.e.~$t \in (0,T)$. Therefore the function $t\mapsto \|u(t)\|_{L^q(\Omega)}^q$ is absolutely continuous on $(0,T]$, and 
$$
\frac{1}{q'} \frac{\d}{\d t} \|u(t)\|_{L^q(\Omega)}^q = \left\langle \partial_t (|u|^{q-2}u)(t),u(t) \right\rangle_{\X}
$$
for a.e.~$t \in (0,T)$. Here we also used the fact that $\partial_t (|u|^{q-2}u)(t)$ and $u(t)$ lie in $\X^*$ and $\X$, respectively, for a.e.~$t \in (0,T)$. Thus using \eqref{enidpre}, we obtain \eqref{energyineq1} for a.e.~$t \in (0,T)$. 

Moreover, as we have seen in Lemma \ref{L:ex-ensol} that $|u|^{q-2}u \in W^{1,2}_{\rm loc}([0,\infty);\X^*)$ and $u \in L^2(0,\infty;\X)$, we also obtain the absolute continuity of the function $t \mapsto \|u(t)\|_{L^q(\Omega)}^q$ on $[0,\infty)$ as well as the energy identity \eqref{energyineq1} for a.e.~$t > 0$. 

\begin{remark}
In this subsection, we have proved the energy identity \eqref{energyineq1} within the present scheme. As for $\theta = 1$, one can prove it more easily by combining the $H^{-1}$ framework for \eqref{pde}--\eqref{ic}. Indeed, the absolute continuity of the function $t \mapsto (m+1)^{-1}\|\rho(t)\|_{L^{m+1}(\Omega)}^{m+1} = (1/q')\|u(t)\|_{L^q(\Omega)}^q$ and the energy identity follow immediately in the $H^{-1}$ framework (see~\cite{HB2}). Moreover, due to the uniqueness (see \S \ref{Ss:uniq} and Remark \ref{R:others}), they also hold true for {weak} solutions. A similar argument may be available for $0 < \theta < 1$.
\end{remark}

\subsection{For general initial data}\label{Ss:gen}

In this subsection, we shall extend the results obtained in Lemma \ref{L:ex-ensol} (with some exceptions) to initial data
\begin{equation}\label{hyp2}
u_0 \in \X \cap L^q(\Omega) \quad \text{ satisfying \eqref{ini-hyp+} \ (instead of \eqref{hyp})}
\end{equation}
for completing a proof for (ii) of Theorem \ref{T:main}.

\begin{lemma}
For every $u_0 \in \X \cap L^q(\Omega)$ satisfying \eqref{ini-hyp+}, all the assertions as in {\rm (ii)} of Theorem {\rm \ref{T:main}} hold true.
\end{lemma}

\begin{proof}
Set $u_{0,n} := \max\{\min \{u_0,n\},-n\} \in \X \cap L^\infty(\Omega)$ for $n \in \N$. Then it follows that
\begin{alignat*}{2}
u_{0,n} &\to u_0 \quad && \mbox{ strongly in } \X \cap L^q(\Omega),\\
\rho_{0,n} := |u_{0,n}|^{q-2}u_{0,n} &\to \rho_0 \quad &&\text{ strongly in } L^{(2^*_\theta)'}(\Omega) \hookrightarrow \X^*
\end{alignat*}
as $n \to \infty$. For each $n \in \N$, let $u_n$ be the unique energy solution to the Cauchy--Dirichlet problem \eqref{pde}--\eqref{ic} with the initial datum $u_{0,n}$. Let $T>0$ be fixed. We deduce from \eqref{conti-dep} that $(|u_n|^{q-2}u_n)$ forms a Cauchy sequence in $C([0,T];\X^*)$, and hence,
\begin{equation}\label{2:cvgs} 
|u_n|^{q-2}u_n \to \rho \quad \text{ strongly in } C([0,T];\X^*)
\end{equation}
for some $\rho \in C([0,T];\X^*)$. In particular, since $(|u_n|^{q-2}u_n)(0) \to \rho(0)$ strongly in $\X^*$, we have $\rho(0)=\rho_0$. Moreover, one can derive from \eqref{enineq1} (or \eqref{energyineq1}) and \eqref{enineq2} that
\begin{align*}
\frac{1}{q'} \|u_n(t)\|_{L^q(\Omega)}^q + \int_0^t \|u_n(r)\|_{\X}^2 \, \d r
&\leq \frac{1}{q'} \|u_{0,n}\|_{L^q(\Omega)}^q
\end{align*}
and
\begin{align}
\MoveEqLeft{
\frac{4}{qq'} \int_0^t \left\| \partial_t (|u_n|^{(q-2)/2}u_n)(r) \right\|_{L^2(\Omega)}^2 \, \d r + \frac{1}{2} \|u_n(t)\|_{\X}^2
}\nonumber\\
&\leq \frac{1}{2} \|u_{0,n}\|_{\X}^2\label{equ2}
\end{align}
for all $t \in [0,T]$. Since $u_{0,n}\rightarrow u_0$ strongly in $\X \cap L^q(\Omega)$, the sequences $(u_n)$ and $(|u_n|^{(q-2)/2}u_n)$ are bounded in $L^\infty(0,T; \X \cap L^q(\Omega))$ and in $W^{1,2}(0,T; L^2(\Omega))$, respectively. Moreover, using the relations
\begin{equation}\label{equ}
\partial_t (|u_n|^{q-2}u_n)(t) + (-\Delta)^\theta u_n(t) = 0 \ \text{ in } \X^* 
\end{equation}
as well as $\|(-\Delta)^\theta u_n(t)\|_{\X^*} = \|u_n(t)\|_{\X}$ for a.e.~$t \in (0,T)$, one can verify that $(|u_n|^{q-2}u_n)$ is bounded in $W^{1,\infty}(0,T;\X^*)$ as in \S \ref{Ss:ex}. Moreover, recalling \eqref{enineq4} with $\alpha = (2^*_\theta)'$ and $\beta = (q-1)((2^*_\theta)'-1)+1$, we see that
\begin{align*}
\MoveEqLeft{
\frac{1}{\alpha} \left\| (|u_n|^{q-2}u_n)(t) \right\|_{L^\alpha(\Omega)}^\alpha + \frac{4}{\beta\beta'} \int_0^t \left\| (|u_n|^{(\beta-2)/2}u_n)(r) \right\|_{\X}^2 \,\d r
}\nonumber\\
&\leq \frac{1}{\alpha} \left\| \rho_{0,n} \right\|_{L^\alpha(\Omega)}^\alpha
\end{align*}
for $t \geq 0$ and $n \in \N$. Since $\rho_{0,n} \to \rho_0$ strongly in $L^{(2^*_\theta)'}(\Omega)$, we can conclude that $(|u_n|^{q-2}u_n)$ is bounded in $L^\infty(0,T;L^{(2^*_\theta)'}(\Omega))$. 

Therefore there exists a (not relabeled) subsequence of $(n)$ such that
\begin{alignat}{2}
u_n &\to u \quad &&\text{ weakly star in } L^\infty(0,T,\X \cap L^q(\Omega)),\label{2:cvg1}\\
|u_n|^{q-2}u_n &\to \rho \quad &&\text{ weakly star in } W^{1,\infty}(0,T;\X^*),\nonumber\\
& &&\mbox{ weakly star in } L^\infty(0,T;L^{(2^*_\theta)'}(\Omega)),\nonumber\\
|u_n|^{(q-2)/2}u_n &\to w \quad &&\text{ weakly in } W^{1,2}(0,T;L^2(\Omega))\nonumber
\end{alignat}
for some weak (star) limits $u$, $\rho$ and $w$. Passing to the limit in \eqref{equ}, we obtain $\partial_t \rho + (-\Delta)^\theta u = 0$ in $L^\infty(0,T;\X^*)$. It remains to check the relations $\rho = |u|^{q-2}u$ and $w = |u|^{(q-2)/2}u$. As in the proof of Lemma \ref{L:ex-ensol} (see Step 4), we can verify from \eqref{EleIne}, \eqref{2:cvgs} and \eqref{2:cvg1} that
\begin{align*}
\MoveEqLeft{
\frac{4}{pp'} \int_0^T \int_{\Omega} \left| |u_n|^{(q-2)/2}u_n - |u|^{(q-2)/2}u \right|^2\, \d x\d t
}\\
&\leq \int_0^T \int_{\Omega} \left( |u_n|^{q-2}u_n - |u|^{q-2}u \right) ( u_n - u )\, \d x \d t\\
&= \int_0^T \left\langle |u_n|^{q-2}u_n - \rho, u_n - u \right\rangle_{\X}\, \d t\\
&\quad + \int_0^T \left\langle \rho - |u|^{q-2}u, u_n - u \right\rangle_{\X}\, \d t \to 0.
\end{align*}
Therefore $(|u_n|^{(q-2)/2}u_n)$ converges to $|u|^{(q-2)/2}u$ strongly in $L^2(Q_T)$, and moreover, as in Step 4 of the proof for Lemma \ref{L:ex-ensol}, we can deduce that $\rho = |u|^{q-2}u$ and $w = |u|^{(q-2)/2}u$ a.e.~in $Q_T$. Furthermore, repeating the same argument as in Step 4 of the proof for Lemma \ref{L:ex-ensol}, one can obtain $u \in C([0,T];L^q(\Omega)) \cap C_{\rm weak}([0,T];\X) \cap C_+([0,T];\X)$, $\rho = |u|^{q-2}u \in C([0,T];L^{q'}(\Omega)) \cap C_{\rm weak}([0,T];L^{(2^*_\theta)'}(\Omega)) \cap C_+([0,T];L^{(2^*_\theta)'}(\Omega))$ and $w = |u|^{(q-2)/2}u \in C([0,T];L^2(\Omega))$. Hence it follows from \eqref{2:cvgs} that
$$
u_n(t) \to u(t) \quad \text{ weakly in } \X \quad \text{ for all } \ t \in [0,T],
$$
which along with \eqref{equ2} yields
\begin{align*}
\frac{4}{qq'} \int_0^t \left\|\partial_t (|u|^{(q-2)/2}u)(r) \right\|_{L^2(\Omega)}^2 \, \d r + \frac{1}{2} \|u(t)\|_{\X}^2
&\leq \frac{1}{2} \|u_0\|_{\X}^2
\end{align*}
for $t \in [0,T]$. Moreover, since $u(s)$ complies with \eqref{hyp2} for $s > 0$, from the uniqueness of the weak solution, we can derive \eqref{energyineq2} for $0 \leq s < t < \infty$ as in the proof of Lemma \ref{L:ex-ensol}. 
\end{proof}

This completes the proof for (ii) of Theorem \ref{T:main}.

\begin{remark}[Weak solutions for more general initial data]\label{R:gen}
If we restrict ourselves to weak solutions of \eqref{pde}--\eqref{ic}, the assumptions on initial data can be relaxed more by establishing smoothing estimates. Suppose that $\Omega$ is smooth so that $\X$ is dense in $L^2(\Omega)$ (see Proposition \ref{P:density}). Let $\rho_0 \in \X^*$ be fixed. We can then take a sequence $(u_{0,n})$ in $\X \cap L^q(\Omega)$ such that $\rho_{0,n} := |u_{0,n}|^{q-2}u_{0,n} \to \rho_0$ strongly in $\X^*$ and $\rho_{0,n} \in L^2(\Omega)$ for $n \in \N$ (see Appendix \S \ref{S:regu}). Then thanks to (ii) of Theorem \ref{T:main}, the Cauchy--Dirichlet problem \eqref{pde}--\eqref{ic} with $\rho_0 = \rho_{0,n}$ admits a unique energy solution $u_n$. It follows from \eqref{conti-dep} that 
$$
|u_n|^{q-2}u_n \to \rho \quad \mbox{ strongly in } C([0,T];\X^*)
$$
for some $\rho \in C([0,T];\X^*)$. Furthermore, we test \eqref{ee} with $u = u_n$ by $(-\Delta)^{-\theta} (|u_n|^{q-2}u_n)(t)$ to see that
$$
\frac 1 2 \frac{\d}{\d t} \left\| (|u_n|^{q-2}u_n)(t) \right\|_{\X^*}^2 + \|u_n(t)\|_{L^q(\Omega)}^q = 0
$$
for a.e.~$t > 0$. Integrating both sides over $(0,t)$, we find that
\begin{equation*}
\frac 1 2 \left\| (|u_n|^{q-2}u_n)(t) \right\|_{\X^*}^2 + \int^t_0 \|u_n(r)\|_{L^q(\Omega)}^q \, \d r = \frac 1 2 \left\| \rho_{0,n} \right\|_{\X^*}^2
\end{equation*}
for $t \geq 0$. Moreover, multiplying \eqref{energyineq1} with $u = u_n$ by $t$, we infer that 
$$
\frac 1 {q'} \frac{\d}{\d t} \left( t \|u_n(t)\|_{L^q(\Omega)}^q \right) + t \|u_n(t)\|_{\X}^2 = \|u_n(t)\|_{L^q(\Omega)}^q
$$
for a.e.~$t > 0$. Therefore combining these facts, we obtain
\begin{equation}\label{ap:1}
\sup_{t \in [0,\infty)} \left( t \|u_n(t)\|_{L^q(\Omega)}^q \right) + \int^\infty_0 t \|u_n(t)\|_{\X}^2 \, \d t \leq C \left\| \rho_{0,n} \right\|_{\X^*}^2.
\end{equation}
Hence for any $\vep > 0$ one can take $s_\vep \in (\vep,2\vep)$ (which may depend on $n$) such that
$$
s_\vep \|u_n(s_\vep)\|_{\X}^2 \leq \frac 1 \vep \int^{2\vep}_\vep t\|u_n(t)\|_{\X}^2 \, \d t,
$$
which along with \eqref{ap:1} yields
\begin{equation}\label{ap:2}
\vep^2 \|u_n(s_\vep)\|_{\X}^2 \leq C \|\rho_{0,n}\|_{\X^*}^2.
\end{equation}
Furthermore, recalling \eqref{energyineq2} with $u = u_n$ and $s = s_\vep$ and using the above facts, we deduce that
\begin{align*}
\MoveEqLeft{
\int^t_{2\vep} \left\| \partial_t (|u_n|^{(q-2)/2}u_n)(r) \right\|_{L^2(\Omega)}^2 \, \d r + \frac 12 \|u_n(t)\|_{\X}^2 
}\\
&\leq \frac 12 \|u_n(s_\vep)\|_{\X}^2 \stackrel{\eqref{ap:2}}\leq C \vep^{-2} \|\rho_{0,n}\|_{\X^*}^2
\end{align*}
for $0 < 2\vep < t < \infty$ and $n \in \N$. Thus since $(\rho_{0,n})$ is bounded in $\X^*$, repeating the same argument so far, we can obtain a \emph{weak} solution $u$ of the Cauchy--Dirichlet problem \eqref{pde}--\eqref{ic} for $\rho_0 \in \X^*$, and therefore, as we have seen in \S \ref{Ss:enid}, the function $t \mapsto \|u(t)\|_{L^q(\Omega)}^q$ is absolutely continuous in $(0,\infty)$ and \eqref{energyineq1} holds for a.e.~$t > 0$. Therefore we can further observe that
\begin{equation}
\left\{
\begin{aligned}
& u \in C_{\rm weak}((0,\infty);\X) \cap C((0,\infty);L^q(\Omega)),\\
& \partial_t (|u|^{q-2}u) \in C_{\rm weak}((0,\infty);\X^*),\\
& |u|^{(q-2)/2}u \in W^{1,2}(\tau,\infty;L^2(\Omega))
\end{aligned}
\right.\label{reg-weak}
\end{equation}
for $\tau > 0$. Thus we obtain

\begin{corollary}[Weak solutions along with the energy identity]\label{C:wex}
Let $1 < q < \infty$ and $0 < \theta \leq 1$. For each $\rho_0 \in \X^*$, the Cauchy--Dirichlet problem \eqref{pde}--\eqref{ic} admits a unique weak solution $u$ such that 
\begin{itemize}
 \item the regularity properties as in \eqref{reg-weak} hold\/{\rm ;}
 \item the function $t \mapsto \|u(t)\|_{L^q(\Omega)}^q$ is absolutely continuous in $(0,\infty)$ and the energy identity \eqref{energyineq1} holds for a.e.~$t > 0$.
\end{itemize}
\end{corollary}

Concerning $\theta = 1$, these facts can be obtained by combining the conclusion in the $H^{-1}$ framework and the uniqueness. An analogous argument are also available for $0 < \theta < 1$ (cf.~see~\cite[Theorem 2.3]{BII}).

Here we emphasize that the right-continuity of $t \mapsto u(t)$ in the strong topology of $\X$ has not yet been proved for general $\rho_0 \in \X^*$. When $q \leq 2^*_\theta$, the weak solution $u$ turns out to fulfill \eqref{energyineq2} for all $0 < s < t < \infty$ (as well as \eqref{en-ineq} for a.e.~$t > 0$), and therefore, $u$ belongs to $C_+((0,\infty);\X)$. Indeed, since $q \leq 2^*_\theta$, one finds that $(|u|^{q-2}u)(t)$ lies on $L^{(2^*_\theta)'}(\Omega)$ for any $t > 0$. Moreover, we observe that $u(t+\,\cdot\,)$ is a weak solution to the Cauchy--Dirichlet problem with the initial datum $u(t) \in \X$. Therefore, from the uniqueness of weak solutions, $u(t+\,\cdot\,)$ satisfies all the energy inequalities, and thus, the right-continuity of $u$ follows as in (ii) of Theorem \ref{T:main}. 

Concerning $q > 2^*_\theta$, employing Steklov's averaging technique (as in~\cite{BII}), we can still derive \eqref{energyineq2} for \emph{almost every} $s > 0$ and all $t \geq s$. Indeed, for each function $t \mapsto w(t)$, we define the \emph{Steklov average} $w^h$ of $w$ as
$$
w^h(t) := \frac 1 h \int^{t+h}_t w(r) \, \d r 
$$
for $t \geq 0$ and $h > 0$. Taking the Steklov average of both sides of \eqref{ee}, we have
$$
\partial_t \rho^h(t) + (-\Delta)^\theta u^h(t) = 0 \ \mbox{ in } \X^* \quad \mbox{ for } \ t \geq 0.
$$ 
Here we used the fact that 
$$
(\partial_t \rho)^h(t) = \frac{\rho(t+h)-\rho(t)}h = \partial_t \rho^h(t)
$$
for $t \geq 0$ and $h > 0$. We further note that $u^h \in W^{1,2}(0,T;\X)$ and $(-\Delta)^\theta u^h \in L^2(0,T;\X^*)$, and therefore, the function $t \mapsto \|u^h(t)\|_{\X}^2$ is absolutely continuous on $[0,\infty)$. Hence testing both sides by $\partial_t u^h(t) = (\partial_t u)^h(t) = [u(t+h)-u(t)]/h \in \X$, we deduce from \eqref{EleIne} that
$$
\frac 4 {qq'} \left\| \partial_t (|u|^{(q-2)/2}u)^h(t) \right\|_{L^2(\Omega)}^2 + \frac 12 \frac{\d}{\d t} \|u^h(t)\|_{\X}^2 \leq 0
$$
for a.e.~$t > 0$ and all $h > 0$. Integrating both sides over $(s,t)$, we infer that
$$
\frac 4 {qq'} \int^t_s \left\| \partial_t (|u|^{(q-2)/2}u)^h(r) \right\|_{L^2(\Omega)}^2 \, \d r + \frac 12 \|u^h(t)\|_{\X}^2 \leq \frac 12 \|u^h(s)\|_{\X}^2
$$
for $0 \leq s \leq t < \infty$ and $h > 0$. Since $u$ lies on $L^2(0,T;\X)$, thanks to Lebesgue's differentiation theorem, it follows that
$$
u^h(s) \to u(s) \quad \mbox{ strongly in } \X \ \mbox{ as } \ h \to 0_+
$$
for \emph{almost every} $s > 0$. Hence passing to the limit as $h \to 0_+$ and using $u \in C_{\rm weak}((0,\infty);\X)$, we can deduce that
$$
\frac 4 {qq'} \int^t_s \left\| \partial_t (|u|^{(q-2)/2}u)(r) \right\|_{L^2(\Omega)}^2 \, \d r + \frac 12 \|u(t)\|_{\X}^2 \leq \frac 12 \|u(s)\|_{\X}^2
$$
for \emph{almost every} $s > 0$ and all $t \geq s$. However, it may not yield $u \in C_+((0,\infty);\X)$, which will be used to bridge the defect of the (absolute) continuity of the energy $t \mapsto \|u(t)\|_{\X}^2$ in later sections. Such an observation may shed light on some utility of our scheme to construct energy solutions.
\end{remark}

\subsection{$L^1$ contraction and B\'enilan--Crandall estimate}\label{Ss:BC}

This subsection is devoted to providing a proof for (iii) of Theorem \ref{T:main}, which follows the same line as~\cite{Vazquez} for $\theta = 1$ and~\cite{PQRV12,KiLe11,BII} for $0 < \theta < 1$ but will also be exhibited for the completeness. 

Suppose that $u_{0,1}, u_{0,2}$ satisfy \eqref{hyp}. Let $\{u_{1,n}\}_{n=1}^N$ and $\{u_{2,n}\}_{n=1}^N$ be solutions to the discretized problem \eqref{En} with initial data $\rho_{0,1} := |u_{0,1}|^{q-2}u_{0,1}$ and $\rho_{0,2} := |u_{0,2}|^{q-2}u_{0,2}$, respectively. Then we see that
\begin{align*}
\MoveEqLeft{
\frac{ ( |u_{1,n+1}|^{q-2}u_{1,n+1} - |u_{2,n+1}|^{q-2}u_{2,n+1} ) - (|u_{1,n}|^{q-2}u_{1,n} - |u_{2,n}|^{q-2}u_{2,n} ) }{\tau}
}\\
&+ (-\Delta)^\theta (u_{1,n+1} - u_{2,n+1}) = 0 \ \text{ in } \X^*
\end{align*}
for $n = 0,1,\ldots,N-1$. Test it by $\eta_m(u_{1,n+1} - u_{2,n+1}) \in \X \cap L^\infty(\Omega)$, where $\eta_m$ is a Lipschitz continuous function given by
$$
\eta_m(s) = \begin{cases}
	     1 &\text{ if } \ s \geq 1/m,\\
	     ms &\text{ if } \ 0 \leq s < 1/m,\\
	     0 &\text{ if } \ s \leq 0
	    \end{cases}
$$
and it converges pointwisely to the sign-function given by
$$
\mathrm{sgn}(s) = \begin{cases}
		   1 &\text{ if } \ s > 0,\\
		   0 &\text{ if } \ s \leq 0
		  \end{cases}
$$
as $m \to \infty$. Then we note from the monotonicity of $\eta_m$ that
\begin{align*}
\left\langle (-\Delta)^\theta(u_{1,n+1} - u_{2,n+1}), \eta_m (u_{1,n+1} - u_{2,n+1}) \right\rangle_{\X} \geq 0.
\end{align*}
For simplicity, we write $u_{j,n+1}^{q-1} := |u_{j,n+1}|^{q-2} u_{j,n+1}$ and $u_{j,n}^{q-1} := |u_{j,n}|^{q-2} u_{j,n}$ for $j = 1,2$. Then we observe that
\begin{align*}
\MoveEqLeft{
\left\langle \frac{ (u_{1,n+1}^{q-1} - u_{2,n+1}^{q-1}) - (u_{1,n}^{q-1} - u_{2,n}^{q-1}) }{\tau}, \eta_m(u_{1,n+1} - u_{2,n+1}) \right\rangle_{\X}
}\\
&= \int_\Omega \frac{ (u_{1,n+1}^{q-1} - u_{2,n+1}^{q-1}) - (u_{1,n}^{q-1} - u_{2,n}^{q-1}) }{\tau} \eta_m(u_{1,n+1} - u_{2,n+1}) \, \d x\\
&\to \int_\Omega \frac{ (u_{1,n+1}^{q-1} - u_{2,n+1}^{q-1}) - (u_{1,n}^{q-1} - u_{2,n}^{q-1}) }{\tau} \mathrm{sgn}(u_{1,n+1} - u_{2,n+1}) \, \d x\\
&= \int_\Omega \frac{(u_{1,n+1}^{q-1} - u_{2,n+1}^{q-1}) - (u_{1,n}^{q-1} - u_{2,n}^{q-1})}{\tau} \mathrm{sgn}(u_{1,n+1}^{q-1} - u_{2,n+1}^{q-1}) \, \d x\\
&\geq \frac 1 \tau \int_\Omega (u_{1,n+1}^{q-1} - u_{2,n+1}^{q-1})_+ \, \d x
- \frac 1 \tau \int_\Omega (u_{1,n}^{q-1} - u_{2,n}^{q-1})_+ \, \d x,
\end{align*}
where $(s)_+ := \max\{0,s\}$ for $s \in \R$. Therefore it follows that
\begin{equation}\label{+-contr}
\int_\Omega (u_{1,n+1}^{q-1} - u_{2,n+1}^{q-1})_+ \, \d x \leq \int_\Omega ( |u_{0,1}|^{q-2}u_{0,1} - |u_{0,2}|^{q-2}u_{0,2} )_+ \, \d x.
\end{equation}
We can similarly prove that
$$
\int_\Omega (u_{1,n+1}^{q-1} - u_{2,n+1}^{q-1})_- \, \d x \leq \int_\Omega ( |u_{0,1}|^{q-2}u_{0,1} - |u_{0,2}|^{q-2}u_{0,2} )_- \, \d x,
$$
where $(s)_- := \max\{0,-s\}$ for $s \in \R$. Thus we obtain
$$
\int_\Omega | u_{1,n+1}^{q-1} - u_{2,n+1}^{q-1} | \, \d x \leq \int_\Omega \left| |u_{0,1}|^{q-2}u_{0,1} - |u_{0,2}|^{q-2}u_{0,2} \right| \, \d x.
$$
Let $\bar u_{j,\tau}$ be the piecewise constant interpolants of $\{u_{j,n}\}_{n=0}^N$ for $j=1,2$. Then
\begin{align*}
\MoveEqLeft{
\int_\Omega \left| (|\bar u_{1,\tau}|^{q-2}\bar u_{1,\tau})(t) - (|\bar u_{2,\tau}|^{q-2}\bar u_{2,\tau})(t) \right| \, \d x
}\\
&\leq \int_\Omega \left| |u_{0,1}|^{q-2}u_{0,1} - |u_{0,2}|^{q-2}u_{0,2} \right| \, \d x  \quad \text{ for } \ t \in [0,T],
\end{align*}
which along with the convergence of $(\bar u_{j,\tau})$ (see \S \ref{Ss:ex}) and the arbitrariness of $T > 0$ implies \eqref{L1-contr}. Moreover, repeating the same argument as in \S \ref{Ss:gen}, we can prove \eqref{L1-contr} for $u_{0,1}$, $u_{0,2}$ satisfying the assumptions as in (ii) of Theorem \ref{T:main}, that is, \eqref{hyp2}. Similarly, one can prove from \eqref{+-contr} the comparison principle: if $u_{0,1} \leq u_{0,2}$ a.e.~in $\Omega$, then $u_1 \leq u_2$ a.e.~in $\Omega \times (0,\infty)$.

The B\'enilan--Crandall estimate can be proved by a scaling argument as in~\cite{Vazquez}. For the completeness, we shall give a proof. We first note that, if $u = u(x,t)$ solves \eqref{pde} and \eqref{bc}, then so does the scaled functions,
$$
u_\lambda = u_\lambda(x,t) := \lambda^{-1/(q-2)} u(x,\lambda t) \quad \text{ for } \ \lambda > 0.
$$
Fix $t > 0$. Let $h \in \R$ be small enough and take $\lambda > 0$ such that $t+h = \lambda t$. Then we observe that
\begin{align*}
\MoveEqLeft{
\left\| (|u|^{q-2}u)(t+h) - (|u|^{q-2}u)(t) \right\|_{L^1(\Omega)}
}\nonumber\\
&= \left\| \lambda^{(q-1)/(q-2)} (|u_\lambda|^{q-2}u_\lambda)(t) - (|u|^{q-2}u)(t) \right\|_{L^1(\Omega)} \nonumber\\
&\leq \left| \lambda^{(q-1)/(q-2)} - 1 \right| \left\| (|u_\lambda|^{q-2}u_\lambda)(t) \right\|_{L^1(\Omega)} \nonumber\\
&\quad + \left\| (|u_\lambda|^{q-2}u_\lambda)(t) - (|u|^{q-2}u)(t) \right\|_{L^1(\Omega)} \nonumber\\
\prf{&= \left| 1 - \lambda^{-(q-1)/(q-2)} \right| \| (|u|^{q-2}u)(\lambda t) \|_{L^1(\Omega)} \nonumber\\
&\quad + \left\| (|u_\lambda|^{q-2}u_\lambda)(t) - (|u|^{q-2}u)(t) \right\|_{L^1(\Omega)}\nonumber\\}
&\stackrel{\eqref{L1-contr}}\leq \left| 1 - \lambda^{-(q-1)/(q-2)} \right| \left\| |u_0|^{q-2} u_0 \right\|_{L^1(\Omega)} \nonumber\\
&\quad + \left\| (|u_\lambda|^{q-2}u_\lambda)(0) - (|u|^{q-2}u)(0) \right\|_{L^1(\Omega)}\nonumber\prf{\\
&= 2 \left| 1 - \lambda^{-(q-1)/(q-2)} \right| \left\| |u_0|^{q-2} u_0 \right\|_{L^1(\Omega)}\nonumber}.
\end{align*}
Since $\lambda = 1 + h/t$, we have
\begin{align}
\MoveEqLeft{
\left\| (|u|^{q-2}u)(t+h) - (|u|^{q-2}u)(t) \right\|_{L^1(\Omega)}
}\nonumber\\
&\leq 2 \left| 1 - \lambda^{-(q-1)/(q-2)} \right| \left\| |u_0|^{q-2}u_0 \right\|_{L^1(\Omega)}\nonumber\\
&= 2 \frac{q-1}{|q-2|} \frac{\| |u_0|^{q-2}u_0 \|_{L^1(\Omega)}} t |h| + O(|h|^2) \ \mbox{ as } \ h \to 0,\label{BV-L1}
\end{align}
which also yields
\begin{align}
\MoveEqLeft{
\left\| \frac{(|u|^{q-2}u)(t+h) - (|u|^{q-2}u)(t)} h \right\|_{L^1(\Omega)}
}\nonumber\\
&\leq 2 \frac{q-1}{|q-2|} \frac{\||u_0|^{q-2}u_0\|_{L^1(\Omega)}} t + O(|h|) \ \mbox{ as } \ h \to 0.\label{BC0}
\end{align}
In case $q > 2$, we have already seen in \S \ref{Ss:ex} that 
$$
|u|^{q-2}u \in W^{1,2}_{\rm loc}([0,\infty);L^{q'}(\Omega)) \subset W^{1,1}_{\rm loc}((0,\infty);L^1(\Omega)),
$$
and therefore, one can obtain
$$
\frac{(|u|^{q-2}u)(t+h) - (|u|^{q-2}u)(t)} h \to \partial_t (|u|^{q-2}u)(t) \quad \text{ strongly in } L^1(\Omega)
$$
for a.e.~$t > 0$, up to a subsequence, as $h \to 0$ (see, e.g.,~\cite[Corollary 1.4.39]{CaHa}). Hence we infer from \eqref{BC0} that
\begin{align*}
\left\| \partial_t (|u|^{q-2}u)(t) \right\|_{L^1(\Omega)}
\leq 2 \frac{q-1}{|q-2|} \frac{\| |u_0|^{q-2}u_0 \|_{L^1(\Omega)}} t
\end{align*}
for a.e.~$t > 0$. Thus \eqref{BC} follows. In case $1 < q < 2$, first of all, we need to prove $\rho = |u|^{q-2}u \in W^{1,1}_{\rm loc}((0,\infty);L^1(\Omega))$. To this end, we note that
$$
\rho = |u|^{q-2}u = \int^{w}_0 p(r) \, \d r, \quad w = |u|^{(q-2)/2}u,
$$
where $p \in L^1_{\rm loc}(\R)$ is given by $p(r) = \frac 2 q (q-1) |r|^{(q-2)/q}$ for $r \in \R$, and recall the following:
\begin{proposition}[See~{\cite[Theorem 1.1]{BG95}}]\label{P:BV}
Let $(a,b)$ be a nonempty open interval. If $w \in W^{1,1}(a,b;L^1(\Omega))$, $p \in L^1_{\rm loc}(\R)$ and
$$
\rho = \int^w_0 p(r) \, \d r \in BV(a,b;L^1(\Omega)),
$$
then $\rho$ belongs to $W^{1,1}(a,b;L^1(\Omega))$\prf{ and $\partial_t \rho(\cdot,t) = p(w(\cdot,t)) \partial_t w(\cdot,t)$ for a.e.~$t \in (a,b)$}. Here  $BV(a,b;L^1(\Omega))$ stands for the class of functions $f \in L^1(a,b; L^1(\Omega))$ of bounded variation.
\end{proposition}
We can also observe from \eqref{BV-L1} that $\rho \in BV(a,b;L^1(\Omega))$ for any $0 < a < b < \infty$; indeed, we see that
\begin{align*}
\liminf_{h \to 0} \frac 1h \int^{b-h}_a \left\| (|u|^{q-2}u)(t+h) - (|u|^{q-2}u)(t) \right\|_{L^1(\Omega)} \, \d t \leq C
\end{align*}
(see, e.g.,~\cite[Appendix]{HB1}). Hence recalling $w = |u|^{(q-2)/2}u \in W^{1,2}(0,\infty;L^2(\Omega)) \subset W^{1,1}_{\rm loc}((0,\infty);L^1(\Omega))$ and employing Proposition \ref{P:BV}, one can conclude that
$$
\rho = |u|^{q-2}u \in W^{1,1}_{\rm loc}((0,\infty);L^1(\Omega)).
$$
Therefore \eqref{BC} can also be verified as before.

In addition, if $\rho_0 = |u_0|^{q-2}u_0 \geq 0$ a.e.~in $\Omega$, then $u(t) \geq 0$ a.e.~in $\Omega$ for $t \geq 0$ by virtue of the comparison principle. The same computation as above yields
\begin{align*}
\MoveEqLeft{
\frac{(|u|^{q-2}u)(t+h)-(|u|^{q-2}u)(t)}h
}\\
&= \frac{(|u_\lambda|^{q-2}u_\lambda)(t)-(|u|^{q-2}u)(t)}h
+ \frac{q-1}{q-2} \frac{(|u_\lambda|^{q-2}u_\lambda)(t)}t + O(h)
\end{align*}
a.e.~in $\Omega$ for $t > 0$ as $h \to 0$. Here noting that $h > 0$ if and only if $\lambda > 1$ and recalling the comparison principle, in case $q > 2$, we deduce that $[(|u_\lambda|^{q-2}u_\lambda)(t)-(|u|^{q-2}u)(t)]/h \leq 0$ a.e.~in $\Omega$ from the fact that $u_\lambda(0) \leq u(0)$ a.e.~in $\Omega$ for $\lambda \geq 1$. Therefore the passage to the limit as $h \to 0_+$ (i.e., $\lambda \to 1_+$) implies \eqref{BC-pnt}. In case $1 < q < 2$, we can obtain an inverse inequality as in \eqref{BC-pnt+}. 

\prf{
\begin{remark}
In case $1 < q < 2$, the integrability of $\partial_t (|u|^{q-2}u)(t)$ is more delicate (cf.~see~\cite[Theorem 1.1, Lemma 2.1]{BeGa}).
\end{remark}
}

\section{Monotonicity of Rayleigh quotient and decay estimates}\label{S:Rayleigh}

This section is devoted to proving Theorem \ref{T:Rayleigh}. Let $u = u(x,t)$ be an energy solution and set $\mathsf{R}(t) := R(u(t))$ for $t \geq 0$ as far as $u(t) \neq 0$ (see below for the definition of the Rayleigh quotient $R(\cdot)$). By a simple calculation, we shall derive from \eqref{energyineq1} and \eqref{energyineq2} that $\d \mathsf{R}/\d t \leq 0$ almost everywhere, from which, however, the monotonicity of $\mathsf{R}$ does not follow immediately, since the absolute continuity of $\mathsf{R}$ has not been guaranteed. In order to fill the gap, we shall exploit the theory for functions of bounded variation (see, e.g.,~\cite{AFP}). Moreover, as a corollary, we shall also derive decay and extinction estimates for energy solutions to the Cauchy--Dirichlet problem \eqref{pde}--\eqref{ic} (see Corollary \ref{C:decay}).

\begin{proof}[Proof of Theorem {\rm \ref{T:Rayleigh}}]
Let $u_0\in \X \cap L^q(\Omega)$ satisfy \eqref{ini-hyp+} and $u_0\not\equiv 0$. Let $u$ be the unique (global-in-time) energy solution to the Cauchy--Dirichlet problem \eqref{pde}--\eqref{ic} with the initial datum $\rho_0 = |u_0|^{q-2}u_0$. Define the \emph{Rayleigh quotient} by
$$
R(w) := \frac{\|w\|_{\X}^2}{\|w\|_{L^q(\Omega)}^2} \quad \text{ for } \ w \in ( \X \cap L^q(\Omega) ) \setminus \{0\}.
$$
Since $u_0 \not \equiv 0$ and the function $t \mapsto \|u(t)\|_{L^q(\Omega)}$ is nonincreasing and (absolutely) continuous on $[0,\infty)$, the map $t \mapsto {\mathsf R}(t) := R(u(t))$ can be defined on the interval $[0,t_*)$ where $t_*$ is given by
\begin{equation}\label{t*}
t_* := \sup \left\{ T > 0 \colon \|u(t)\|_{L^q(\Omega)} > 0 \ \text{ for all } \ t \in[0,T] \right\} \in (0,\infty].
\end{equation}
When $t_* < \infty$, the solution $u = u(x,t)$ is said to \emph{extinct in finite time} and $t_*$ is called the \emph{extinction time} of $u$. 

Let $0 < T < t_*$ be fixed. The function $t \mapsto \|u(t)\|_{L^q(\Omega)}^q$ is absolutely continuous and uniformly away from $0$ on $[0,T]$ (indeed, it is nonincreasing and $u(T) \not \equiv 0$). Moreover, due to \eqref{energyineq1} and \eqref{energyineq2} with $s = 0$, the derivative $(\d/\d t) \|u(t)\|_{L^q(\Omega)}^q$ is uniformly bounded a.e.~in $(0,T)$, and hence, the function $g(t) := 1/\|u(t)\|_{L^q(\Omega)}^2$ is Lipschitz continuous on $[0,T]$. Thanks to (ii) of Theorem \ref{T:main}, the function $f(t) := \|u(t)\|_{\X}^2$ is nonincreasing on $[0,T]$ (and hence, of bounded variation in $(0,T)$) and right-continuous on $[0,T]$. Therefore ${\mathsf R} = fg$ is also right-continuous on $[0,T]$ and of bounded variation in $(0,T)$ (see~\cite[Proposition 3.2]{AFP}). Therefore $\mathsf{R}$ is differentiable a.e.~in $(0,T)$. 

We now claim that $\d {\mathsf R}/\d t\leq 0$ a.e.~in $(0,T)$. Indeed, we have
$$
\frac{\d {\mathsf R}}{\d t}(t) = \frac{ \|u(t)\|_{L^q(\Omega)}^2 \frac{\d}{\d t} \|u(t)\|_{\X}^2 - \|u(t)\|_{\X}^2 \frac{\d}{\d t} \|u(t)\|_{L^q(\Omega)}^2 }{ \|u(t)\|_{L^q(\Omega)}^4 }.
$$
Recall from \eqref{energyineq1} and \eqref{en-ineq} that
\begin{align*}
\|u(t)\|_{\X}^2 &= -\frac{1}{q'} \frac{\d}{\d t} \|u(t)\|_{L^q(\Omega)}^q,\\
\frac{\d}{\d t} \|u(t)\|_{\X}^2 &\leq -\frac{8}{qq'} \left\| \partial_t (|u|^{(q-2)/2}u)(t) \right\|_{L^2(\Omega)}^2
\end{align*}
for a.e.~$t \in (0,T)$. Moreover, by a simple computation, since $\|u(t)\|_{L^q(\Omega)} > 0$ for $t \in [0,T]$, we see that
$$
\frac{\d}{\d t} \|u(t)\|_{L^q(\Omega)}^2 = \frac{2}{q} \|u(t)\|_{L^q(\Omega)}^{2-q} \frac{\d}{\d t} \|u(t)\|_{L^q(\Omega)}^q
$$
for a.e.~$t \in (0,T)$. Combining all these facts, we deduce that
\begin{align*}
\lefteqn{
\frac{\d {\mathsf R}}{\d t}(t)
}\\
&\leq \frac{ -\frac{8}{qq'} \| \partial_t (|u|^{(q-2)/2}u)(t) \|_{L^2(\Omega)}^2 \|u(t)\|_{L^q(\Omega)}^2 + \frac{2}{qq'} \|u(t)\|_{L^q(\Omega)}^{2-q} ( \frac{\d}{\d t} \|u(t)\|_{L^q(\Omega)}^q )^2 }{ \|u(t)\|_{L^q(\Omega)}^4 }\\
&= \frac{ -\frac{8}{qq'} \|\partial_t (|u|^{(q-2)/2}u)(t)\|_{L^2(\Omega)}^2 \|u(t)\|_{L^q(\Omega)}^q + \frac{2}{qq'} (\frac{\d}{\d t}\|u(t)\|_{L^q(\Omega)}^q)^2 }{ \|u(t)\|_{L^q(\Omega)}^{q+2} }
\end{align*}
for a.e.~$t \in (0,T)$. Then noting that
\begin{align*}
\frac{\d}{\d t} \|u(t)\|_{L^q(\Omega)}^q
&= \frac{\d}{\d t} \left\| (|u|^{(q-2)/2}u)(t) \right\|_{L^2(\Omega)}^2\\
&= 2 ( \partial_t (|u|^{(q-2)/2}u)(t), (|u|^{(q-2)/2}u)(t) )_{L^2(\Omega)}
\end{align*}
for a.e.~$t \in (0,T)$, we obtain
\begin{align}
\frac{\d {\mathsf R}}{\d t}(t) 
&\leq -\frac{8}{qq'} \frac{ \| \partial_t (|u|^{(q-2)/2}u)(t) \|_{L^2(\Omega)}^2 \| (|u|^{(q-2)/2}u)(t) \|_{L^2(\Omega)}^2 }{ \|u(t)\|_{L^q(\Omega)}^{q+2} }\nonumber\\
&\quad + \frac{8}{qq'}\frac{ ( \partial_t (|u|^{(q-2)/2}u)(t), (|u|^{(q-2)/2}u)(t) )_{L^2(\Omega)}^2 }{ \|u(t)\|_{L^q(\Omega)}^{q+2} } \leq 0 \label{dR-neg}
\end{align}
for a.e.~$t\in(0,T)$. Here the last inequality follows from Schwarz's inequality.

We next prove the monotonicity of the Rayleigh quotient $t \mapsto \mathsf{R}(t)$ on $[0,\infty)$, which is still unclear, since $\mathsf{R}$ is of bounded variation only and may not be absolutely continuous. Since $f(t) = \|u(t)\|_{\X}^2$ is of bounded variation in $(0,T)$ and right-continuous on $[0,T]$, thanks to~\cite[Theorem 3.28]{AFP} and Lebesgue's decomposition theorem, it follows that, for $0 < s < t < T$,
\begin{align*}
 f(t) - f(s) = \mathrm{D}f((s,t]) = \int^t_s \dfrac{\d f}{\d r}(r) \, \d r + (\mathrm{D}f)_s((s,t]),
\end{align*}
where $\mathrm{D}f$ and $(\mathrm{D}f)_s$ denote the distributional derivative of $f$ represented as a finite Radon measure in $(0,T)$ and its singular part (with respect to the one-dimensional Lebesgue measure $\mathcal{L}^1$), respectively. Since $f$ is nonincreasing (see, e.g.,~\cite[Corollary 3.29]{AFP}), we see that
$$
f(t) - f(s) \leq \int^t_s \dfrac{\d f}{\d r} (r) \, \d r \quad \text{ for } \ 0 < s < t < T,
$$
which implies $(\mathrm{D}f)_s((s,t]) \leq 0$ for any $(s,t] \subset (0,T)$. Since $f$ is of bounded variation and $g$ is Lipschitz continuous in $(0,T)$, by virtue of~\cite[Proposition 3.2]{AFP} and Lebesgue's decomposition theorem, we observe that
\begin{equation}\label{DR}
\mathrm{D}{\mathsf R} = g \mathrm{D}f + f \frac{\d g}{\d t}\mathcal{L}^1
= g \dfrac{\d f}{\d t} \mathcal{L}^1 + g (\mathrm{D}f)_s + f \frac{\d g}{\d t}\mathcal{L}^1,
\end{equation}
where $\mathcal{L}^1$ denotes the one-dimensional Lebesgue measure. Therefore as ${\mathsf R}$ is of bounded variation in $(0,T)$ and right-continuous on $[0,T]$, we infer from~\cite[Theorem 3.28]{AFP} and \eqref{DR} that
\begin{align*}
{\mathsf R}(t) - {\mathsf R}(s) 
&= \mathrm{D}{\mathsf R}((s,t])\\
&= \int^t_s g(r) \dfrac{\d f}{\d r}(r) \, \d r + [g (\mathrm{D}f)_s]((s,t]) + \int^t_s f(r) \dfrac{\d g}{\d r}(r) \, \d r\\
&= \int^t_s \dfrac{\d {\mathsf R}}{\d r} (r) \, \d r + \int_{(s,t]}g \, \d(\mathrm{D}f)_s \stackrel{\eqref{dR-neg}} \leq 0
\end{align*}
for $0 < s < t < T$. Indeed, $g$ is (Lipschitz) continuous and positive on $[0,T]$ and $(\mathrm{D}f)_s$ is a nonpositive finite Radon measure in $(0,T)$. Thus the function $t \mapsto {\mathsf R}(t)$ is nonincreasing in $(0,T)$, and hence, it is so on $[0,t_*)$ from the right-continuity of ${\mathsf R}$ on $[0,T]$ and the arbitrariness of $T < t_*$. This completes the proof of Theorem \ref{T:Rayleigh}.
\end{proof}

Thanks to Theorems \ref{T:main} and \ref{T:Rayleigh}, we can prove Corollary \ref{C:decay}.

\begin{proof}[Proof of Corollary {\rm \ref{C:decay}}]
We assume \eqref{hyp-decay} (which also implies $|u_0|^{q-2}u_0 \in L^{(2_\theta^*)'}(\Omega)$). Let $u = u(x,t)$ be the energy solution of \eqref{pde}--\eqref{ic} with initial data $\rho_0 = |u_0|^{q-2}u_0$ and let $t_* \in (0,\infty]$ be the extinction time (see \eqref{t*}). Recalling the energy identity \eqref{energyineq1} and using the Sobolev--Poincar\'e inequality \eqref{SP}, we find that
\begin{equation}\label{ei+SP}
\frac{1}{q'}\frac{\d}{\d t}\|u(t)\|_{L^q(\Omega)}^q + C_q^{-2}\|u(t)\|_{L^q(\Omega)}^2\leq 0
\end{equation}
for a.e.~$t > 0$. In case $1 < q < 2$, we obtain
$$
\|u(t)\|_{L^q(\Omega)} \leq \left( \|u_0\|_{L^q(\Omega)}^{-(2-q)} + \frac{2-q}{q-1} C_q^{-2} t \right)^{-1/(2-q)}
$$
for all $t \geq 0$; whence $\|u(t)\|_{L^q(\Omega)}$ decays algebraically as $t \to \infty$. In case $2 < q \leq 2_\theta^*$, it follows from \eqref{ei+SP} that
$$
\|u(t)\|_{L^q(\Omega)} \leq \left( \|u_0\|_{L^q(\Omega)}^{q-2} - \frac{q-2}{q-1} C_q^{-2} t \right)_+^{1/(q-2)}
$$
for all $t \geq 0$. Therefore $u$ vanishes at a finite time, and therefore, $t_*$ is finite. Moreover, we see that
$$
t_*\leq \frac{q-1}{q-2} C_q^2 \|u_0\|_{L^q(\Omega)}^{q-2} < \infty.
$$
On the other hand, \eqref{ei+SP} also yields
\begin{equation}\label{ODEbis}
\frac{\d}{\d t} \|u(t)\|_{L^q(\Omega)}^{q-2} \leq -\frac{q-2}{q-1} C_q^{-2}
\end{equation}
for a.e.~$t \in (0,t_*)$, since $\|u(t)\|_{L^q(\Omega)} > 0$ for all $t \in [0,t_*)$. Integrating both sides of \eqref{ODEbis} over $(t,t_*)$, we infer that
$$
\|u(t)\|_{L^q(\Omega)} \geq \left( \frac{q-2}{q-1} C_q^{-2} \right)^{1/(q-2)} (t_*-t)_+^{1/(q-2)}
$$
for all $t \in [0,t_*)$.

Next, recalling the energy identity \eqref{energyineq1} and using the nonincrease of the Rayleigh quotient (see Theorem \ref{T:Rayleigh}), namely,
\begin{equation}
\|u(t)\|_{\X}^2 = R(u(t)) \|u(t)\|_{L^q(\Omega)}^2 \leq R(u_0) \|u(t)\|_{L^q(\Omega)}^2 \label{rev-ineq}
\end{equation}
for $t \in [0,t_*)$, we have
\begin{equation}\label{ei+R}
\frac{1}{q'} \frac{\d}{\d t} \|u(t)\|_{L^q(\Omega)}^q + R(u_0) \|u(t)\|_{L^q(\Omega)}^2 \geq 0
\end{equation}
for a.e.~$t > 0$. In case $1 < q < 2$, we derive that
$$
\frac{\d}{\d t} \|u(t)\|_{L^q(\Omega)}^{-(2-q)} \leq \frac{2-q}{q-1} R(u_0)
$$
for a.e.~$t \in (0,t_*)$; whence it follows that
$$
\|u(t)\|_{L^q(\Omega)} \geq \left( \|u_0\|_{L^q(\Omega)}^{-(2-q)} + \frac{2-q}{q-1} R(u_0)t \right)^{-1/(2-q)}
$$
for all $t \in [0,t_*)$. Therefore we see that $t_* = \infty$, since the right-hand side of the above is positive for any $t\geq 0$. Consequently, there exists a constant $c > 0$ such that
$$
c(1+t)^{-1/(2-q)} \leq \|u(t)\|_{L^q(\Omega)} \leq c^{-1}(1+t)^{-1/(2-q)}
$$
for all $t\geq 0$. In case $2 < q \leq 2_\theta^*$, one can derive from \eqref{ei+R} that
$$
\frac{\d}{\d t} \|u(t)\|_{L^q(\Omega)}^{q-2} \geq -\frac{q-2}{q-1} R(u_0)
$$
for a.e.~$t \in (0,t_*)$. Thus it leads us to obtain
$$
\|u(t)\|_{L^q(\Omega)} \leq \left( \frac{q-2}{q-1} R(u_0) \right)^{1/(q-2)} (t_*-t)_+^{1/(q-2)}
$$
for all $t \geq 0$. In particular, substituting $t = 0$ to the above, we find that
$$
t_* \geq \frac{q-1}{q-2} R(u_0)^{-1} \|u_0\|_{L^q(\Omega)}^{q-2}.
$$
Thus \eqref{asymp1} and \eqref{asymp2} have been verified. Finally, using \eqref{SP} and \eqref{rev-ineq}, we can derive \eqref{asymp1} and \eqref{asymp2} with different constants and $\|\cdot\|_{L^q(\Omega)}$ replaced by $\|\cdot\|_{\X}$. Thus we have completed the proof of Corollary \ref{C:decay}.
\end{proof}

\section{Convergence to asymptotic profiles}\label{S:conv}

In this section, we shall prove Theorems \ref{T:ap} and \ref{T:conv}. Throughout this section, we assume \eqref{hyp-ap}. In \S \ref{Ss:rescale}, we shall discuss regularity and energy inequalities for the rescaled solution $v = v(x,s)$ defined by \eqref{def1}, \eqref{def2}. Subsection \ref{Ss:quasiconv} is concerned with a proof for Theorem \ref{T:ap}. In \S \ref{Ss:LS}, we shall provide a variant of the so-called {\L}ojasiewicz--Simon gradient inequality, which will be used in \S \ref{Ss:conv} to prove Theorem \ref{T:conv} for the (fractional) fast diffusion case, i.e., $2 < q < 2^*_\theta$.

\subsection{Rescaled energy solutions}\label{Ss:rescale}

Let $u = u(x,t)$ be the energy solution to \eqref{pde}--\eqref{ic} with the initial datum $u_0$ satisfying \eqref{hyp-decay} and let $v = v(x,s)$ be the rescaled function defined by \eqref{def1}, \eqref{def2}. From Theorem \ref{T:main}, it is straightforward to see that 
\begin{align*}
v &\in C_{\text{weak}}([0,\infty),\X) \cap C_+([0,\infty),\X) \cap C([0,\infty),L^q(\Omega)),\\
|v|^{q-2}v &\in C([0,\infty),L^{q'}(\Omega)),\\
\partial_s (|v|^{q-2}v) &\in C_{\rm weak}([0,\infty);\X^*) \cap C_+([0,\infty),\X^*),\\
|v|^{(q-2)/2}v &\in W^{1,2}(0,\infty;L^2(\Omega)),
\end{align*}
and moreover, by means of Definition \ref{D:sol}, $v$ turns out to be the \emph{energy solution} to the Cauchy--Dirichlet problem \eqref{eq:1.6}--\eqref{eq:1.8} with the initial datum $v_0$ defined by \eqref{v0}. Furthermore, from the fact that $v \in C_+([0,\infty);\X)$, the initial condition \eqref{eq:1.8} can also be replaced with
$$
v(\cdot,0) = v_0 \quad \mbox{ in } \Omega.
$$
By virtue of Corollary \ref{C:decay}, there exist positive constants $c,C$ such that
\begin{equation}\label{v-bdd}
0 < c \leq \|v(s)\|_{L^q(\Omega)} \leq C_q \|v(s)\|_{\X} \leq C < \infty 
\end{equation}
for $s \geq 0$. Furthermore, it follows from \eqref{energyineq1}--\eqref{en-ineq} that

\begin{lemma}[Energy inequalities for the rescaled solution]\label{L:v-enin}
\begin{enumerate}
\item[\rm (i)] The function $s \mapsto \|v(s)\|_{\X}^2$ is of locally bounded variation on $[0,\infty)$, and moreover, the function $s \mapsto \|v(s)\|_{L^q(\Omega)}^q$ is absolutely continuous on $[0,\infty)$. Define a functional $J : \X \to \R$ by
$$
J(w) = \frac{1}{2} \|w\|_{\X}^2 - \frac{\lambda_q}{q} \|w\|_{L^q(\Omega)}^q \quad \text{ for } \ w \in \X,
$$
where $\lambda_q := (q-1)/|q-2| > 0$. Then the function $s \mapsto J(v(s))$ is of locally bounded variation on $[0,\infty)$ and differentiable almost everywhere. It also holds that
\begin{equation}\label{estJ}
\frac{4}{qq'} \left\| \partial_s (|v|^{(q-2)/2}v)(s) \right\|_{L^2(\Omega)}^2 + \frac{\d}{\d s} J(v(s)) \leq 0
\end{equation}
for a.e.~$s > 0$. Furthermore, $J(v(\cdot))$ is nonincreasing on $[0,\infty)$.
\item[\rm (ii)] The function $s \mapsto \|(|v|^{q-2}v)(s)\|_{\X^*}^2$ is also absolutely continuous on $[0,\infty)$. Define a functional $K : L^q(\Omega) \to \R$ by 
$$
K(w) = \frac{1}{q'} \|w\|_{L^q(\Omega)}^q - \frac{\lambda_q}{2} \left\| |w|^{q-2}w \right \|_{\X^*}^2 \quad \text{ for } \ w \in L^q(\Omega).
$$
Then the function $s \mapsto K(v(s))$ is absolutely continuous and nonincreasing on $[0,\infty)$, and moreover, it holds that
\begin{equation}\label{estK}
\left\| \partial_s (|v|^{q-2}v)(s) \right\|_{\X^*}^2 + \frac{\d}{\d s} K(v(s)) = 0
\end{equation}
for a.e.~$s > 0$\/{\rm ;} whence it follows that
\begin{equation}\label{dv-bdd}
\int_0^\infty \left\| \partial_s (|v|^{q-2}v)(s) \right\|_{\X^*}^2 \,\d s < \infty.
\end{equation}
\end{enumerate}
\end{lemma}

\begin{proof}
We first prove (i). It is clear from Theorem \ref{T:main} along with \eqref{def1}, \eqref{def2} that the function $s \mapsto \|v(s)\|_{\X}^2$ is of locally bounded variation on $[0,\infty)$ and the function $s \mapsto \|v(s)\|_{L^q(\Omega)}^q$ is absolutely continuous on $[0,\infty)$. Thus $J(v(\cdot))$ turns out to be of locally bounded variation on $[0,\infty)$, and hence, it is differentiable a.e.~in $(0,\infty)$. Estimate \eqref{estJ} follows from \eqref{en-ineq}, \eqref{def1}, \eqref{def2} with a simple computation. 

We next verify that the function $s \mapsto J(v(s))$ is nonincreasing on $[0,\infty)$ as in Section \ref{S:Rayleigh}. Let $T > 0$ be arbitrarily fixed and define functions $\phi,\psi : [0,T] \to \R$ of bounded variation by 
$$
\phi(s) = J(v(s)), \quad \psi(s) = \frac 12 \|v(s)\|_{\X}^2
$$
for $s \in [0,T]$. Then due to Lebesgue's decomposition theorem, we observe that
\begin{equation}\label{Dphi-decomp}
\mathrm{D} \phi = (\mathrm{D} \psi)_s + \frac{\d \phi}{\d s} \mathcal{L}^1,
\end{equation}
where $\mathrm{D} \phi$ and $(\mathrm{D} \psi)_s$ denote the distributional derivative of $\phi$ represented as a finite Radon measure in $(0,T)$ and the singular part (i.e., $(\mathrm{D} \psi)_s \perp \mathcal{L}^1$) of the distributional derivative $\mathrm{D} \psi$ of $\psi$, respectively, and $\mathcal{L}^1$ denotes the one-dimensional Lebesgue measure. As in the proof of Theorem \ref{T:Rayleigh}, since $\psi$ is the product of a smooth positive function and a nonincreasing right-continuous function, we can also show that $(\mathrm{D} \psi)_s((s_1,s_2]) \leq 0$ for any $(s_1,s_2] \subset (0,T)$.\prf{\footnote{\UUU Indeed, we can write $\psi(s) = g(s) h(s)$ for a positive smooth function $g$ and a nonincreasing right-continuous function $h$. Since $\psi = gh$ is right-continuous on $[0,T)$ and of bounded variation in $(0,T)$, we have
$$
\mathrm{D} \psi = \frac{\d \psi}{\d s} \mathcal{L}^1 + g (\mathrm{D} h)_s,
$$
which yields $(\mathrm{D} \psi)_s = g (\mathrm{D} h)_s$. Moreover, as $h$ is nonincreasing, one can check that $(\mathrm{D}h)_s ((s_1,s_2]) \leq 0$ for any $(s_1,s_2] \subset (0,T)$. Therefore from the positivity of $g$, we conclude that 
$$
(\mathrm{D} \psi)_s((s_1,s_2]) = \int_{(s_1,s_2]} g \,\d (\mathrm{D} h)_s \leq 0
$$
for any $(s_1,s_2] \subset (0,T)$.}} As $\phi = J(v(\cdot))$ is right-continuous on $[0,T]$ and of bounded variation in $(0,T)$, we infer from \eqref{estJ}, \eqref{Dphi-decomp} and~\cite[Theorem 3.28]{AFP} that
\begin{align*}
\phi(s_2) - \phi(s_1)
&= (\mathrm{D} \phi)((s_1,s_2])\\
&= (\mathrm{D} \psi)_s((s_1,s_2]) + \int_{s_1}^{s_2} \frac{\d\phi}{\d s}(s) \,\d s\\
&\leq - \frac{4}{qq'} \int_{s_1}^{s_2} \left\| \partial_s ( |v|^{(q-2)/2}v )(s) \right\|_{L^2(\Omega)}^2 \, \d s \leq 0
\end{align*}
for any $0 < s_1 \leq s_2 < T$. From the arbitrariness of $T$ and the right-continuity of $\phi$ on $[0,\infty)$, the function $J(v(\cdot))$ turns out to be nonincreasing on $[0,\infty)$.

We next prove (ii). Since $|v|^{q-2}v \in W^{1,2}_{\rm loc}([0,\infty);\X^*)$ and $(-\Delta)^{-\theta}$ is a Riesz map from $\X^*$ to $\X$, the function $s \mapsto \|(|v|^{q-2}v)(s)\|_{\X^*}^2 = \| (-\Delta)^{-\theta} (|v|^{q-2}v)(s)\|_{\X}^2$ is absolutely continuous on $[0,\infty)$ such that
$$
\frac 12 \frac{\d}{\d s} \left\| (|v|^{q-2}v)(s) \right\|_{\X^*}^2 = \left\langle (|v|^{q-2}v)(s), (-\Delta)^{-\theta} \partial_s (|v|^{q-2}v)(s) \right\rangle_{\X}
$$
for a.e.~$s \in (0,\infty)$. Accordingly, $K(v(\cdot))$ is absolutely continuous on $[0,\infty)$ and differentiable a.e.~in $(0,\infty)$. Moreover, \eqref{estK} follows from the direct computation of testing \eqref{eq:1.6} with $(-\Delta)^{-\theta} \partial_s (|v|^{q-2}v)(s) \in \X$ (see also Lemma \ref{L:chainrule}). Moreover, by virtue of \eqref{v-bdd}, we see that $K(v(s))$ is bounded from below for $s \geq 0$. Therefore we conclude that
$$
\int_0^\infty \left\| \partial_s (|v|^{q-2}v)(s) \right\|_{\X^*}^2 \, \d s < \infty.
$$
This completes the proof.
\end{proof}

\subsection{Quasi-convergence to asymptotic profiles}\label{Ss:quasiconv}

In this subsection, we shall prove Theorem \ref{T:ap}, which can be proved as in the classical case, i.e., $\theta = 1$ (see, e.g.,~\cite{BH80, AK13}) but we shall briefly give a proof for the completeness. 

\begin{proof}[Proof of Theorem {\rm \ref{T:ap}}]
Let $(s_n)$ be a sequence in $(0,\infty)$ such that $s_n \to \infty$. Due to \eqref{dv-bdd}, for each $n \in \N$, one can take $\sigma_n \in [s_n, s_n + 1]$ such that
$$
\partial_s (|v|^{q-2}v) (\sigma_n) \to 0 \quad \text{ strongly in } \X^*.
$$
\prf{Indeed, such a sequence $(\sigma_n)$ exists; if not, there exist $\epsilon > 0$ and an infinite number of intervals $[s_n, s_n + 1]$ on which $\| \partial_s (|v|^{q-2}v)(s) \|_{\X^*}^2 > \epsilon$, but this contradicts \eqref{dv-bdd}, since the union of these intervals is of infinite measure.} By use of \eqref{v-bdd}, one can further take a (not relabeled) subsequence of $(n)$ and $\phi \in \X$ such that
\begin{equation}\label{v-wcvg}
v(\sigma_n) \to \phi \quad \text{ weakly in } \X,
\end{equation}
which also implies that
$$
(-\Delta)^\theta v(\sigma_n) \to (-\Delta)^\theta \phi \quad \text{ weakly in } \X^*.
$$
Moreover, due to the compact embedding $\X \hookrightarrow L^q(\Omega)$, we can deduce that
\begin{equation}\label{vq-cvg1}
(|v|^{q-2}v)(\sigma_n) \to |\phi|^{q-2}\phi \quad \text{ strongly in } L^{q'}(\Omega).
\end{equation}
Passing to the limit in \eqref{eq:1.6}--\eqref{eq:1.8}, we get
\begin{equation}\label{limprob}
(-\Delta)^\theta \phi = \lambda_q |\phi|^{q-2}\phi \quad \text{ in } \X^*.
\end{equation}
Moreover, testing \eqref{eq:1.6} by $v(\sigma_n)$, we see that
\begin{align*}
\|v(\sigma_n)\|_{\X}^2 
&= - \left\langle \partial_s (|v|^{q-2}v)(\sigma_n), v(\sigma_n) \right\rangle_{\X} + \lambda_q \|v(\sigma_n)\|_{L^q(\Omega)}^q\\
&\to \lambda_q \|\phi\|_{L^q(\Omega)}^q.
\end{align*}
Testing \eqref{limprob} by $\phi$, we deduce that
$$
\lambda_q \|\phi\|_{L^q(\Omega)}^q = \|\phi\|_{\X}^2.
$$
Thus we obtain
$$
\|v(\sigma_n)\|_{\X}^2 \to \|\phi\|_{\X}^2,
$$
which along with \eqref{v-wcvg} and the uniform convexity of $\|\cdot\|_{\X}$ implies that
\begin{equation}\label{v-scvg}
v(\sigma_n) \to \phi \quad \text{ strongly in } \X.
\end{equation}
Since $J(v(\sigma_n)) \to J(\phi)$ and $J(v(\cdot))$ is nonincreasing on $[0,\infty)$, it follows that $J(v(s_n)) \to J(\phi)$. On the other hand, recalling \eqref{v-bdd}, we infer from \eqref{v-scvg} that $\phi \neq 0$.

From \eqref{v-bdd}, there exist a (not relabeled) subsequence of $(n)$ and $\psi \in \X$ such that
$$
v(s_n) \to \psi \quad \text{ weakly in } \X \ \text{ and strongly in } L^q(\Omega),
$$
which leads us to get
$$
(|v|^{q-2}v)(s_n) \to |\psi|^{q-2}\psi \quad \text{ strongly in } L^{q'}(\Omega).
$$
Moreover, we see that
\begin{align*}
\MoveEqLeft{
\left\| (|v|^{q-2}v)(s_n) - (|v|^{q-2}v)(\sigma_n) \right\|_{\X^*}
}\\
&\leq \int_{s_n}^{\sigma_n} \left\| \partial_s (|v|^{q-2}v)(r) \right\|_{\X^*} \, \d r\\
&\leq \left( \int_{s_n}^{\infty} \left\| \partial_s (|v|^{q-2}v)(r) \right\|_{\X^*}^2 \, \d r \right)^{1/2} \sqrt{\sigma_n - s_n} \to  0,
\end{align*}
which along with \eqref{dv-bdd} and \eqref{vq-cvg1} yields
$$
(|v|^{q-2}v)(s_n) \to |\phi|^{q-2}\phi \quad \text{ strongly in } \X^*.
$$
Thus we obtain $\phi = \psi$. One finds that
\begin{align*}
\frac{1}{2} \| v(s_n) \|_{\X}^2
&= J(v(s_n)) + \frac{\lambda_q}{q} \|v(s_n)\|_{L^q(\Omega)}^q\\
&\to J(\phi) + \frac{\lambda_q}{q} \|\phi\|_{L^q(\Omega)}^q = \frac12 \|\phi\|_{\X}^2.
\end{align*}
From the uniform convexity of $\|\cdot\|_{\X}$, we conclude that
$$
v(s_n) \to \phi \quad\text{ strongly in } \X,
$$
which completes the proof.
\end{proof}

\subsection{{\L}ojasiewicz--Simon gradient inequalities of Feireisl--Simondon type for the restricted fractional Laplacian}\label{Ss:LS}

This subsection is devoted to recalling a fractional variant of the {\L}ojasiewicz--Simon gradient inequality of Feireisl--Simondon type, which is developed in~\cite{ASS19} and will play a crucial role in order to prove Theorem \ref{T:conv} for the (fractional) fast diffusion case $q > 2$. 

The \emph{{\L}ojasiewicz inequality} was proved by S.~{\L}ojasiewicz~\cite{Lo} in 1965 (see also~\cite{LoZu}) and utilized for proving the (full) convergence of gradient flows in $\R^d$ to a single equilibrium as time goes to infinity. This innovative method was then extended by L.~Simon~\cite{Simon83} to infinite-dimensional settings and employed to prove the convergence of gradient flows in function spaces. It is often refereed as the \emph{{\L}ojasiewicz--Simon gradient inequality} (see, e.g.,~\cite{Jen98,HJ98,Har00,HJ01,HJK03,Chill03,CHJ09,HJ11}). An essential assumption for the {\L}ojasiewicz(--Simon) inequality is the \emph{analyticity} of the functional for which the inequality and gradient flows are considered; however, such an assumption may be too restrictive in view of applications to nonlinear PDEs. Indeed, even power nonlinearity may be ruled out. Feireisl and Simondon~\cite{FeiSim00} filled such a defect for some concrete functionals, which are associated with the Dirichlet problem for a wide class of (local) semilinear elliptic equations, with the aid of the classical elliptic regularity theory. Furthermore, the result of~\cite{FeiSim00} was also extended in~\cite{ASS19} to the restricted fractional Laplacian with the Dirichlet condition (we also refer the reader to, e.g.,~\cite{BlaBol18,HaMa19,FeeM20,YagiI} for other recent developments).

Let us consider the Dirichlet problem,
\begin{equation}\label{DPg}
\left\{
\begin{aligned}
(-\Delta)^\theta \varphi + g(\varphi) &= 0 && \text{ in } \Omega,\\
u &= 0 && \text{ on } \R^d \setminus \Omega,
\end{aligned}
\right.
\end{equation}
where $\Omega$ is a bounded domain of $\R^d$, $0 < \theta \leq 1$ and $g : \R \to \R$ is a function. For $0 < a, b \leq \infty$, we introduce assumptions for $g$ as follows:

\begin{description}
\item[(H0)] $g \in C^1(\R)$, $g(0)=0$.
\item[(H1)] {\bf (Analyticity)} $g \in C^\infty(-a,b)$, and moreover, for every $\alpha \in (0,a)$ and $\beta \in (0,b)$, there exist constants $C, M \geq 0$ such that
$$
|g^{(n)}(r)| \leq C M^n n! \quad \text{ for } \ r \in (-\alpha,\beta)
$$
for $n \in \N$ large enough.
\item[(H2)] {\bf (Analyticity with singularity at the origin)} $g \in C^\infty(0,b)$, and moreover, for every $\beta \in (0,b)$, there exist constants $C, M \geq 0$ such that
$$
|g^{(n)}(r)| \leq C \frac{M^n n!}{r^n} \quad \text{ for } \ r \in (0,\beta)
$$
for $n \in \N$ large enough.
\end{description}
In addition, if either $a$ or $b$ is infinite, we also assume the following growth condition for $g$:
\begin{description}
\item[(H3)] There exist $C \geq 0$ and $\sigma \in [1, 2_\theta^* - 1] \cap [1,\infty)$ such that
$$
|g'(r)| \leq C ( |r|^{\sigma-1} + 1 ) \quad \text{ for } \ r \in \R.
$$
\end{description}

Define a functional $I : \X \to (-\infty,\infty]$ by
$$
I(w):=
\begin{cases}
\frac{1}{2} \|w\|_{\X}^2 + \int_\Omega \hat{g}(w(x)) \,\d x & \text{ if } \ w \in \X, \; \hat{g}(w(\cdot)) \in L^1(\Omega),\\
\infty & \text{ otherwise},
\end{cases}
$$
where $\hat{g} : \R \to \R$ is a primitive function of $g$ such that $\hat{g}(0) = 0$. Under the hypothesis (H3), it is easy to see that $I$ is finite over $\X$ and of class $C^1$ with the Fr\'echet derivative,
$$
I'(w) := (-\Delta)^\theta w + g(w) \in \X^* \quad \text{ for } \ w \in \X.
$$
Then a weak form of the equation \eqref{DPg} reads,
$$
\varphi \in \X \ \text{ and } \ I'(\varphi) = 0 \ \text{ in } \X^*.
$$

Now we recall the following

\begin{theorem}[{\L}ojasiewicz--Simon inequalities for $I$ (cf.{~\cite[Theorem 5]{ASS19}})]\label{T:LS}
Let $\Omega$ be a bounded $C^{1,1}$ domain of $\R^d$. Assume {\rm (H0)}. In addition, when $0 < \theta < 1$, suppose that $d > 2\theta$. Let $\varphi \in \X \cap L^\infty(\Omega)$ be a weak solution to \eqref{DPg}.
\begin{enumerate}
 \item[(a)] Assume that either of the following holds\/{\rm :}
 \begin{enumerate}
  \item[(i)] {\rm (H1)} with $a = b = \infty$ and {\rm (H3)},
  \item[(ii)] {\rm (H2)} with $b = \infty$, {\rm (H3)} and $\varphi > 0$ a.e.~in $\Omega$.
 \end{enumerate}
Then there exist an exponent $\sigma \in (0,1/2]$ and constants $\omega, \delta>0$ such that
$$
 | I(w) - I(\varphi) |^{1-\sigma}\leq \omega \| I'(w) \|_{\X^*} = \omega \left\| (-\Delta)^\theta w + g(w) \right\|_{\X^*}
$$
for all $w\in \X$ satisfying $\|w-\varphi\|_{\X}<\delta$.
\item[(b)] Let $\eta > 0$ and assume that either of the following holds\/{\rm :}
 \begin{enumerate}
  \item[(iii)] {\rm (H1)} and $\|\varphi\|_{L^\infty} < \eta$ with $a,b > 0$ satisfying $\eta < \min\{a,b\} < \infty$,
  \item[(iv)] {\rm (H2)} and $0 < \varphi < \eta$ a.e.~in $\Omega$ with $b > 0$ satisfying $\eta < b < \infty$.
 \end{enumerate}
Then there exist an exponent $\sigma \in (0,1/2]$ and constants $\omega, \delta>0$ such that
$$
 | I(w) - I(\varphi) |^{1-\sigma} \leq \omega \| I'(w) \|_{\X^*} = \omega \left\| (-\Delta)^\theta w + g(w) \right\|_{\X^*}
$$
for all $w \in \X \cap L^\infty(\Omega)$ satisfying $\|w-\varphi\|_{\X} < \delta$ and $\|w\|_{L^\infty(\Omega)} < \eta$.
\end{enumerate}
\end{theorem}

\begin{proof}
The case $\theta = 1$ was first studied in~\cite{FeiSim00}. The above-mentioned theorem is almost same as Theorem 5 of~\cite{ASS19}, whose proof is concerned with the fractional case $0 < \theta < 1$ but clearly available for the case $\theta = 1$ as well. On the other hand, in (H1) and (H2) mentioned above, growth conditions are ``local'' in contrast with the original ones in~\cite[Theorem 5]{ASS19}; indeed, in (H1) above, the constants $C, M$ may depend on the choice of $\alpha,\beta$. However, this slight modification does not require any major change of the original proof in~\cite{ASS19}. To be more precise, a simple modification is needed only for the proof of Proposition 6.2 of~\cite{ASS19}, where the analyticity around $\varphi$ of the operator $w \mapsto g(w)$ from $X^\theta_p := \{w \in \X \cap L^p(\Omega) \colon (-\Delta)^\theta w \in L^p(\Omega)\}$ endowed with the norm $\|\cdot\|_{X^\theta_p} := \|\cdot\|_{L^p(\Omega)} + \|\cdot\|_{\X} + \|(-\Delta)^\theta \cdot\|_{L^p(\Omega)}$ into $L^p(\Omega)$ is proved for $p \in (1,\infty)$ large enough. Indeed, as $d > 2 \theta$, the space $X^\theta_p$ is continuously embedded in $L^\infty(\Omega)$ for $p \in (1,\infty)$ large enough.\footnote{In case $\theta = 1$, it is nothing but the $L^p$ elliptic estimate and a Sobolev embedding. Indeed, $X^\theta_p$ coincides with $W^{2,p}(\Omega) \cap H^1_0(\Omega)$. Moreover, we have $W^{2,p}(\Omega) \hookrightarrow L^\infty(\Omega)$ for $p > d/2$ without any restriction of $d$ (cf.~see~\cite{FeiSim00}). On the other hand, such an $L^p$ elliptic estimate may fail for $0 < \theta < 1$.} Hence any $w \in X^\theta_p$ lying on a small neighbourhood of $\varphi$ (in $X^\theta_p$) is uniformly bounded in $L^\infty(\Omega)$, say $\|w\|_{L^\infty(\Omega)} < r_0$ for some $r_0 > 0$. We then apply either (H1) or (H2) by taking $\alpha,\beta > r_0$ to prove the analyticity.
\end{proof}

\subsection{Full-convergence to asymptotic profiles}\label{Ss:conv}

This subsection is devoted to proving Theorem \ref{T:conv}. Assume that $1 < q < 2_\theta^*$ and that $\phi > 0$ a.e.~in $\Omega$. We shall show that 
\begin{equation}\label{fullcvg}
v(s) \to \phi \quad \text{ strongly in } \X \quad \text{ as } \ s \to \infty.
\end{equation}
First of all, assuming that $\Omega$ is a bounded $C^{1,1}$ domain of $\R^d$, one can prove $\phi \in L^\infty(\Omega)$ even for $0 < \theta < 1$ as in Appendix \S \ref{A:moser} by performing a standard bootstrap argument (with the aid of \eqref{EleIne}) and by exploiting a fractional elliptic regularity (see~\cite[Proposition 1.4]{ROS14-2}). The following lemma will be used for the fast diffusion case, i.e., $q > 2$.
\begin{lemma}
Let $2 < q < 2^*_\theta$. Under the same assumptions as in Theorem {\rm \ref{T:conv}}, there exist constants $\delta > 0$, $c_0 > 0$ and $\sigma \in (0,1/2]$ such that 
\begin{equation}\label{enin-LS}
c_0 \left\| \partial_s (|v|^{q-2}v)(s) \right\|_{\X^*} \leq - \frac{\d}{\d s} \left( J(v(s)) - J(\phi) \right)^\sigma,
\end{equation}
whenever $\| v(s) - \phi \|_{\X} < \delta$ and $J(v(s)) - J(\phi) > 0$.
\end{lemma}

\begin{proof}
Recall the energy inequality \eqref{estJ}, i.e.,
$$
\frac{4}{qq'} \left\| \partial_s (|v|^{(q-2)/2}v)(s) \right\|_{L^2(\Omega)}^2 + \frac{\d}{\d s} J(v(s)) \leq 0
$$
for a.e.~$s > 0$. Hence since $|v|^{(q-2)/2}v \in W^{1,2}(0,\infty;L^2(\Omega))$ and $q > 2$, we can deduce that
$$
\partial_s (|v|^{q-2}v)(s) = \frac{2}{q'} |v|^{(q-2)/2} \partial_s (|v|^{(q-2)/2}v)(s) \quad \text{ for a.e. } s > 0.
$$
\prf{Indeed, the map $w \in L^2(\Omega) \mapsto |w|^{(2-q')/q'}w \in L^{q'}(\Omega)$ is of class $C^1$. }Moreover, thanks to H\"older's inequality,
$$
\left\| \partial_s (|v|^{q-2}v)(s) \right\|_{L^{q'}(\Omega)} \leq \frac{2}{q'} \left\| |v|^{(q-2)/2}(s) \right\|_{L^{2q/(q-2)}} \left\| \partial_s (|v|^{(q-2)/2}v)(s) \right\|_{L^2(\Omega)}
$$
for a.e.~$s > 0$. Recalling that $\| |v|^{(q-2)/2}(s) \|_{L^{2q/(q-2)}(\Omega)} = \|v(s)\|_{L^q(\Omega)}^{(q-2)/2}$ is bounded for $s \geq 0$ and using the Sobolev--Poincar\'e inequality \eqref{SP}, we can take a constant $c > 0$ such that
$$
c \left\| \partial_s (|v|^{q-2}v)(s) \right\|_{\X^*}^2 \leq \frac{4}{qq'} \left\| \partial_s (|v|^{(q-2)/2}v)(s) \right\|_{L^2(\Omega)}^2
$$
for a.e.~$s > 0$. Consequently, \eqref{estJ} yields
$$
c \left\|\partial_s (|v|^{q-2}v)(s) \right\|_{\X^*}^2 + \frac{\d}{\d s} J(v(s)) \leq 0
$$
for a.e.~$s > 0$. Noting that $v$ solves
$$
\partial_s (|v|^{q-2}v)(s) = - J'(v(s)) \quad \text{ in } \X^*
$$
for a.e.~$s > 0$ (see Remark \ref{R:GF}), we see that
\begin{equation}\label{C1}
c\|J'(v(s))\|_{\X^*} \left\| \partial_s (|v|^{q-2}v)(s) \right\|_{\X^*} + \frac{\d}{\d s} J(v(s)) \leq 0
\end{equation}
for a.e.~$s > 0$. 

We now apply (a) of Theorem \ref{T:LS}. To this end, we set $g(r) := - \lambda_q |r|^{q-2}r$ and check (H0), (H3) and either (H1) or (H2). We can immediately check (H0), i.e., $g \in C^1(\R)$ and $g(0) = 0$, because $q > 2$. From the assumption $q < 2_\theta^*$, (H3) also follows immediately. We next check (H2) for the case of (i), i.e., $\phi \geq 0$ a.e.~in $\Omega$. We note that $\phi > 0$ in $\Omega$ (see, e.g.,~\cite{ROS14}). For each $\beta > 0$ and $n \in \N$ satisfying $n > q$,
\begin{align*}
|g^{(n)}(r)|
&= \lambda_q \left| (q-1)(q-2) \dots (q-n) \right| r^{q-n-1}\\
&\leq K_q \beta^{q-1} \frac{n!}{r^n} \quad \text{ for } \ r \in (0,\beta)
\end{align*}
for some constant $K_q$ depending only on $q$. Thus (H2) follows with $b = \infty$. We can further check (H1) for the case of (ii), i.e., $q$ is even. Indeed, $g$ is a monomial, whose $n$-th derivative $g^{(n)}$ hence vanishes for $n$ large enough. Thus (H1) follows with $a = b = \infty$. 

Thanks to (a) of Theorem \ref{T:LS}, one can take $\sigma \in (0,1/2]$ and constants $\delta, \omega>0$ such that
$$
\left| J(w) - J(\phi) \right|^{1-\sigma} \leq \omega \| J'(w) \|_{\X^*}
$$
for $w \in \X$ satisfying $\| w - \phi \|_{\X} < \delta$. Therefore we obtain from \eqref{C1},
$$
c \omega^{-1} \left( J(v(s)) - J(\phi) \right)^{1-\sigma} \left\| \partial_s (|v|^{q-2}v)(s) \right\|_{\X^*} + \frac{\d}{\d s} J(v(s)) \leq 0,
$$
whenever $\| v(s) - \phi \|_{\X} < \delta$. Therefore we obtain
$$
c \omega^{-1} \left\| \partial_s (|v|^{q-2}v)(s) \right\|_{\X^*} \leq - \frac{1}{\sigma} \frac{\d}{\d s} \left( J(v(s)) - J(\phi) \right)^{\sigma},
$$
provided $\| v(s) - \phi \|_{\X} < \delta$ and $J(v(s)) - J(\phi) > 0$.
\end{proof}

The next lemma is concerned with the uniqueness of positive solutions to sublinear elliptic equations and will be employed for the porous medium case, i.e., $1 < q < 2$.

\begin{lemma}[Uniqueness of positive solutions for $1 < q < 2$~{\cite[Theorem 1]{BrOs},~\cite[Corollary 2.4]{BSV},~\cite[Lemma 3.2]{FraVol23}}]\label{L:uniq-DP}
Let $\Omega$ be a bounded $C^{1,1}$ domain of $\R^d$ and let $0 < \theta \leq 1$ and $1 < q < 2$. Then the positive solution to the Dirichlet problem \eqref{eq:1.10}, \eqref{eq:1.11} is unique.
\end{lemma}

This lemma can be proved as in~\cite[Corollary 2.4]{BSV}. Moreover, it also follows from a more general uniqueness result, which is a fractional variant of Br\'ezis--Oswald's uniqueness result~\cite{BrOs} and which will be provided with a proof in Appendix \S \ref{A:BO}. 

\prf{Here we also give a proof of \eqref{L:uniq-DP} based on~\cite[Corollary 2.4]{BSV}.
\begin{proof}
Let $\phi_1$, $\phi_2$ be positive solutions to \eqref{eq:1.10}, \eqref{eq:1.11}. Set
$$
u_1(x,t) := t^{-1/(2-q)} \phi_1(x), \quad u_2(x,t) := (t+1)^{-1/(2-q)} \phi_2(x)
$$
for $x \in \Omega$ and $t > 0$. Then both $u_1$ and $u_2$ solve \eqref{pde}, \eqref{bc}. From Hopf's lemma as well as elliptic regularity (see, e.g.,~\cite{Evans} for $\theta = 1$ and~\cite{ROS14} for $0 < \theta < 1$), one can take $t_0 > 0$ small enough so that
$$
u_1(x,t_0) \geq u_2(x,t_0)
$$
for $x \in \Omega$. Therefore due to the comparison principle for \eqref{pde}--\eqref{ic} (see Theorem \ref{T:main}) we infer that $u_1 \leq u_2$ in $\Omega \times (t_0,\infty)$, which implies
$$
\phi_2(x) \leq \left( \frac{t+1}t \right)^{1/(2-q)} \phi_1(x) 
$$
for $x \in \Omega$ and $t \geq t_0$. Thus passing to the limit as $t \to \infty$, we obtain $\phi_2 \leq \phi_1$ in $\Omega$. Exchanging the roles of $\phi_1$ and $\phi_2$ and repeating the same argument, we can derive the inverse inequality. Therefore we conclude that $\phi_1 = \phi_2$, that is, the positive solution is unique.
\end{proof}
}

We are now ready for proving Theorem \ref{T:conv}.

\begin{proof}[Proof of Theorem {\rm \ref{T:conv}}]
In case $q > 2$, assume either (i) or (ii). Let $v$ and $\phi$ be an energy solution to \eqref{eq:1.6}--\eqref{eq:1.8} and a nontrivial solution to \eqref{eq:1.10}, \eqref{eq:1.11}, respectively, such that
\begin{equation}\label{hyp-cvg}
 v(s_n) \to \phi \quad \text{ strongly in } \X
\end{equation}
for some $s_n \to \infty$. Assume to the contrary that there exist $\epsilon_0 > 0$ and a sequence $(\sigma_n)$ in $(0,\infty)$ diverging to $\infty$ such that 
\begin{equation}\label{contr}
\| v(\sigma_n) - \phi \|_{\X} \geq \epsilon_0 \quad \text{ for } \  n \in \N.
\end{equation}
Extracting a (not relabeled) subsequence of $(n)$, we can assume, without loss of generality, that
$$
s_n < \sigma_n < s_{n+1} \quad \text{ for } \  n \in \N.
$$
Moreover, thanks to Theorem \ref{T:ap}, one can further extract a (not relabeled) subsequence of $(n)$ such that
$$
v(\sigma_n) \to \psi \quad \text{ strongly in } \X
$$
for some $\psi \in \X \setminus \{0\}$. We see immediately that $\|\psi - \phi\|_{\X} \geq \epsilon_0 > 0$ from \eqref{contr}. From the decrease of $J(v(\cdot))$, we see that
$$
J(v(s)) \geq J(\phi) = J(\psi) \quad \text{ for } \ s \geq 0.
$$
If $J(v(s_0)) - J(\phi) = 0$ for some $s_0 > 0$, then \eqref{estJ} yields $v(s) = v(s_0)$ for all $s \geq s_0$, which contradicts \eqref{hyp-cvg} and \eqref{contr}. Hence we may suppose, without loss of generality, that
$$
J(v(s)) - J(\phi) > 0 \quad \text{ for } \  s \geq 0.
$$
Let $0 < \delta_0 < \min\{\delta,\epsilon_0\}$. From the fact that $v \in C_+([0,\infty);\X) \cap C_{\rm weak}([0,\infty);\X)$, for $n \in \N$ large enough, we can still take $\tilde{s}_n \in (s_n,\sigma_n)$ such that
\begin{equation}\label{v-phi}
\| v(\tilde{s}_n) - \phi \|_{\X} = \delta_0 \quad \text{ and } \quad \| v(s) - \phi \|_{\X} < \delta_0 \ \text{ for } \ s \in [s_n,\tilde{s}_n)
\end{equation}
(see~\cite{A16}). Since $J(v(\cdot))$ is nonincreasing on $[0,\infty)$ thanks to (i) of Lemma \ref{L:v-enin}, so is $(J(v(\cdot))-J(\phi))^\sigma > 0$. Hence recalling~\cite[Corollary 3.29]{AFP}, one can verify from \eqref{enin-LS} that 
$$
c_0 \int_{s_n}^{\tilde{s}_n} \left\|\partial_s (|v|^{q-2}v)(s) \right\|_{X_\theta^*} \, \d s \leq \left( J(v(s_n)) - J(\phi) \right)^\sigma \to 0,
$$
which implies
\begin{equation}\label{diff-conv}
(|v|^{q-2}v)(\tilde{s}_n) - (|v|^{q-2}v)(s_n) \to 0 \quad \text{ strongly in } \X^*.
\end{equation}
On the other hand, by virtue of Theorem \ref{T:ap}, there exists $\tilde{\phi} \in \X \setminus \{0\}$ such that, up to a (not relabeled) subsequence,
$$
v(\tilde{s}_n) \to \tilde{\phi} \quad \text{ strongly in } \X.
$$
Since $\X \hookrightarrow L^q(\Omega)$ and $L^{q'}(\Omega) \hookrightarrow \X^*$ compactly, \eqref{diff-conv} implies $\tilde{\phi} = \phi$. However, \eqref{v-phi} implies
$$
\| \tilde{\phi} - \phi \|_{\X} = \delta_0 > 0,
$$
which is a contradiction. Thus \eqref{fullcvg} follows.

In case $1 < q < 2$, suppose that $\phi \geq 0$ a.e.~in $\Omega$. The full convergence \eqref{fullcvg} has already been proved for nonnegative solutions in~\cite{BSV} by exploiting the B\'enilan--Crandall estimate for \eqref{pde}--\eqref{ic} to guarantee the monotonicity of $v = v(x,s) \geq 0$ in the rescaled time $s$. Moreover, it is also verified in~\cite[Theorem 1.1]{FraVol23} from a variational point of view for (possibly) sign-changing  $v = v(x,s)$ emanating from initial data whose energies are less than the second least energy for \eqref{eq:1.10}, \eqref{eq:1.11}. Here we can also prove \eqref{fullcvg} immediately from the quasi-convergence of $v$ to a nonnegative asymptotic profile $\phi \geq 0$, which turns out to be positive in $\Omega$ due to the fractional strong maximum principle (see, e.g.,~\cite[Theorem 2.1]{GrSe}), as well as the uniqueness of the positive solution to \eqref{eq:1.10}, \eqref{eq:1.11} (see Lemma \ref{L:uniq-DP}). This completes the proof.
\end{proof}

\appendix

\section{Some lemma}

The following lemma may be standard, but it is given just for the convenience of the reader.

\begin{lemma}\label{L:NemOp}
Let $\Omega$ be a measurable subset of $\mathbb{R^N}$. Let $1 < q_1, q_2 < \infty$ and set $\alpha := q_1/q_2$. Then the map $u \mapsto |u|^{\alpha-1}u $ is continuous from $L^{q_1}(\Omega)$ to $L^{q_2}(\Omega)$.
\end{lemma}

This lemma can be assured with the use of a general theory of Nemytskii operators; here we give a simpler proof, which is specific to power nonlinearities. 

\begin{proof}
Let $u_n, u \in L^{q_1}(\Omega)$ be such that $u_n \to u$ strongly in $L^{q_1}(\Omega)$. Then we can extract a (not relabeled) subsequence such that $u_{n}$ converges to $u$ a.e.~in $\Omega$. Also, $\| |u_{n}|^{\alpha-1}u_{n} \|_{L^{q_2}(\Omega)}^{q_2} = \| u_{n} \|_{L^{q_1}(\Omega)}^{q_1}$ is bounded. Furthermore, we may further extract a (not relabeled) subsequence of $(n)$ such that $|u_{n}|^{\alpha-1}u_{n} \to v$ weakly in $L^{q_2}(\Omega)$ for some $v\in L^{q_2}(\Omega)$. It is then easy to show (e.g., with the aid of Mazur's Lemma, see \cite{B-FA}) that the pointwise limit of $( |u_{n}|^{\alpha-1}u_{n} )$ and its weak limit in $L^{q_2}(\Omega)$ coincide, so that $v = |u|^{\alpha-1}u$. Noting by assumption that
$$
\| |u_{n}|^{\alpha-1}u_{n} \|_{L^{q_2}(\Omega)}^{q_2} = \|u_{n}\|_{L^{q_1}(\Omega)}^{q_1} \to \|u\|_{L^{q_1}(\Omega)}^{q_1} = \| |u|^{\alpha-1}u \|_{L^{q_2}(\Omega)}^{q_2},
$$
we can deduce from the uniform convexity of $L^{q_2}(\Omega)$ that $|u_{n}|^{\alpha-1}u_{n} \to |u|^{\alpha-1}u$ strongly in $L^{q_2}(\Omega)$. Moreover, from the uniqueness of the limit, the whole sequence $(|u_n|^{\alpha-1}u_n)$ also converges to $|u|^{\alpha-1}u$ strongly in $L^{q_2}(\Omega)$.
\end{proof}

\section{Density results in fractional Sobolev spaces}\label{S:apdx}

Let $\Omega$ be a measurable subset of $\R^d$, $p \in [1,\infty)$ and $\theta \in(0,1)$. The fractional Sobolev space on $\Omega$ is defined by
$$
W^{\theta,p}(\Omega) := \left\{ u \in L^p(\Omega) \colon [u]_{W^{\theta,p}(\Omega)} < \infty \right\}, 
$$
where 
$$
[u]_{W^{\theta,p}(\Omega)} := \left( \iint_{\Omega\times\Omega} \frac{|u(x)-u(y)|^p}{|x-y|^{d+p\theta}} \,\d x \d y \right)^{1/p},
$$
and which is endowed with the norm,
$$
\|u\|_{W^{\theta,p}(\Omega)} := \|u\|_{L^p(\Omega)} + [u]_{W^{\theta,p}(\Omega)}.
$$
We also define the space,
\begin{align*}
H^{\theta,p}(\R^d) &:= \Big\{ u \in L^p(\R^d) \colon \\
&\qquad \|u\|_{H^{\theta,p}(\R^d)} := \| \mathcal{F}^{-1} [(1 + |\xi|^2)^{\theta/2} \mathcal{F}u ] \|_{L^p(\R^d)} < \infty \Big\},
\end{align*}
where $\mathcal{F}$ and $\mathcal{F}^{-1}$ are the Fourier transform operator and its inverse, respectively, endowed with the norm $\|\cdot\|_{H^{\theta,p}(\R^d)}$. In particular, we write $H^\theta(\R^d) = H^{\theta,2}(\R^d)$. It is shown in~\cite[Examples 1.8 and 2.12]{Lunardi} that
\begin{align}
W^{\theta,p}(\R^d) &= (L^p(\R^d),W^{1,p}(\R^d))_{\theta,p},\label{intpl}\\
H^{\theta,p}(\R^d) &= [L^p(\R^d),W^{1,p}(\R^d)]_\theta\nonumber
\end{align}
for $\theta \in (0,1)$ and $p \in (1,\infty)$. Here $(X,Y)_{\theta,p}$ and $[X,Y]_\theta$ denote the \emph{real} and \emph{complex interpolation spaces}, respectively, between $X$ and $Y$ for two Banach spaces $X$, $Y$. For more details on the interpolation theory, we refer the reader to, e.g.,~\cite{Lunardi}. The following two theorems are well known:

\begin{theorem}\label{T:density}
Let $\theta \in (0,1)$ and $p \in [1,\infty)$. The space $C_c^\infty(\R^d)$ of compactly supported smooth functions is dense in $W^{\theta,p}(\R^d)$.
\end{theorem}

\begin{theorem}\label{T:ext}
Let $\theta \in (0,1)$ and $p \in [1,\infty)$ and let $\Omega \subset \R^d$ be an open set of class $C^{0,1}$ with bounded boundary $\partial \Omega$. Then there exists an extension operator from $W^{\theta,p}(\Omega)$ to $W^{\theta,p}(\R^d)$. Namely, there exists a bounded linear operator $T : W^{\theta,p}(\Omega) \to W^{\theta,p}(\R^d)$ such that $(Tu)|_{\Omega} = u$ for $u \in W^{\theta,p}(\Omega)$.
\end{theorem}

We refer the reader to~\cite[Theorems 2.4 and 5.4]{Hitchhiker} for proofs of Theorems \ref{T:density} and \ref{T:ext}. In particular, the following theorem holds:

\begin{theorem}
Let $\Omega \subset \R^d$ be an open set of class $C^{0,1}$ with bounded boundary $\partial \Omega$ and let $\theta \in (0,1)$ and $p\in[1,\infty)$. Then $C^\infty(\overline{\Omega})$ is dense in $W^{\theta,p}(\Omega)$.
\end{theorem}

Combining \eqref{intpl} with Theorem \ref{T:ext}, one can show that
$$
W^{\theta,p}(\Omega) = ( L^p(\Omega),W^{1,p}(\Omega) )_{\theta,p},
$$
whenever $\Omega$ is a Lipschitz open subset of $\R^d$.

We focus from now on to the case where $p = 2$. As shown in~\cite[Corollary 4.37]{Lunardi} (or~\cite[Proposition 3.4]{Hitchhiker}), when $p = 2$, the two definitions of fractional Sobolev spaces over $\R^d$ coincide, that is, $W^{\theta,2}(\R^d) = H^\theta(\R^d)$ with an equivalence of the norms. However, when $p \neq 2$, $H^{\theta,p}(\R^d)$ is strictly included in $W^{\theta, p}(\R^d)$, or the contrary, depending on the choice of $p$. This also justifies the following notation used hereafter:
$$
H^\theta(\Omega) := W^{\theta,2}(\Omega), \quad \|\cdot\|_{H^\theta(\Omega)} := \|\cdot\|_{W^{\theta,2}(\Omega)}, \quad [\,\cdot\,]_{H^\theta(\Omega)} := [\,\cdot\,]_{W^{\theta,2}(\Omega)}.
$$
The space $H^\theta(\Omega)$ is a Hilbert space endowed with the inner product,
$$
(u,v)_{H^{\theta}(\Omega)} := (u,v)_{L^2(\Omega)} + \iint_{\Omega\times\Omega} \frac{ (u(x)-u(y))(v(x)-v(y)) }{|x-y|^{d+2\theta}}\, \d x \d y.
$$

\prf{Another characterization of the space $H^\theta(\Omega)$ in terms of fractional powers of a self-adjoint operator is also available. Let us denote by $D(S)$ the set of $u \in H^1(\Omega)$ for which the map
$$
v \in H^1(\Omega) \mapsto (u,v)_{H^1(\Omega)} := (u,v)_{L^2(\Omega)} + (\nabla u,\nabla v)_{L^2(\Omega)} \in \R
$$
is continuous with respect to the topology induced by $L^2(\Omega)$. Then, for all $u \in D(S)$, there exists a unique $Su \in L^2(\Omega)$ such that 
\begin{equation}\label{weakform}
(u,v)_{H^1(\Omega)} = (Su,v)_{L^2(\Omega)} \quad \text{ for } \ v \in H^1(\Omega).
\end{equation}
Equation \eqref{weakform} is a weak form of the elliptic equation,
$$
\left\{
\begin{aligned}
&-\Delta u + u = Su &&\text{ in } \Omega,\\
&\nabla u \cdot \mathbf{n} = 0 &&\text{ on } \partial\Omega.
\end{aligned}\right.
$$
(see, e.g.,~\cite[Chap.~9]{B-FA}). {Here $\mathbf{n}$ stands for the outward unit normal vector field on $\partial \Omega$.}  Hence, by a standard elliptic {regularity} theory (see, e.g.,~\cite[Chap.~9]{B-FA}), we have 
\begin{align*}
D(S) &:= \left\{ u \in H^2(\Omega) \colon \nabla u \cdot \mathbf{n} = 0 \text{ on } \partial\Omega \right\},\\
Su &= -\Delta u + u \quad \text{ for } \ u \in D(S).
\end{align*}
Then $S$ is a positive {unbounded self-adjoint} operator on $L^2(\Omega)$, and it can be shown {(see, e.g.,~\cite[Theorem 9.31]{B-FA})} that there exist a Hilbert basis $(e_n)_{n\geq 1}$ of $L^2(\Omega)$ and a nondecreasing sequence $(\lambda_n)_{n\geq 1}$ of positive reals {satisfying} $\lambda_n \to \infty$ such that
$$
e_n \in D(S), \quad S e_n = \lambda_n e_n.
$$
For $\sigma \geq 0$, the fractional power $S^\sigma$ of $S$ is defined by
\begin{align*}
D(S^\sigma) &:= \left\{ \sum_{n \geq 1} a_n e_n \in L^2(\Omega) \colon \sum_{n \geq 1} a_n^2 \lambda_n^{2\sigma} < \infty \right\},\\
S^\sigma u &:= \sum_{n \geq 1} \lambda_n^\sigma a_n e_n \quad \text{ for } \ u = \sum_{n \geq 1} a_n e_n \in D(S^\sigma),
\end{align*}
and moreover, the {linear space} $D(S^\sigma)$ is endowed with the graph norm $\|u\|_{D(S^{\sigma})} = \|u\|_{L^2(\Omega)} + \|S^{\sigma} u\|_{L^2(\Omega)}$ for $u \in D(S^\sigma)$. Note that $D(S^{1/2}) = H^1(\Omega)$ and $D(S^0) = L^2(\Omega)$ with equivalence of the norms. Then it is shown in~\cite[Theorem 4.36]{Lunardi} that
$$
H^\theta(\Omega) = D(S^{\theta/2})
$$
with equivalence of the norms, whenever $\Omega$ is a Lipschitz open set with bounded boundary.}

In what follows, we assume that $\Omega \subset \R^d$ is bounded and smooth (see (7.10) of~\cite[p.\,34]{LM} for more detail) so that there exists a $C^\infty$-extension $\rho$ of the distance function $\mathrm{dist}(\cdot,\partial\Omega)$ to the boundary, i.e. $\rho \in C^\infty(\overline{\Omega})$, $\rho > 0$ on $\Omega$, $\rho = 0$ on $\partial\Omega$, and moreover, 
$$
\lim_{x \to x_0} \frac{\rho(x)}{\mathrm{dist}(x,\partial\Omega)} = d
\quad \mbox{if } \ x_0 \in \partial\Omega 
$$ 
for some constant $d>0$ uniformly for $x_0 \in \partial \Omega$. We set 
\begin{align*}
H_0^\theta(\Omega) &:= \overline{C_c^\infty(\Omega)}^{\|\cdot\|_{H^\theta(\Omega)}} \ \text{ for } \ 0 < \theta < 1,\\
H^{1/2}_{00}(\Omega) &:= \left\{ u \in H^{1/2}_0(\Omega) \colon \rho^{-1/2}u \in L^2(\Omega) \right\}.
\end{align*}
Moreover, $H^\theta_0(\Omega)$ is endowed with the induced norm of $H^\theta(\Omega)$, while $H^{1/2}_{00}(\Omega)$ is endowed with the norm $\|u\|_{H^{1/2}(\Omega)} + \|\rho^{-1/2}u\|_{L^2(\Omega)}$. The following theorems can be found in~\cite[Chap.\,1, \S 9, \S 11]{LM}. We also refer the reader to~\cite[Chap.\,1, \S 7]{LM} for the definition of the space $H^{\theta}(\partial\Omega)$, whenever $\Omega$ is a bounded open set of $\R^d$ with smooth boundary and $\theta \in (0,1)$.

\begin{theorem}[{\cite[Chap.\,1, Theorem 9.4]{LM}}]\label{T:trace}
Let $\Omega$ be a bounded open subset of $\R^d$ with smooth boundary $\partial \Omega$ and let $\theta > 1/2$. Then the map
$$
u \in C^\infty(\overline{\Omega}) \mapsto u|_{\partial\Omega} \in C^\infty(\partial\Omega)
$$
is linear and continuous from $H^\theta(\Omega)$ to $H^{\theta-1/2}(\partial\Omega)$. Therefore it is uniquely extended to a continuous operator $\gamma$ from $H^\theta(\Omega)$ to $H^{\theta-1/2}(\partial\Omega)$.
\end{theorem}

\begin{theorem}[{\cite[Chap.\,1, Theorem 11.1]{LM}}]\label{T:cas1}
Let $\Omega$ be a bounded open subset of $\R^d$ with smooth boundary. If $\theta \in (0,1/2]$, then $H_0^\theta(\Omega) = H^\theta(\Omega)$. If $\theta \in (1/2,1]$, then $H_0^\theta(\Omega)$ is a strict subspace of $H^\theta(\Omega)$.
\end{theorem}

\begin{theorem}[{\cite[Chap.\,1, Theorem 11.4]{LM}}]\label{T:ext0}
Let $\Omega$ be a bounded open subset of $\R^d$ with smooth boundary and let $\theta \in (0,1) \setminus \{1/2\}$. Then the operator
$$
u \in H^\theta_0(\Omega) \mapsto \tilde{u}  \in H^\theta(\R^d)
$$
is continuous from $H_0^\theta(\Omega)$ to $H^\theta(\R^d)$. Here $\tilde u$ stands for the zero-extension of $u$ outside $\Omega$.
\end{theorem}

\begin{theorem}[{\cite[Chap.\,1, Theorem 11.5]{LM}}]\label{T:trace_ker}
Let $\Omega$ be a bounded open subset of $\R^d$ with smooth boundary and let $\theta \in (1/2,1]$. Then the following conditions are equivalent\/{\rm :}
\begin{enumerate}
\item $u \in H^\theta_0(\Omega)$,
\item $u \in H^\theta(\Omega)$ and $\gamma(u)=0$, where $\gamma : H^\theta(\Omega) \to H^{\theta-1/2}(\partial \Omega)$ is the trace operator given as in Theorem {\rm \ref{T:trace}}.
\end{enumerate}
\end{theorem}

\begin{theorem}[{\cite[Chap.\,1, Theorem 11.7]{LM}}]\label{T:undemi}
Let $\Omega$ be a bounded open subset of $\R^d$ with smooth boundary. Then it holds that
$$
H^{1/2}_{00}(\Omega) = [L^2(\Omega),H^1_0(\Omega)]_{1/2} = (L^2(\Omega),H^1_0(\Omega))_{1/2,2}
$$
with equivalence of the norms. Moreover, $H^{1/2}_{00}(\Omega)$ is a strict subspace of $H^{1/2}(\Omega)$ with a strictly finer topology.
\end{theorem}

We are now ready to identify the space,
$$
\X := \left\{u \in H^\theta(\mathbb{R}^d) \colon u = 0 \text{ a.e.~in } \R^d \setminus\Omega \right\}
$$
endowed with the seminorm $\|\cdot\|_{\X} := [\,\cdot\,]_{H^\theta(\Omega)}$ (see also Proposition \ref{P:Poincare}).

\begin{theorem}\label{T:ident}
Let $\Omega$ be a bounded open subset of $\R^d$ with smooth boundary and let $\theta \in (0,1)$.
\begin{enumerate}
\item If $\theta \neq 1/2$, then $\X = H^\theta_0(\Omega)$ with equivalence of the norms,
\item $\mathcal{X}_{1/2}(\Omega) = H^{1/2}_{00}(\Omega)$ with equivalence of the norms.
\end{enumerate}
\end{theorem}

\begin{proof}
From the definition of $\X$, we immediately find the continuous embedding,
\begin{equation}\label{XH-embd}
\X \hookrightarrow H^{\theta}(\Omega).
\end{equation}
In case $0 < \theta < 1/2$, from Theorem \ref{T:cas1}, it follows that $\X \hookrightarrow H^\theta(\Omega) = H_0^\theta(\Omega)$ with a continuous injection. Moreover, from Theorem \ref{T:ext0}, there exists a constant $C > 0$ such that, for any $u \in H_0^\theta(\Omega) = H^\theta(\Omega)$, $\|\tilde{u}\|_{H^\theta(\R^d)} \leq C\|u\|_{H^\theta(\Omega)}$, where $\tilde{u}$ denotes the zero extension of $u$. By definition, we see that $\|\tilde{u}\|_{H^\theta(\R^d)}\geq\|u\|_{\X}$. Thus we obtain $H_0^\theta(\Omega) = H^\theta(\Omega) \hookrightarrow \X$ with a continuous injection.

In case $1/2 < \theta < 1$, we show as in the previous case that $H_0^\theta(\Omega) \hookrightarrow \X$ with a continuous injection, by using Theorem \ref{T:ext0}. To prove the inverse embedding, we use the trace characterization of $H_0^\theta(\Omega)$ provided by Theorem \ref{T:trace_ker}. Letting $u \in \X$, we show that $\gamma(u) = 0 $. Note that, even though $u = 0$ on $\R^d \setminus \Omega$, this relation is not totally straightforward, since $\gamma(u)$ is defined from the restriction $u|_\Omega$ of $u$ to the domain $\Omega$. Since $\Omega$ is bounded, there exists a smooth bounded domain $\Omega' \subset \R^d$ such that $\Omega \Subset \Omega'$. Set $\Gamma = \partial\Omega$ and $\Omega'' =\Omega' \setminus (\Omega \cup \Gamma)$, which is a bounded and smooth subset of $\R^d$. Then $u \in H^\theta(\Omega')$ and $u|_{\Omega''} = 0$. Let $(\varphi_n)$ be a sequence in $C^\infty_c(\R^d)$ such that
$$
\varphi_n \to u \quad\text{ strongly in } H^\theta(\R^d).
$$
Indeed, it is possible by virtue of Theorem \ref{T:density}. In particular, we see that
\begin{alignat*}{3}
\varphi_n|_{\Omega''} &\to u|_{\Omega''} = 0 \quad &&\text{ strongly in } H^\theta(\Omega''),\\
\varphi_n|_{\Omega} &\to u|_{\Omega} \quad &&\text{ strongly in } H^\theta(\Omega).
\end{alignat*}
Applying Theorem \ref{T:trace} to (each connected component of) the open set $\Omega''$, we see that 
\begin{equation}\label{u-Gam-1}
\varphi_n|_{\Gamma} \to 0 \quad \text{ strongly in } H^{\theta-1/2}(\Gamma).
\end{equation}
Moreover, applying now Theorem \ref{T:trace} to the domain $\Omega$, we find that
\begin{equation}\label{u-Gam-2}
\varphi_n|_{\Gamma} \to \gamma(u) \quad \text{ strongly in } H^{\theta-1/2}(\Gamma),
\end{equation}
where $\gamma$ denotes the trace operator from $H^\theta(\Omega)$ to $H^{\theta-1/2}(\partial \Omega)$. Hence it follows from \eqref{u-Gam-1} and \eqref{u-Gam-2} that
$$
\gamma(u) = 0.
$$
By virtue of Theorem \ref{T:trace_ker}, we deduce that
$$
\X \hookrightarrow H_0^\theta (\Omega),
$$
where the injection is continuous due to \eqref{XH-embd}.

In case $\theta=1/2$, let $u \in H^{1/2}_{00}(\Omega)$ and denote by the same letter $u$ the zero extension of $u$ to $\R^d$. We have
\begin{equation}\label{X-half-norm}
\|u\|_{\mathcal{X}_{1/2}(\Omega)}^2 = \|u\|_{H^{1/2}(\Omega)}^2 + 2A
\end{equation}
with
$$
A := \int_{\Omega} \left( \int_{\Omega^c} \frac{|u(x)|^2}{|x-y|^{d+1}} \,\d y \right) \d x.
$$
Let us denote
$$
\rho_1(x) := \mathrm{dist}(x,\partial \Omega) \ \text{ for } \ x \in \Omega.
$$
Then we have
$$
A \leq \int_\Omega |u(x)|^2 \left( \int_{B(x;\rho_1(x))^c} \frac{1}{|x-y|^{d+1}} \,\d y \right) \d x,
$$
where $B(z;r)$ denotes the open ball centered at $z \in \R^d$ of radius $r > 0$. By change of variables, we have
\begin{align*}
A &\leq 
\int_\Omega |u(x)|^2 \left( d \omega_d \int^{\infty}_{\rho_1(x)} \frac 1 {r^2} \,\d r \right) \d x\\
&= d \omega_d \int_\Omega \frac{|u(x)|^2}{\rho_1(x)} \,\d x,
\end{align*}
where $\omega_d$ is the volume of the $d$-dimensional unit ball. We deduce that
$$
\|u\|_{\mathcal{X}_{1/2}(\Omega)}^2\leq \|u\|_{H^{1/2}(\Omega)}^2+2d\omega_d \| \rho_1^{-1/2}u \|_{L^2(\Omega)}^2.
$$
Thus $u \in \mathcal{X}_{1/2}(\Omega)$, and therefore, the continuous embedding $H^{1/2}_{00}(\Omega) \hookrightarrow \mathcal{X}_{1/2}(\Omega)$ follows.

We next derive the inverse embedding. Since $\partial\Omega$ is compact, for any $x \in \Omega$, there exists $p(x) \in \partial\Omega$ such that $\mathrm{dist}(x,\partial\Omega)=|p(x)-x|$. Moreover, for $y \in \partial\Omega$ and $r > 0$, we set
$B_{y,r} := B(y+r{\bf n}(y);r)$, where $\mathbf{n}(y)$ denotes the outer normal at $y \in \partial \Omega$, and denote by $C_{y,r}$ an arbitrary open hypercube centered at $y+r\mathbf{n}(y)$ of side length $\sqrt{2}r$ with a side parallel to $\mathbf{n}(y)$. Then we observe that
$$
C_{y,r} \subset B_{y,r}.
$$
If $\Omega$ complies with a uniform exterior sphere condition (e.g., $\Omega$ is a bounded $C^{1,1}$ domain), there exists $\varepsilon>0$ such that 
$$
B_{y,r} \subset \Omega^c
$$
for all $r \in (0,\varepsilon)$ and $y \in \partial \Omega$. 
Let $u \in \mathcal{X}_{1/2}(\Omega)$. Then it follows that
\begin{align*}
A &\geq \int_{\{ x \in \Omega \colon \mathrm{dist}(x,\partial\Omega) < \varepsilon\}} |u(x)|^2 \left( \int_{C_{p(x),{\rho_1}(x)}} \frac{1}{|x-y|^{d+1}} \,\d y \right) \d x\\
&\quad + \int_{ \{ x \in \Omega \colon \mathrm{dist}(x,\partial\Omega) \geq \varepsilon \}} |u(x)|^2 \left( \int_{\Omega^c} \frac{1}{|x-y|^{d+1}} \,\d y \right) \d x.
\end{align*}
Let us denote by $I_1$ and $I_2$ the first and second terms in the right-hand side, respectively. Let $x \in \Omega$ be such that $\mathrm{dist}(x,\partial\Omega) < \varepsilon$. By change of variables, 
$$
y = x + \mathbf{n}(p(x)) {\rho_1}(x) \eta_1 + {\rho_1}(x) \tilde \eta
$$
for 
\begin{align*}
\eta_1 &\in I := \left] 2-\tfrac{\sqrt{2}}2, 2+\tfrac{\sqrt{2}}2 \right[,\\
\tilde{\eta} &\in R := \left\{ \tilde y = y_2 \mathbf{e}_2 + \cdots + y_d \mathbf{e}_d \colon |y_j| < \sqrt{2}/2 \ \mbox{ for } \ j = 2,\ldots,d \right\},
\end{align*}
where $\{\mathbf{e}_2, \ldots, \mathbf{e}_d\}$ is a basis of the $(d-1)$-dimensional subspace orthogonal to the normal vector $\mathbf{n}(p(x))$, we have
\begin{align*}
\MoveEqLeft{
\int_{C_{p(x),{\rho_1}(x)}} \frac{1}{|x-y|^{d+1}} \,\d y
}\\
&= \int_{{\rho_1}(x)I} \left( \int_{{\rho_1}(x)R} \frac 1 {(y_1^2 + |\tilde{y}|^2)^{(d+1)/2}} \, \d \tilde{y} \right) \d y_1 \\
&= \frac 1 {{\rho_1}(x)} \int_I \left( \int_R \frac{1}{(\eta_1^2+|\tilde{\eta}|^2)^{(d+1)/2}} \,\d \tilde{\eta}\right) \d \eta_1 = \frac{C_1}{{\rho_1}(x)}
\end{align*}
for some constant $C_1$. Hence it follows that
$$
I_1 \geq C_1 \int_{ \{ x \in \Omega \colon \mathrm{dist}(x,\partial\Omega)<\varepsilon\}} |{\rho_1}(x)^{-1/2}u(x)|^2 \,\d x.
$$

On the other hand, since $\Omega$ is bounded, there exists $M > 0$ such that $\Omega \subset B(0;M)$. Let $x \in \Omega$ be such that $\mathrm{dist}(x,\partial\Omega)\geq \varepsilon$. Then we observe that 
\begin{align*}
\int_{\Omega^c}\frac{1}{|x-y|^{d+1}} \,\d y
&\geq \int_{B(0;M)^c}\frac{1}{|x-y|^{d+1}} \,\d y\\
&\geq \int_{B(0;M)^c} \frac{1}{(M+|y|)^{d+1}} \,\d y =: C_2.
\end{align*}
Thus,
\begin{align*}
I_2 &\geq C_2 \int_{ \{x \in \Omega \colon \mathrm{dist}(x,\partial\Omega)\geq\varepsilon \}} |u(x)|^2 \,\d x\\
&\geq C_2 \varepsilon\int_{ \{x \in \Omega \colon \mathrm{dist}(x,\partial\Omega)\geq\varepsilon \}} |{\rho_1}(x)^{-1/2}u(x)|^2 \,\d x.
\end{align*}
Hence,
$$
A \geq \min\{C_1,C_2\varepsilon\} \|{\rho_1}^{-1/2}u\|_{L^2(\Omega)}^2,
$$
and therefore, we conclude from \eqref{X-half-norm} that
$$
\mathcal{X}_{1/2}(\Omega) \hookrightarrow H^{1/2}_{00}(\Omega)
$$
with a continuous injection. This completes the proof.
\end{proof}

Thus we obtain
\begin{corollary}
Let $\Omega$ be a bounded open subset of $\R^d$ with smooth boundary and let $\theta \in(0,1)$. Then $C_c^\infty(\Omega)$ is dense in $\X$.
\end{corollary}

\begin{proof}
From Theorem \ref{T:ident}, we deduce that $\X = H^\theta_0(\Omega) = \overline{C_c^\infty(\Omega)}^{\|\cdot\|_{H^\theta(\Omega)}}$ as far as $\theta \neq 1/2$. On the other hand, by virtue of Theorems \ref{T:undemi} and \ref{T:ident}, we find that $\mathcal{X}_{1/2}(\Omega) = H^{1/2}_{00}(\Omega) = (L^2(\Omega),H^1_0(\Omega))_{1/2,2}$ with equivalence of norms. By general theory on real interpolation spaces (see~\cite[Proposition 1.17]{Lunardi}), we conclude that
$$
H^1_0(\Omega) \hookrightarrow \mathcal{X}_{1/2}(\Omega)
$$
with a continuous and densely defined injection. Consequently, since $C_c^\infty(\Omega)$ is dense in $H^1_0(\Omega)$, it is so in $\mathcal{X}_{1/2}(\Omega)$.
\end{proof}

\section{Regularization of initial data}\label{S:regu}

Let $\rho_0 \in \X^*$ and let $q > 1$ be fixed. From the density of $\X$ in $L^2(\Omega)$, one can take a sequence $(\tilde{\rho}_{0,n})$ in $L^2(\Omega)$ such that $\tilde{\rho}_{0,n} \to \rho_0$ strongly in $\X^*$. Let $u_{0,n} \in X := \X \cap L^q(\Omega)$ be the unique solution to the equation,
\begin{equation}\label{app-eq}
|u_{0,n}|^{q-2}u_{0,n} + \frac 1n (-\Delta)^\theta u_{0,n} = \tilde{\rho}_{0,n} \ \mbox{ in } X^*.
\end{equation}
Indeed, existence of such $u_{0,n} \in X$ can be assured, e.g., in a variational method, and moreover, since $\tilde \rho_{0,n} \in L^2(\Omega)$, one can check $|u_{0,n}|^{q-2}u_{0,n} \in L^2(\Omega)$ as in \S \ref{Ss:ex} (see Steps 1 and 2). Test \eqref{app-eq} by $u_{0,n}$ to see that
$$
\|u_{0,n}\|_{L^q(\Omega)}^q + \frac 1n \|u_{0,n}\|_{\X}^2 \leq \|\tilde{\rho}_{0,n}\|_{\X^*} \|u_{0,n}\|_{\X},
$$
which yields
$$
\frac 1n \|u_{0,n}\|_{\X} \leq \|\tilde{\rho}_{0,n}\|_{\X^*} \quad \mbox{ and } \quad \frac 1n \|u_{0,n}\|_{L^q(\Omega)}^q \leq \|\tilde{\rho}_{0,n}\|_{\X^*}^2
$$
for $n \in \N$. Since $(\tilde{\rho}_{0,n})$ is bounded in $\X^*$ and $q > 1$, we observe that
$$
\frac 1n u_{0,n} \to 0 \quad \mbox{ weakly in } \X,
$$
and hence, it follows that
$$
|u_{0,n}|^{q-2}u_{0,n} \to \rho_0 \quad \mbox{ weakly in } \X^*.
$$
Testing \eqref{app-eq} by $(-\Delta)^{-\theta} (|u_{0,n}|^{q-2}u_{0,n})$, we deduce that
$$
\left\| |u_{0,n}|^{q-2}u_{0,n} \right\|_{\X^*}^2 + \frac 1n \|u_{0,n}\|_{L^q(\Omega)}^q \leq \|\tilde{\rho}_{0,n}\|_{\X^*} \left\| |u_{0,n}|^{q-2}u_{0,n} \right\|_{\X^*},
$$
whence we obtain
$$
\limsup_{n \to \infty} \left\| |u_{0,n}|^{q-2}u_{0,n} \right\|_{\X^*} \leq \|\rho_0\|_{\X^*}.
$$
Therefore from the uniform convexity of $\|\cdot\|_{\X^*}$ we conclude that
$$
|u_{0,n}|^{q-2}u_{0,n} \to \rho_0 \quad \mbox{ strongly in } \X^*.
$$

\section{Boundedness of weak solutions to \eqref{eq:1.10}, \eqref{eq:1.11}}\label{A:moser}

Let $0 < \theta < 1$ and $1 < q < 2^*_\theta$. We formally test \eqref{eq:1.10} by $|\phi|^{r-2}\phi$ for some $r > 2$, which will be determined later. It then follows from \eqref{EleIne} that
$$
\frac{8}{rr'} \left[ |\phi|^{(r-2)/2}\phi \right]_{H^\theta(\R^d)}^2 \leq \lambda_q \int_\Omega |\phi|^{q+r-2} \, \d x.
$$
Using the Sobolev--Poincar\'e inequality \eqref{SP}, we deduce that
\begin{equation}\label{M-iter}
c \|\phi\|_{L^{2^*_\theta \frac{r}2}(\Omega)}^{r} = c \left\| |\phi|^{(r-2)/2}\phi \right\|_{L^{2^*_\theta}(\Omega)}^2 \leq \lambda_q \|\phi\|_{L^{q+r-2}(\Omega)}^{q+r-2}
\end{equation}
for some constant $c > 0$.

We start with setting $p_0 = 2^*_\theta > q$ and choosing $r_1$ so that $q + r_1 - 2 = p_0$, that is, $r_1 = p_0 - q + 2 > 2$. Then we derive from \eqref{M-iter} with $r = r_1$ that $\phi \in L^{p_1}(\Omega)$ where 
$$
p_1 := 2^*_\theta \, \frac{r_1}2 = 2^*_\theta \, \frac{p_0 - q + 2}2 > p_0.
$$
Let $(p_n)$ be a sequence which is iteratively constructed as in \eqref{pn} below. We claim that $p_n \to \infty$ as $n \to \infty$. Indeed, we let $n \in \N$ (and start with $n = 1$). Suppose that $p_j > p_{j-1}$ for $j = 1,2,\ldots,n$. Then exploiting \eqref{M-iter} with $r = r_{n+1} := p_n - q + 2 > 2$, we infer that $\phi \in L^{p_{n+1}}(\Omega)$ with
\begin{equation}\label{pn}
p_{n+1} := 2^*_\theta \, \frac{r_{n+1}}2 = 2^*_\theta \, \frac{p_n-q+2}2 = \left( 2^*_\theta \, \frac{p_n-q+2}{2p_n} \right) p_n.
\end{equation}
Here we note that
$$
2^*_\theta \, \frac{p_n-q+2}{2p_n} \geq 2^*_\theta \, \frac{p_0-q+2}{2p_0} = \frac{2^*_\theta-q}{2} + 1 > 1 \quad \mbox{ if } \ q > 2
$$
and
$$
2^*_\theta \, \frac{p_n-q+2}{2p_n} = \frac{2^*_\theta}2  \, \frac{p_n-q+2}{p_n} > \frac{2^*_\theta}2 > 1 \quad \mbox{ if } \ 1 < q < 2.
$$
Thus we obtain $p_{n+1} > p_n$. Therefore we conclude that $(p_n)$ is strictly increasing, and moreover, we see that
$$
p_n \geq \left( \frac{2^*_\theta-\max\{q,2\}}{2} + 1 \right)^n p_0 \to \infty
$$
as $n \to \infty$. Hence we can take $n \in \N$ such that
$$
\frac{p_n}{q-1} > \frac{d}{2\theta},
$$
and then, since $\phi \in L^{p_n}(\Omega)$, that is, $|\phi|^{q-2}\phi \in L^{p_n/(q-1)}(\Omega)$, due to~\cite[Proposition 1.4]{ROS14-2}, since $\Omega$ is a bounded $C^{1,1}$ domain of $\R^d$, we can reach the H\"older regularity,
$$
\phi \in C^\beta(\R^d)
$$
for some $\beta \in (0,1)$. Since $\phi = 0$ in $\R^d \setminus \Omega$ and $\Omega$ is bounded, we obtain the boundedness of $\phi$ in $\R^d$. The case where $\theta = 1$ is standard.

\section{Uniqueness of positive solutions to fractional sublinear elliptic equations}\label{A:BO}

Lemma \ref{L:uniq-DP} can be proved through Br\'ezis--Oswald's approach for the restricted fractional Laplacian. Br\'ezis--Oswald's approach (see \cite{BrOs}) allows us to prove uniqueness of positive solutions to general sublinear (nonlocal) elliptic problems of the form,
\begin{alignat}{4}
(-\Delta)^\theta u &= f(x,u) \ &&\text{ in } \ \Omega, \label{BO1}\\
u &\geq 0,\ u\not\equiv 0 \ &&\text{ in } \ \Omega, \label{BO2}\\
u &= 0 \ &&\text{ in } \ \R^d \setminus \Omega,\label{BO3}
\end{alignat}
where $\Omega\subset \R^d$ is a bounded $C^{1,1}$ domain and $f = f(x,u):\Omega\times [0,\infty)\rightarrow \R$. While the original paper \cite{BrOs} is concerned with the local case $\theta = 1$, we extend below the result to the nonlocal case $0 < \theta < 1$. We make the same assumption on $f$ as the original paper \cite{BrOs}:
\begin{description}
 \item[(F1)] For a.e.~$x \in \Omega$, the function $u \mapsto f(x,u)$ is continuous on $[0,\infty)$ and the function $u \mapsto f(x,u)/u$ is strictly decreasing in $(0,\infty)$,
 \item[(F2)] for each $u \in [0,\infty)$, the function $x \mapsto f(x,u)$ belongs to $L^\infty(\Omega)$,
 \item[(F3)] there exists a constant $C > 0$ such that
$$ 
f(x,u) \leq C (|u|+1)
$$
for all $u \geq 0$ and a.e.~$x \in \Omega$.
\end{description}
In particular, \eqref{BO1} with the special choice $f(x,u) = \lambda_q |u|^{q-2}u$ corresponds to \eqref{eq:1.10}.

Now, we have
\begin{theorem}\label{T:BO}
Let $\Omega$ is a bounded $C^{1,1}$ domain of $\R^d$ and assume {\rm (F1)--(F3)}. Then the problem \eqref{BO1}--\eqref{BO2} admits at most one weak solution $u\in \X\cap L^\infty(\Omega)$.
\end{theorem}

We start with the following 
\begin{lemma}\label{L:BO}
Assume {\rm (F1)--(F3)} and let $u,v \in \X \cap L^\infty(\Omega)$ be two weak solutions to the problem \eqref{BO1}--\eqref{BO2}. Then $u,v \in C^\theta(\R^d)$, $u > 0$ and $v > 0$ in $\Omega$, and $u^2/v, v^2/u \in \X$.
\end{lemma}

\begin{proof}
We first note by (F1) and (F3) that
$$
- |f(x,\|u\|_\infty)| \leq f(x,u(x)) \leq C (\|u\|_\infty+1) \quad \text{ for a.e. } x \in \Omega.
$$
Therefore (F1)--(F3) imply that the function $x \mapsto f(x,u(x))$ belongs to $L^\infty(\Omega)$. We deduce from the fractional elliptic regularity result~\cite[Proposition 1.4]{ROS14} that $u \in C^\theta(\R^d)$. We can similarly show that $v \in C^\theta(\R^d)$. Moreover, we see from (F1) and (F2) that there exists a constant $M > 0$ such that 
$$
f(x,u(x)) \geq -M u(x) \quad \text{for a.e. } x \in \Omega.
$$
Therefore it holds that
$$
(-\Delta)^\theta u + Mu \geq 0 \quad \text{in } \Omega,
$$
and the same is true for $v$. Thanks to Hopf's lemma (see~\cite{GrSe}) for the restricted fractional Laplacian, we deduce that $u > 0$ and $v > 0$ in $\Omega$, and that $u/v \in L^\infty(\Omega)$ and $v/u \in L^\infty(\Omega)$. Therefore we have
\begin{align*}
\left\|\frac{u^2}{v}\right\|_{\X}^2 
&= \iint_{\R^d \times \R^d} \frac{|(\frac{u^2}{v})(x) - (\frac{u^2}{v})(y)|^2}{|x-y|^{d+2\theta}} \, \d x\d y\\
&= \iint_{\R^d \times \R^d} \Big| \frac{u(x)}{v(x)}\left( u(x)-u(y) \right) + \frac{u(x)u(y)}{v(x)v(y)} \left(v(y)-v(x)\right)\\
&\quad +\frac{u(y)}{v(y)} \left(u(x)-u(y)\right) \Big|^2 \frac1{|x-y|^{d+2\theta}} \,\d x\d y.
\end{align*}
Therefore we see that
$$
\left\|\frac{u^2}{v}\right\|_{\X}^2\leq C \left( \|u\|_{\X}^2 + \|v\|_{\X}^2 \right) < +\infty.
$$
Consequently, we obtain $u^2/v\in \X$, and moreover, we can similarly show that $v^2/u \in \X$. This completes the proof.
\end{proof}

We next give a proof for Theorem \ref{T:BO}.

\begin{proof}[Proof of Theorem {\rm \ref{T:BO}}]
Let $u,v\in \X\cap L^\infty(\Omega)$ be two weak solutions to \eqref{BO1}--\eqref{BO2}. The strategy is the same as the local case (see~\cite{BrOs}): we formally divide by $u$ the PDE for $u$ (equation \eqref{BO1}), and we divide by $v$ the PDE for $v$ (equation \eqref{BO1} with $u$ replaced by $v$). Then we test the difference of the two equations by $u^2-v^2$, and we conclude the proof by using assumption (F1) and some monotony of $(-\Delta)^\theta$. Those computations are rigorously justified with the aid of Lemma \ref{L:BO}. Indeed, by virtue of (F1), we have 
\begin{equation}\label{eq:BOf}
\int_\Omega \left( \frac{f(x,u(x))}{u(x)} - \frac{f(x,v(x))}{v(x)} \right) \left(u^2(x)-v^2(x)\right) \,\d x \leq 0,
\end{equation}
with equality if and only if $u = v$ in $\Omega$. On the other hand, the left-hand side of \eqref{eq:BOf} is equal to 
\begin{align*}
\MoveEqLeft{
\int_\Omega f(x,u(x))u(x) \,\d x - \int_\Omega f(x,u(x)) \left(\frac{v^2}{u}\right)(x) \,\d x 
}\\
&\quad- \int_\Omega f(x,v(x)) \left(\frac{u^2}{v}\right)(x) \,\d x + \int_\Omega f(x,v(x))v(x) \,\d x\\
&= \left\langle (-\Delta)^\theta u, u \right\rangle_{\X} - \left\langle (-\Delta)^\theta u, \frac{v^2}{u} \right\rangle_{\X}\\
&\quad - \left\langle (-\Delta)^\theta v, \frac{u^2}{v} \right\rangle_{\X} + \left\langle (-\Delta)^\theta v, v \right\rangle_{\X}.
\end{align*}
The first two terms in the right-hand side can be rewritten as
$$
\langle (-\Delta )^\theta u, u \rangle_{\X} = \iint_{\R^d \times \R^d} \frac{u^2(x)+u^2(y)-2u(x)u(y)}{|x-y|^{d+2\theta}} \,\d x\d y,
$$
and
\begin{align*}
\left\langle (-\Delta )^\theta u, \frac{v^2}{u} \right\rangle_{\X}
&= \iint_{\R^d \times \R^d} \bigg[ v^2(x)+v^2(y) \\ 
&\quad - \Big( \frac{u(y)}{u(x)} v^2(x) + \frac{u(x)}{u(y)} v^2(y) \Big) \bigg] \frac1{|x-y|^{d+2\theta}}\, \d x\d y
\end{align*}
and the same equations hold when $u$ and $v$ are exchanged. Therefore we can eventually derive that the left-hand side of \eqref{eq:BOf} is equal to
\begin{align*}
\MoveEqLeft{
\iint_{\R^d \times \R^d} \bigg[ \frac{v(y)}{v(x)} \Big( u(x)-\frac{v(x)}{v(y)}u(y) \Big)^2 }\\
&\quad + \frac{u(y)}{u(x)} \Big( v(x)-\frac{u(x)}{u(y)}v(y) \Big)^2 \bigg] \frac1{|x-y|^{d+2\theta}}\, \d x\d y \geq 0.
\end{align*}
As a consequence, \eqref{eq:BOf} is an equality, and therefore, $u = v$ in $\Omega$. This completes the proof.
\end{proof}

\bibliographystyle{plain}
\bibliography{bibliography}

\end{document}